\documentclass[11pt]{article}
\usepackage[left=1in,top=1in,right=1in,bottom=1in,head=0in,nofoot]{geometry}
\usepackage{amsmath,amssymb,amsthm}
\usepackage{natbib}
\usepackage{graphicx}
\usepackage{times}
\graphicspath{{./figs/}}
\usepackage{enumerate}
\usepackage{caption}
\usepackage{subcaption}
\usepackage[toc]{appendix}
\usepackage{bbold}
\usepackage{authblk}
\usepackage{chngcntr}
\usepackage{apptools}
\AtAppendix{\counterwithin{lemma}{section}}
\AtAppendix{\counterwithin{theorem}{section}}
\AtAppendix{\counterwithin{proposition}{section}}
\AtAppendix{\counterwithin{corollary}{section}}
\AtAppendix{\counterwithin{equation}{section}}

\usepackage[usenames,dvipsnames]{xcolor}

\renewcommand*{\eqref}[1]{(\ref{#1})}

\newcommand{\lb}{\left(}
\newcommand{\rb}{\right)}
\newcommand{\E}{\mathbb{E}}
\renewcommand{\P}{\mathbb{P}}
\newcommand{\R}{\mathbb{R}}
\renewcommand{\S}{\mathcal{S}}

\newcommand{\T}{\mathcal{T}}
\newcommand{\F}{\mathcal{F}}
\newcommand{\eps}{\epsilon}

\newcommand{\td}[1]{\tilde{#1}}
\newcommand{\Var}{\mathrm{Var}}
\newcommand{\Cov}{\mathrm{Cov}}
\newcommand{\tr}{\mathrm{tr}}
\newcommand{\op}{\mathrm{op}}
\newcommand{\diag}{\mathrm{diag}}

\newcommand{\tauunadj}{\hat{\tau}_{\mathrm{unadj}}}
\newcommand{\exps}[1]{\exp\left\{#1\right\}}
\newcommand{\htau}{\hat{\tau}_{\mathrm{adj}}}
\newcommand{\taue}{\hat{\tau}_{e}}
\newcommand{\hdtau}{\hat{\tau}_{\mathrm{adj}}^{\mathrm{de}}}

\newcommand{\HC}{\mathrm{HC}}

\newcommand*{\tran }{^{\mkern-1.5mu\mathsf{T}}}
\newcommand{\cure}{\mathcal{E}}

\newcommand{\lihua}[1]{{\color{black}#1}}
\newcommand{\id}{\mathbb{I}}
\newcommand{\projone}{\mathbb{V}}
\newcommand{\one}{\mathbb{1}}
\newcommand{\oneone}{\one}
\newcommand{\onezero}{\one}
\newcommand{\onet}{\one}
\newcommand{\OLS}{OLS}
\newcommand{\ATE}{ATE}
\newcommand{\iid}{i.i.d.}

\DeclareMathOperator*{\argmin}{arg\,min}

\DeclareMathOperator*{\rank}{rank}
\newtheorem{theorem}{Theorem}
\newtheorem{lemma}{Lemma}
\newtheorem{proposition}{Proposition}
\newtheorem{corollary}{Corollary}

\newtheorem{assumption}{Assumption}

\allowdisplaybreaks

\begin{document}

\title{Regression adjustment in completely randomized experiments with a diverging number of covariates}
 
\author[1]{Lihua Lei}
\affil[1]{Departments of Statistics, Stanford University \thanks{lihualei@stanford.edu}}

\author[2]{Peng Ding}
\affil[2]{Departments of Statistics, University of California, Berkeley \thanks{pengdingpku@berkeley.edu}}

\maketitle

\begin{abstract}
Randomized experiments have become important tools in empirical research. In a completely randomized treatment-control experiment, the simple difference in means of the outcome is unbiased for the average treatment effect, and covariate adjustment can further improve the efficiency without assuming a correctly specified outcome model. In modern applications, experimenters often have access to many covariates, motivating the need for a theory of covariate adjustment under the asymptotic regime with a diverging number of covariates. We study the asymptotic properties of covariate adjustment under the potential outcomes model and propose a bias-corrected estimator that is consistent and asymptotically normal under weaker conditions. Our theory is purely randomization-based without imposing any parametric outcome model assumptions. To prove the theoretical results, we develop novel vector and matrix concentration inequalities for sampling without replacement.
\end{abstract}

~

\section{Introduction}\label{sec:intro}

\subsection{Randomized experiment and Neyman's randomization model}

Randomized experiments have been powerful tools in agricultural, industrial, biomedical, and social sciences \citep[e.g.,][]{fisher37, kempthorne52, box2005statistics, rosenberger2015randomization, duflo2007using, gerber2012field, imbens15}. In a treatment-control experiment,  let $Y_{i}(1)$ and $ Y_{i}(0)$ be the potential outcomes if unit $i\in\{1,\ldots, n \}$ receives the treatment and control, respectively \citep{neyman23}. Define the parameter of interest as the average treatment effect (\ATE) 
$
\tau =  n^{-1}\sum_{i=1}^{n}\tau_{i},
$
where $ \tau_{i} = Y_{i}(1) - Y_{i}(0)$ is the individual treatment effect for unit $i.$ In a completely randomized experiment, the experimenter randomly assigns $n_1$ units to the treatment group and $n_0$ units to the control group, with $n=n_1+n_0$. Let $T_{i}$ denote the assignment of the $i$-th unit where $T_{i} = 1$ corresponds to the treatment and $T_i=0$ corresponds to the control. For unit $i$, only $Y_{i}^{\mathrm{obs}}  = Y_{i}(T_{i})$ is observed while the other potential outcome $Y_{i}(1 - T_{i})$ is missing. 
  
\citet{neyman23} assumed that all potential outcomes are fixed and the randomness comes solely from the treatment indicators. This finite-population perspective has a long history for analyzing randomized experiments \citep[e.g.][]{kempthorne52, imbens15, dasgupta2015causal, Middleton2015cluster, mukerjee2018using, fogarty2018regression, li16, li2020rerandomization}. It clarifies the role of the study design in the analysis without postulating a hypothetical outcome generating process. By contrast, the super-population perspective \citep[e.g.][]{tsiatis2008covariate, berk2013covariance, negi2019robust} assumes that the potential outcomes and other individual characteristics are independent and identically distributed (\iid) ~draws from some distribution. These two perspectives are both popular in the literature, but they are different in the source of randomness: the finite-population perspective conditions on the potential outcomes and quantifies the uncertainty from the treatment assignment, and the super-population averages over the potential outcomes and quantifies the uncertainty from the independent sampling process.  We focus on the former throughout the paper.

Let $\one$ denote the vector with all entries $1$, $\id$ the identity matrix, and $\projone= \id  -  (\one\tran \one )^{-1} \one \one\tran$  the projection matrix orthogonal to $\one$, with appropriate dimensions depending on the context.   Let $\|\cdot\|_{q}$ be the vector $q$-norm, i.e. $\|\alpha\|_{q} = (\sum_{i=1}^{n}|\alpha_{i}|^{q})^{ 1 / q }$ and $\|\alpha\|_{\infty} = \max_{1\leq i \leq n}|\alpha_{i}|$. Let $\|\cdot\|_{\op}$ denote the matrix operator norm. Let $N(0, 1)$ denote the standard normal distribution, and $t(\nu)$ denote standard $t$ distribution with degrees of freedom $\nu$.

\subsection{Average treatment effect estimates with and without regression adjustment}

Let $\T_{t} = \{i: T_{i} = t\}$ be the indices and $n_{t} = |\T_{t}|$ be the fixed sample size for the treatment arm $t\in \{0, 1\}$. Consider a completely randomized experiment in which $\mathcal{T}_1$ is a random size-$n_1$ subset of $\{1,\ldots, n\}$ uniformly over all $n!/( n_1! n_0! )$ subsets.  The simple difference in means 
$$
\tauunadj =  n_1^{-1} \sum_{i\in \T_{1}}Y_{i}^{\mathrm{obs}} -  n_0^{-1} \sum_{i\in \T_{0}}Y_{i}^{\mathrm{obs}} 
=  n_1^{-1} \sum_{i\in \T_{1}}Y_{i}(1) -  n_0^{-1} \sum_{i\in \T_{0}}Y_{i}(0)
$$
is unbiased for $\tau$ with variance $S_{1}^{2}/n_{1} + S_{0}^{2} / n_{0} - S_{\tau}^{2}/n$ \citep{neyman23}, where $S_{1}^{2}, S_{0}^{2}$ and $ S_{\tau}^{2}$ are the finite-population variances of the $Y_{i}(1)$'s, $Y_{i}(0)$'s and $\tau_{i}$'s, respectively.

The experimenter usually collects pre-treatment covariates. If they are predictive of the potential outcomes,   incorporating them in the analysis can improve the estimation efficiency. Suppose unit $i$ has a $p$-dimensional vector of pre-treatment covariates $x_{i}\in \R^{p}$. Early works on the analysis of covariance assumed a constant treatment effect \citep{fisher37, kempthorne52, hinkelmann07}, under which a commonly-used  estimate is the coefficient of the treatment indicator of the ordinary least squares (\OLS) fit of the $Y_{i}^{\mathrm{obs}}$'s on $T_i$'s and $x_i$'s. \citet{freedman08a} criticized this  approach, showing that it can be even less efficient than $\tauunadj $ in the presence of treatment effect heterogeneity, and the estimated standard error based on the \OLS ~can be inconsistent for the true standard error under the randomization model.

\citet{lin13} proposes a simple solution. We center the covariates at $n^{-1}\sum_{i=1}^{n}x_{i} = 0$ without loss of generality. His estimator for $\tau$ is the coefficient of the treatment indicator in the \OLS ~fit of the $Y_{i}^{\mathrm{obs}}$'s on $T_i$'s, $x_i$'s and the  interaction terms $T_ix_i$'s. His estimator is consistent, asymptotically normal, and more efficient than $\tauunadj$.  He further shows that the Eicker--Huber--White standard error is consistent for the true standard error. His results hold under the finite-population randomization model, without assuming that the linear model is correct.

We use an alternative formulation of the regression adjustment and consider the following family of covariate-adjusted  estimators:
\begin{equation}\label{eq:family_regadj}
\hat{\tau}(\gamma_{1}, \gamma_{0}) =  n_1^{-1} \sum_{i\in \T_{1}}(Y_{i}^{\mathrm{obs}} - x_{i}\tran \gamma_{1}) -  n_0^{-1} \sum_{i\in \T_{0}}(Y_{i}^{\mathrm{obs}} - x_{i}\tran \gamma_{0}).
\end{equation}
Because $n_{t}^{-1}\sum_{i\in \T_{t}} x_{i}\tran \gamma_t $ has expectation zero over all possible randomizations, the estimator in \eqref{eq:family_regadj} is unbiased for any fixed coefficient vectors $\gamma_t \in \R^{p}$ $(t=0,1)$. It is the difference-in-means estimator with potential outcomes replaced by $\{ Y_{i}(1) - x_{i}\tran \gamma_{1}, Y_{i}(0) - x_{i}\tran \gamma_{0} \} _{i=1}^n $.

Let $Y(t) = (Y_{1}(t), \ldots, Y_{n}(t))\tran \in \R^{n}$ denote the vector of potential outcomes under treatment $t$ and $X = (x_{1} , \ldots, x_{n} )\tran $ denote the covariate matrix. Without loss of generality, we assume
$
\one\tran X = 0$ and $\rank(X) = p,
$
i.e., the covariate matrix has centered columns and full column rank. Otherwise, we transform $X$ to $\projone X$ and remove the redundant columns to ensure the full column rank condition. This operation does not affect inferential validity because $X$ is fixed, or, equivalently, our inference conditions on $X$.

Let $\beta_{t}$ be the population \OLS ~coefficient of regressing $Y(t)$ on $( \one, X)$: 
\begin{eqnarray}
(\mu_{t}, \beta_{t}) = \argmin_{\mu\in \R, \beta\in \R^{p}}\|Y(t) - \mu\one - X\beta\|_{2}^{2}  
= \lb  n^{-1} \sum_{i=1}^{n}Y_{i}(t) ,     (X\tran X)^{-1}X\tran Y(t)  \rb  , \label{eq:mut-ols}
\end{eqnarray}
which holds because $X$ is orthogonal to $\one$. \citet[][Example 9]{li2017general} show that the \OLS ~coefficients $(\beta_{1}, \beta_{0})$ in \eqref{eq:mut-ols} minimize the variance of the estimator in \eqref{eq:family_regadj}. We emphasize that $\beta_{1}$ and $\beta_{0}$ are both unobserved population quantities.

The classical analysis of covariance chooses $\gamma_{1} = \gamma_{0} = \hat{\beta}$, the coefficient of the covariates in the \OLS ~fit of the $Y_{i}^{\mathrm{obs}}$'s on $T_i$'s and $x_i$'s with an intercept. This strategy implicitly assumes away treatment effect heterogeneity, and can lead to inferior properties when $\beta_1 \neq \beta_0$ \citep{freedman08a}. \cite{lin13} chooses $\gamma_{1} = \hat{\beta}_1$ and $\gamma_0 = \hat{\beta}_0$, the coefficients of the covariates in the \OLS ~fit of $Y_{i}^{\mathrm{obs}}$'s on $x_i$'s with an intercept, in the treatment and control groups, respectively. Numerically, this is identical to the estimator obtained from the regression with interactions discussed before. Throughout the rest of the paper, we refer to it as the regression-adjusted estimator.

\subsection{Our contributions}\label{subsec:contribution}

In practice, experiments often have many covariates. Therefore, it is important to approximate the sampling distribution with $p$ growing with the sample size $n$ at a certain rate. Under the finite-population randomization model, \citet{bloniarz16} discussed a high dimensional regime with possibly larger $p$ than $n$ but assumed that the potential outcomes could be well approximated by a sparse linear combination of the covariates, under the regime with the number of non-zero coefficients being much smaller than $n^{1/2} / \log p$. Under a super-population framework, \citet{wager16} discussed covariate adjustment using the \OLS ~and some other machine learning techniques.

We study the regression-adjusted estimator under the finite-population perspective in the regime where $p < n$ but $p$ grows with $n$ at a certain rate. We argue that this type of large-$n$-moderate-$p$ asymptotics is more important than the large-$n$-fixed-$p$ asymptotics to analyze completely randomized experiments when $p$ is not a negligible number compared to $n$. For instance, the study on pulmonary artery catheter in  \citet{bloniarz16} has 1013 subjects with 59 covariates. In this case, $p$ is approximately $n^{0.6}$ and thus the inferential guarantees based on fix-$p$ asymptotics are questionable.

We focus on this estimator because it is widely used in practice thanks to its simplicity, and it does not require any tuning parameter, unlike other high dimensional or machine learning methods. As in the classic linear regression, the asymptotic properties depend crucially on the maximum leverage score 
$$
\kappa = \max_{1\leq i\leq n} H_{ii},
$$
where the $i$-th leverage score $H_{ii}$ is $i$-th diagonal entry of the hat matrix $H = X(X\tran X)^{-1}X\tran $.  Under the regime $\kappa \log p\rightarrow 0$, we prove the consistency of the regression-adjusted estimator under mild moment conditions on the population \OLS ~residuals. In the favorable case where all leverage scores are close to their average $p / n$, the consistency holds if $p = o(n / \log n)$.

In addition, we prove that the regression-adjusted estimator is asymptotically normal under $\kappa p\rightarrow 0$ and extra mild conditions, with the same variance formula as in the fixed-$p$ regime. Furthermore, we propose a debiased estimator, which is asymptotically normal under an even weaker assumption $\kappa^{2} p \log p\rightarrow 0$, with the same variance as before. Therefore, this new estimator reduces the asymptotic bias without inflating the asymptotic variance. In the favorable case where all leverage scores are close to their average $p / n$, the regression-adjusted estimator is asymptotically normal when $p = o(n^{1/2})$, but the debiased estimator is asymptotically normal when $p = o\{ n^{2/3} / (\log n)^{1/3}\} $. The regression-adjusted estimator may also be asymptotically normal in the latter regime, but it requires an extra condition; see Theorem \ref{thm:asym_normality_htau}. In our simulation, the debiased estimator indeed yields better finite-sample inferences.

For statistical inference, we propose several asymptotically conservative variance estimators, which yield valid asymptotic Wald-type confidence intervals for the \ATE. We prove the results under the same conditions as required for the asymptotic normality. To prove these results, we also make some technical contributions by proving novel vector and matrix concentration inequalities for sampling without replacement.

\section{Regression Adjustment}\label{sec:estimators}

\subsection{Point estimators}\label{sec::pointestimators}

We reformulate the regression-adjusted estimator. The \ATE  ~$\tau$ is the difference between the two intercepts of the population \OLS ~coefficients in \eqref{eq:mut-ols}:
$
\tau =   n^{-1}  \sum_{i=1}^n Y_i(1) -  n^{-1}  \sum_{i=1}^n Y_i(0)  =   \mu_{1} - \mu_{0}.
$ 
Therefore, we focus on estimating $\mu_1$ and $\mu_0.$  Let $X_{t}\in \R^{n_{t}\times p}$ denote the sub-matrix formed by the rows of $X$, and $Y_{t}^{\mathrm{obs}}\in \R^{n_{t}}$ the subvector of $Y^{\mathrm{obs}}=(Y_1^{\mathrm{obs}},\ldots, Y_n^{\mathrm{obs}} )\tran $, with indices in $\T_{t}$ $(t=0,1)$.  The regression-adjusted estimator follows two steps. First, for $t\in \{0, 1\}$, we regress $Y_{t}^{\mathrm{obs}}$ on $X_{t}$ with an intercept, and obtain the fitted intercept $\hat{\mu}_{t}\in \R$ and coefficient of the covariate $\hat{\beta}_{t}\in \R^{p}$. Second, we estimate $\tau$ by
\begin{equation}\label{eq:estimator}
\htau = \hat{\mu}_{1} - \hat{\mu}_{0}.
\end{equation}
In general, $\htau $ is biased in finite samples. Correcting the bias gives stronger theoretical guarantees as our later asymptotic analysis confirms. Here we propose a bias-corrected estimator. Define the potential residuals based on the population \OLS ~as  
\begin{equation}
  \label{eq:et}
  e(t) = Y(t) - \mu_{t} - X\beta_{t},\quad (t=0,1) . 
\end{equation}
The property of the \OLS ~guarantees that $e(t)$ is orthogonal to $\one$ and $X$: 
\begin{equation}
  \label{eq:error_key_property}
  \one\tran e(t) = 0,\quad 
  X\tran e(t) = 0 ,\quad (t=0,1). 
\end{equation}
Let $\hat{e}\in \R^{n}$ be the vector residuals from the sample \OLS, where $  \hat{e}_{i}  =  Y_{i}^{\mathrm{obs}} - \hat{\mu}_{1} - x_{i}\tran \hat{\beta}_{1}$ for the treated units and $ \hat{e}_{i} =   Y_{i}^{\mathrm{obs}} - \hat{\mu}_{0} - x_{i}\tran \hat{\beta}_{0}$ for the control units. 
For any vector $\alpha\in \R^{n}$, let $\alpha_{t}$ denote the subvector of $\alpha$ with indices in $\T_{t}$, e.g., $Y_{t}(1)$ and $ e_{t}(1)$ are the subvectors of $Y(1)$ and $e(1)$ corresponding to the units in treatment arm $t$, respectively. 
Let
$$
H = X(X\tran X)^{-1}X\tran,\quad H_t= X_{t}(X_{t}\tran X_{t})^{-1}X_{t}\tran
$$
be the hat matrices of $X$ and $X_{t}$, respectively. Let $H_{ii}$ be the $i$-th diagonal element of $H$, also termed as the leverage score, and let $H_{t,ii}$ be the diagonal element of $H_t $ corresponding to unit $i$. From the higher order asymptotic expansion, the bias of $\htau$ depends on 
\begin{equation}\label{eq:hatDeltat}
\Delta_{t} = n^{-1}\sum_{i=1}^{n}  e_{i}(t)  H_{ii}, \quad
\Delta = \max\{ |\Delta_1|, |\Delta_0|   \}.
\end{equation}
With the empirical analogs $\hat{\Delta}_{t} = n_{t}^{-1}\sum_{i\in \T_{t}} \hat{e}_{i}  H_{ii}$, we introduce the following debiased estimator: 
\begin{equation}\label{eq:estimator_debiased}
\hdtau = \htau - \lb \frac{n_{1}}{n_{0}}\hat{\Delta}_{0} - \frac{n_{0}}{n_{1}}\hat{\Delta}_{1}\rb . 
\end{equation}
When $p = 1$, \eqref{eq:estimator_debiased} reduces to the bias formula in \citet[][Section 6 point (iv)]{lin13}. Thus \eqref{eq:estimator_debiased} is an extension to the multivariate case. With some algebraic manipulations, we can show that $\hdtau$ is a finite-population analog of \cite{tan2014second}'s bias-corrected regression estimator in the context of survey sampling with a fixed $p$.

\subsection{Variance estimators}
With a fixed $p$, \cite{lin13} proved that $n^{1/2}(\htau - \tau)$ is asymptotically normal with variance
\begin{eqnarray}
\sigma_{n}^{2}  &=&  n_1^{-1} \sum_{i=1}^{n}e_i^2(1) +  n_0^{-1} \sum_{i=1}^{n}e_i^2(0) -  n^{-1} \sum_{i=1}^{n}\{e_{i}(1)- e_{i}(0)\}^{2}. \label{eq:sigman}
\end{eqnarray}
Formula \eqref{eq:sigman} motivates conservative variance estimators since the third term in \eqref{eq:sigman} has no consistent estimator without further assumptions on $e(1)$ and $e(0)$. Ignoring it and estimating the first two terms in \eqref{eq:sigman} by their sample analogs, we have the following variance estimator: 
\begin{equation}\label{eq:sigman2hat_simple}
\hat{\sigma}^{2} = \frac{n}{n_{1}(n_{1} - 1)}\sum_{i\in \T_{1}} \hat{e}_i^2 + \frac{n}{n_{0}(n_{0} - 1)}\sum_{i\in \T_{0}} \hat{e}_i^2.
\end{equation}
Although \eqref{eq:sigman2hat_simple} appears to be conservative due to the neglect of the third term in \eqref{eq:sigman}, we find in numerical experiments that it typically underestimates $\sigma_{n}^{2}$ if the number of covariates is large. The classic linear regression literature \citep[e.g.][]{mackinnon13} suggests rescaling the residual as $\td{e}_{i} = \zeta_{i}\hat{e}_{i}$, where $\zeta_{i} = 1$ for HC0, $\zeta_{i} = \{ (n_t-1)/(n_t-p) \} ^{1/2}$ for HC1, $\zeta_{i} = 1 / (1 - H_{t, ii})^{1/2}$ for HC2 and $\zeta_{i} = 1 / (1 - H_{t, ii})$ for HC3, for $i\in \T_{t}$. HC0 corresponds to the estimator \eqref{eq:sigman2hat_simple} without corrections. Previous literature has shown that the above corrections, especially HC3, are effective in improving the finite sample performance of variance estimator in linear regression under independent super-population sampling. More interestingly, it is also beneficial to use these rescaled residuals in the context of a completely randomized experiment, motivating 
\begin{equation}
  \hat{\sigma}_{\HC j}^{2} = \frac{n}{n_{1}(n_{1} - 1)}\sum_{i\in \T_{1}}   \td{e}_{i, j}^{2} + \frac{n}{n_{0}(n_{0} - 1)}\sum_{i\in \T_{0}} \td{e}_{i, j}^{2}
\end{equation}
with  residual $\td{e}_{i, j}$  corresponding to HC$j$ for $j =0,  1, 2, 3$. Based on the normal approximations, we can construct Wald-type confidence intervals for $\tau$ based on point estimators $\htau$ and $\hdtau$ with estimated standard errors $\hat{\sigma}_{\HC j}/n^{1/2}$.

\section{Main Results}\label{sec:main}

\subsection{Regularity conditions}\label{subsec::conditions}

We embed the finite-population quantities $\{x_i, Y_i(1), Y_i(0)\}_{i=1}^n$ into a sequence, and impose regularity conditions on this sequence. The first condition is on the sample sizes. 

\begin{assumption}\label{assumption::A1}
$n / n_{1}  = O(1)$ and $n / n_{0} = O(1)$.
\end{assumption}

Assumption \ref{assumption::A1} holds automatically if treatment and control groups have fixed proportions, e.g., $n_1/n=n_0/n=1/2$ for balanced experiments. It is not essential and can be removed at the cost of complicating the statements. 

The second condition is on $\kappa = \max_{1\leq i \leq n}H_{ii}$, the maximum leverage score, which also plays a crucial role in the theory of classic linear models \citep[e.g.][]{huber73, mammen89}. 

\begin{assumption}\label{assumption::A2}
$\kappa \log p = o(1)$.
\end{assumption}

The maximum leverage score satisfies  
\begin{equation}
  \label{eq:Hiibound}
  p / n = \tr(H)/ n\le \kappa \le \|H\|_{\op} = 1  \Longrightarrow \kappa \in [p / n, 1] .
\end{equation}
Assumption \ref{assumption::A2} permits influential observations as long as $\kappa = o(1 / \log p)$.
In the favorable case with $\kappa = O(p / n)$, it reduces to $p\log p / n\rightarrow 0$, which permits $p$ to grow as fast as $n^{\gamma}$ for any $0\leq \gamma < 1$. Moreover, it implies
\begin{equation}
  \label{eq:pn}
p/n \le \kappa = o\lb 1 / \log p \rb = o(1) \Longrightarrow p = o(n).
\end{equation}

Assumptions \ref{assumption::A1} and \ref{assumption::A2} are useful for establishing consistency. The following two extra conditions are useful for the variance estimation and asymptotic normality. The third condition is on the correlation between the potential residuals from the population \OLS ~in \eqref{eq:et}. 

\begin{assumption}\label{assumption::A3}
There exists a constant $\eta > 0$ independent of $n$ such that 
$$
\rho_{e} =     e(1)\tran e(0)    / \{  \| e(1) \|_2 \| e(0) \|_2 \}    > -1 + \eta .
$$
\end{assumption}

Assumption \ref{assumption::A3} is mild because it is unlikely to have a perfectly negative sample correlation between the treatment and control potential residuals in practice.

The fourth condition is on the following two measures of the potential residuals: 
\[
\cure_{2} = n^{-1} \max \left\{\|e(0)\|_{2}^{2} , \|e(1)\|_{2}^{2} \right\}, \quad \cure_{\infty} = \max \left\{\|e(0)\|_{\infty}, \|e(1)\|_{\infty}\right\}.
\]

\begin{assumption}\label{assumption::A4}
$\cure_{\infty}^{2} / (n\cure_{2}) = o(1).$
\end{assumption}

Assumption \ref{assumption::A4} is a Lindeberg--Feller-type condition requiring that no single residual dominates the others. A similar form appeared in \citet{hajek1960limiting}'s finite-population central limit theorem. Previous works require more stringent assumptions on the fourth moment \citep{lin13, bloniarz16} while Assumption \ref{assumption::A4} allows for heavy-tailed outcomes with $\cure_2$ growing with $n$.

These assumptions are weaker than those in previous works \citep[e.g.][]{lin13, bloniarz16, li2020rerandomization}. Supplementary Material II provides further discussions.

\subsection{Asymptotic expansions and consistency}

We start with the asymptotic expansions of $\htau$ and $\hdtau$.

\begin{theorem}\label{cor:expansion}
  Under Assumptions \ref{assumption::A1} and \ref{assumption::A2},
  \begin{eqnarray}
  &    \htau - \tau &= \taue+ O_\P\left[ \Delta + \left\{ \cure_{2}(\kappa^{2} p\log p + \kappa) / n \right\}^{1/2}\right] \label{eq:expansion_htau} , \\
   &   \hdtau - \tau &= \taue + O_\P\left[ \left\{ \cure_{2}(\kappa^{2} p \log p + \kappa) / n \right\}^{1/2}\right], 
   \label{eq:expansion_hdtau}
  \end{eqnarray}
  where $ \taue = \oneone \tran e_{1}(1) / n_{1} -  \onezero \tran e_{0}(0) / n_{0}$ is the difference in means of the potential residuals. 
\end{theorem}

In \eqref{eq:expansion_htau} and \eqref{eq:expansion_hdtau}, $\taue $ has mean $0$ and variance $\sigma_n^2/n$ \citep{neyman23}, which is $O_\P\{ (\sigma_n^2 / n)^{1/2} \}$ by Chebyshev's inequality. Based on the definitions in \eqref{eq:hatDeltat}, we further have
 \begin{equation}\label{eq:Deltat}
 \Delta^{2} = \max_{t=0,1}\Delta_{t}^{2}
 \le   
 \left( n^{-1} \sum_{i=1}^{n}H_{ii} \right) \left\{  \max_{t=0,1} n^{-1} \sum_{i=1}^{n}e_{i}^2(t)  H_{ii}    \right\} 
 \le
 \cure_{2} \kappa p/n
 \end{equation}
 by the Cauchy--Schwarz inequality and the facts that $\sum_{i=1}^{n}H_{ii}=p$ and $H_{ii} \leq \kappa $. 
Since $\kappa \le 1$ and $\sigma_n^2 = O(\cure_2)$, Theorem \ref{cor:expansion} implies   
$$
   \htau - \tau  =  O_\P\left[\left\{ \cure_{2}(\kappa p + 1) / n \right\}^{1/2}\right] , \quad 
     \hdtau - \tau  =   O_\P\left[\left\{ \cure_{2}(\kappa^{2} p \log p + 1) / n \right\}^{1/2}\right],
     $$ 
which  further imply the following consistency results by requiring the right-hand sides  to vanish.

\begin{theorem}\label{thm:consistency}
   Under Assumptions \ref{assumption::A1} and \ref{assumption::A2},
$\htau$ is consistent if 
 $
 \cure_{2} = o\{ n / (\kappa p + 1 ) \},
$
and $\hdtau$ is consistent if
$
\cure_{2} = o\{  n / ( \kappa^{2} p \log p + 1) \}.
$
\end{theorem}

Theorem \ref{thm:consistency} implies the following consistency results for a fixed or diverging $p$. 

\begin{corollary}\label{corollary::consistency}
Under Assumptions \ref{assumption::A1} and \ref{assumption::A2}, both $\htau$ and $\hdtau$ are consistent if either (i) $p$ is fixed and  $\cure_{2} = o(n)$ or (ii) $p$ is diverging with $n$ and $\cure_2 = O(n / p)$. 
\end{corollary}

We can prove Corollary \ref{corollary::consistency} by verifying the more stringent condition for the consistency of $\htau$  in Theorem \ref{thm:consistency}. With a fixed $p$ and $\cure_{2} = o(n)$, we have
 $ \cure_{2} / \{n / (\kappa p + 1 )\} \leq  \cure_{2}/n \times (p+1)  \rightarrow 0$ because $\kappa \le 1$; with a diverging $p$ and $\cure_2 = O(n / p)$, we have
$
  \cure_{2}  /\{   n / (\kappa p + 1 ) \}  
=   \cure_{2} / (n /p ) \times  (\kappa + 1/p) \rightarrow 0
$ 
 because Assumption \ref{assumption::A2} implies $\kappa = o(1 / \log p)$.

\subsection{Asymptotic normality and variance estimation}
In \eqref{eq:expansion_htau} and \eqref{eq:expansion_hdtau}, $\taue $ is asymptotically normal with mean $0$ and variance $\sigma_n^2/n$. Therefore, the asymptotic normalities of $\htau$ and $\hdtau$ hold if the the remainders vanish after being multiplied by $n^{1/2} / \sigma_{n}$. We first present the result for $\htau$.

\begin{theorem}\label{thm:asym_normality_htau}
  Under Assumptions \ref{assumption::A1}--\ref{assumption::A4}, $ n^{1/2}(\htau - \tau) / \sigma_{n} \rightsquigarrow N(0, 1)$ if
$\kappa^{2} p \log p = o(1)$ and $ n\Delta^{2} = o(\cure_{2}).$  
\end{theorem}

The term $n\Delta^2$ is the squared bias of $n^{1/2}\htau$. If it vanishes, $\htau$ has the same asymptotic normality as $\taue $. We can use Theorem \ref{thm:asym_normality_htau} to find more interpretable sufficient conditions to replace $ n\Delta^{2} = o(\cure_{2})$. 
An upper bound on  $\Delta$ is in \eqref{eq:Deltat}. So an obvious sufficient condition is $\kappa p = o(1)$, which also implies $\kappa^{2} p \log p = (\kappa p) (\kappa \log p ) =   o(1)$ under Assumption \ref{assumption::A2}. On the other hand, because  $e(t)$ has mean zero, we have $
 \Delta_{t} =  n^{-1}\sum_{i=1}^{n}  e_{i}(t)   \lb H_{ii} -  p / n \rb 
 $, which helps to derive another upper bound on $\Delta$. Define the maximum absolute deviation of the $H_{ii}$'s from their average as 
 $$
 \kappa_0  = \max_{1\leq i\leq n} |H_{ii} - p / n| ,
 $$
 and then we can use the Cauchy--Schwarz inequality to obtain 
  \[
  \Delta = \max_{t=0,1}|\Delta_{t}|\le \kappa_0 
 \max_{t=0,1}  n^{-1} \sum_{i=1}^{n}|e_{i}(t)| 
 \le \kappa_0  \cure_{2}^{1/2}.
 \]
So another sufficient condition is  $\kappa_0 = o(n^{-1/2})$. This condition implies that $\kappa \leq \kappa_0 + p/n =o(n^{-1/2}) + p/n$, which, coupled with $p = o\{n^{2/3} / (\log n)^{1/3}\}$, implies   
$
\kappa^{2} p \log p = o(1). 
$
The following corollary summarizes the results from the above discussion. 

\begin{corollary}\label{cor:asym_normality_htau}
  Under Assumptions \ref{assumption::A1}--\ref{assumption::A4}, $ n^{1/2}(\htau - \tau) / \sigma_{n} \rightsquigarrow N(0, 1)$ if either (i) $\kappa p = o(1)$ or (ii) 
$p = o\{ n^{2/3} / (\log n)^{1/3}\}$ and   $\kappa_0 = o(n^{-1/2})$. 
\end{corollary}

Consider  the favorable case with $\kappa = O(p / n)$. Condition (i) reduces to $p = o(n^{1/2})$, so Corollary \ref{cor:asym_normality_htau} extends \cite{lin13}'s result to $p = o(n^{1/2})$ without any further assumptions.
Condition (ii) states that when all the leverage scores are within an $o(n^{-1/2})$ neighborhood of their average $p / n$, the requirement on $p$ can be relaxed to $o\{ n^{2/3} / (\log n)^{1/3}\}$.  Supplementary Material II shows that when the $x_i$s are realizations of multivariate normal vectors as assumed by \cite{wager16}, the leverage score conditions hold with high probability.

Although we can relax the constraint on the dimension $p$ under condition (ii), it is not ideal to impose an extra condition on the leverage scores. When $p > n^{1/2}$, the leverage score condition is more stringent than that in the favorable case. By contrast, the debiased estimator is asymptotically normal without any additional condition.

\begin{theorem}\label{thm:asym_normality_hdtau}
  Under Assumptions \ref{assumption::A1}--\ref{assumption::A4}, $ n^{1/2}(\hdtau - \tau) / \sigma_{n} \rightsquigarrow N(0, 1)$ if $\kappa^{2} p \log p = o(1)$. 
\end{theorem}

In the favorable case with $\kappa = O(p / n)$, the condition in Theorem \ref{thm:asym_normality_hdtau} reduces to $p^{3}\log p / n^2 = o(1)$, which permits $p$ to grow as fast as $o\{ n^{2/3} / (\log n)^{1/3}\}$, verifying the claim in Section \ref{sec:intro}. In general, it is strictly weaker than the condition in Theorem \ref{thm:asym_normality_htau}, which relies on an extra assumption that $n\Delta^2 = o(\cure_2)$. In the favorable case, as shown in Corollary \ref{cor:asym_normality_htau}, Theorem \ref{thm:asym_normality_hdtau} removes the condition on $\kappa_0$.

The variance estimators $\hat{\sigma}_{\HC j}^{2}$'s are all asymptotically equivalent because the correction terms  are negligible under our asymptotic regime. They are asymptotically conservative estimators of $\sigma_{n}^{2}$, so the Wald-type confidence intervals for $\tau$ are all asymptotically conservative.
\begin{theorem}\label{thm:variance_est}
Under Assumptions \ref{assumption::A1}--\ref{assumption::A4}, there exists a non-negative sequence $a_{n} = o_{\P}(1)$ which depends on $\{ x_i, Y_i(1), Y_i(0)\}_{i=1}^{n}$ such that $\hat{\sigma}_{\HC j}^{2} / \sigma_{n}^{2}\ge 1 - a_{n}$ for all $j\in \{0, 1, 2, 3\}$. 
\end{theorem}

\subsection{Comparison with existing results}\label{sec::comparison}

Theoretical analyses under the finite-population randomization model are challenging due to the lack of probability tools. The closest work to ours is \cite{bloniarz16}, which allows $p$ to grow with $n$ and potentially exceed $n$. However, they assume that the potential outcomes have sparse linear representations based on the covariates, and require $s = o (n^{1/2} / \log p) $ where $s$ is a measure of sparsity. Under additional regularities conditions, they show that $\hat{\tau}(\hat{\beta}_1^{\text{lasso}}, \hat{\beta}_0^{\text{lasso}})$ is consistent and asymptotically normal with $(\hat{\beta}_1^{\text{lasso}}, \hat{\beta}_0^{\text{lasso}})$ being the LASSO coefficients of the covariates. Although the LASSO-adjusted estimator can handle ultra-high dimensional case where $p >\!\!> n$, it has three limitations. First, the requirement $s <\!\!< n^{1/2} / \log p$ is stringent. 
Second, the penalty level of the LASSO depends on unobserved quantities. Although they use the cross-validation to select the penalty level, the theoretical properties of this procedure is still unclear. Third, their ``restrictive eigenvalue condition'' imposes certain non-singularity on the submatrices of the covariate matrix. However, the covariate matrix can be ill-conditioned especially when interaction terms of the basic covariates are included in practice. In addition, this condition is computationally challenging to check. 
Although our results cannot deal with the case of $p > n$, we argue that $p<n$ without sparsity is an important regime in many applications. 

Due to the numerical equivalence of the regression-adjusted estimator to the \OLS~estimator, it is attempting to view our theory as a special case of the existing literature on high dimensional linear models \citep[e.g.][]{huber73, portnoy85, mammen89, lei16, cattaneo2018inference}. However, the two approaches are fundamentally different. They assume a linear model for  the observed outcomes $Y_{i}^{\mathrm{obs}} = \alpha + T_{i}\tau + x_{i}\tran \beta + \eps_{i}$, where $T_i$ denotes the treatment indicator, $x_{i}$ denotes the covariates to be adjusted for, and $\eps_{i}$ denotes the random error for unit $i$. Under their framework, the linear model must be correctly-specified with the random error $\eps_{i}$ being an important component in statistical inference. Moreover, a linear model implicitly assumes treatment-unit additivity, that is, the treatment effect is either constant or uncorrelated with covariates. By contrast, we do not assume any correctly specified linear model for the potential outcomes but treat them as fixed quantities instead. \citet{neyman23}'s model allows for arbitrary treatment effect heterogeneity which suggests that the additive linear model is an inadequate specification \citep{freedman08a}. Therefore, the results in this paper are distinct from those assuming linear models; they are not directly comparable. Similarly, although \cite{wager16} relax the assumption of the constant treatment effect in linear models and can handle the high dimensional case with sparsity level $s = o(n / \log p)$ or $p / n \rightarrow \gamma\in (0, \infty)$, their theory requires $X$ to be normal and $Y(t)$ to be a homoskedastic linear model of $X$. By contrast, our analysis needs none of these assumptions.

\section{Numerical Experiments}\label{sec:experiment}

\subsection{Data generating process}\label{sec::dgp}

To confirm and complement our theory, we use extensive numerical experiments to examine the finite-sample performance of the estimators $\htau$ and $\hdtau$ as well as the variance estimators $\hat{\sigma}^{2}_{\textup{HC}j}$ for $j = 0, 1, 2, 3$. To save space, we only present the results for one synthetic data and relegate the results for other synthetic data to Supplementary Material III. All programs to replicate the results in this article can be found in \texttt{https://github.com/lihualei71/RegAdjNeymanRubin/.}

We set $n = 2000, n_{1} = n\pi_{1}$ with $\pi_{1} = 0.2$ and generate a matrix $\mathcal{X}\in \R^{n\times n}$ with \iid ~entries from $t(2)$. We only generate one copy of $X$ per experiment and keep it fixed. For each exponent $\gamma \in \{0, 0.05, \ldots, 0.7 \}$, we let $p = \lceil n^{\gamma}\rceil$ and take the first $p$ columns of $\mathcal{X}$ as the covariate matrix. In Supplementary Material III, we also simulate $X$ with $N(0,1)$ and $t(1)$ entries with both $\pi_{1}\in \{0.2, 0.5\}$.  We select $t(2)$ distribution for presentation because it is neither too idealized as $N(0, 1)$, for which $\kappa\sim p / n$, nor too irregular as $t(1)$. It is helpful to illustrate and complement our theory.

With $X$, we construct the potential outcomes from $ Y(1) = X\beta_{1}^{*} +  \eps(1) $ and $ Y(0) = X\beta_{0}^{*} + \eps(0)$. 
Because 
$\hat{\beta}_{t} - \beta_{t}^{*} = (X_{t}\tran X_{t})^{-1}X_{t} \tran\eps(t)$ does not depend on $\beta_{t}^{*}$, we take $\beta_{1}^{*} =  \beta_{0}^{*} = 0 \in \R^{p}$ without changing the bias, variance and coverage properties of the estimates $\htau$ and $\hdtau$. We generate $\{ \eps(1), \eps(0) \} $ as realizations of random vectors with \iid ~entries from $N(0,1)$, $t(2)$, or $t(1)$. We also consider another case with $\eps(1) = \eps(0)$ that corresponds to the sharp null hypothesis in Supplementary Material III.
Given $X\in \R^{n\times p}$ and potential outcomes $Y(1), Y(0)\in \R^{n}$, we generate 5000 binary vectors $T\in \R^{n}$, and for each $T$, we observe half of the potential outcomes.

\subsection{Repeated sampling evaluations}\label{subsec:simulation}

Based on the observed data, we obtain two estimates $\htau$ and $\hdtau$, as well as four variance estimates $\hat{\sigma}^{2}_{\textup{HC}j}\,\, (j = 0,1,2,3)$ and the theoretical asymptotic variance $\sigma_{n}^{2}$. Below $\hat{\tau}$ can be either $\htau$ or $\hdtau$, and $\hat{\sigma}^2$ can be any of the five estimates. 
Let $\hat{\tau}_{1}, \ldots, \hat{\tau}_{R}$ denote the estimates in $R=5000$ replicates, and $\tau$ denote the true \ATE. The empirical relative absolute bias is
$
n^{1/2}|  R^{-1}\sum_{k=1}^{R}\hat{\tau}_{k} - \tau | / \sigma_{n}.
$
Similarly, let $\hat{\sigma}_{1}^{2}, \ldots, \hat{\sigma}_{R}^{2}$ denote the variance estimates obtained in $R$ replicates, and $\hat{\sigma}_{*}^{2}$ denote the empirical variance of $( n^{1/2}\hat{\tau}_{1}, \ldots, n^{1/2} \hat{\tau}_{R})$. We compute the standard deviation inflation ratio  
$
 R^{-1}\sum_{k=1}^{R}  \hat{\sigma}_{k} / \hat{\sigma}_{*} .
$
Note that $\hat{\sigma}_{*}^{2}$ is an unbiased estimate of true sampling variance of $n^{1/2}\hat{\tau}$, which can be different from the theoretical asymptotic variance $\sigma_{n}^{2}$.
For each estimate and variance estimate, we compute the $t$-statistic $n^{1/2}(\hat{\tau} - \tau) / \hat{\sigma}$ 
. For each $t$-statistic, we estimate the empirical $95\%$ coverage rate by the proportion within $[-1.96, 1.96]$, the $95\%$ quantile range of $N(0,1)$.

In summary, we compute three measures defined above: the relative bias, standard deviation inflation ratio, and $95\%$ coverage rate.  We repeat $50$ times using different random seeds and record the medians of each measure. Figure \ref{fig::simulation01} summarizes the results.

\subsection{Results}\label{subsec:simulation_results}

\begin{figure}[htp]
	\centering
	\begin{subfigure}{0.6\textwidth}  
		\includegraphics[width=\textwidth]{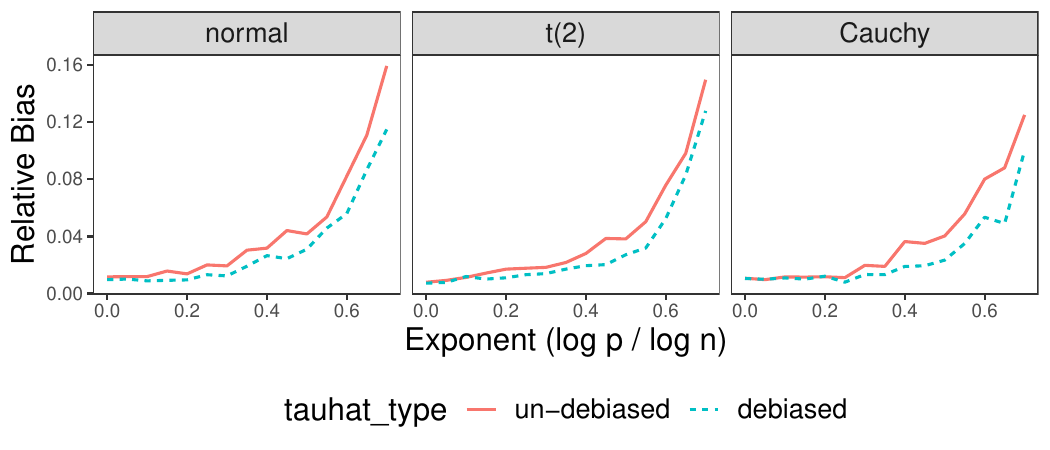}
		\caption{ Relative bias of $\hdtau$ and $\htau$.}\label{fig:t2_bias}
	\end{subfigure}
	\begin{subfigure}{0.6\textwidth}  
		\includegraphics[width=\textwidth]{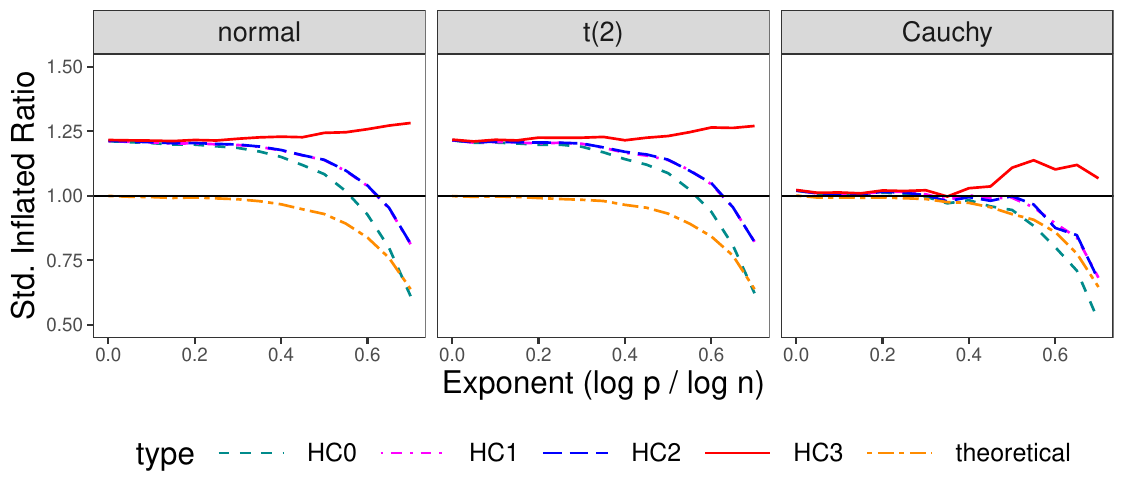}
		\caption{Ratio of standard deviation between five standard deviation estimates, $\sigma_{n}, \hat{\sigma}_{\textup{HC}0}, \hat{\sigma}_{\textup{HC}1}, \hat{\sigma}_{\textup{HC}2}, \hat{\sigma}_{\textup{HC}3}$, and the true standard deviation of $\htau$.}\label{fig:t2_sdinflate}
	\end{subfigure}
		\begin{subfigure}{0.6\textwidth}  
		\includegraphics[width=\textwidth]{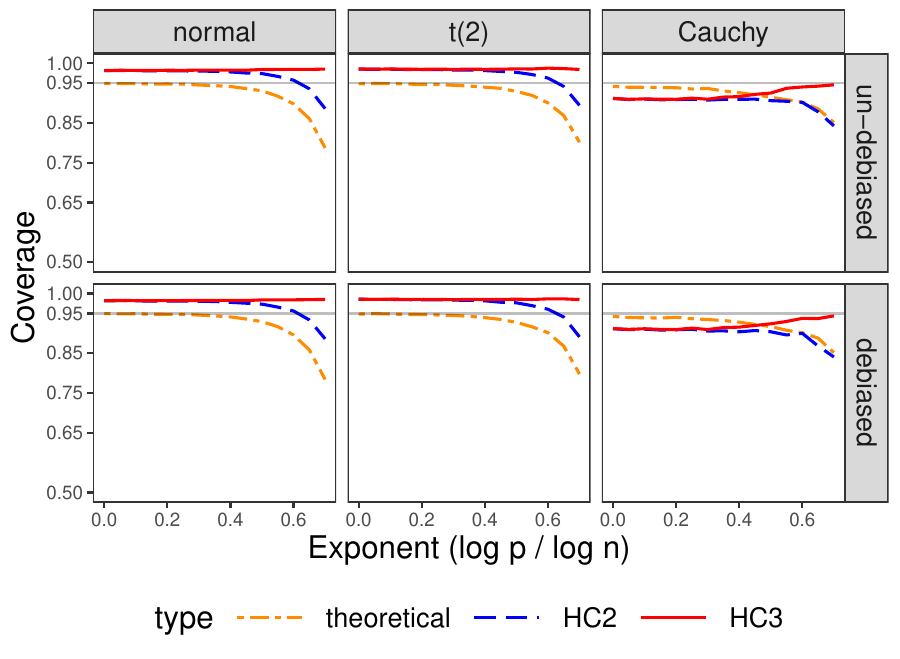}
		\caption{Empirical $95\%$ coverage rates of $t$-statistics derived from two estimators and four variance estimators (
                  ``theoretical'' for $\sigma_{n}^{2}$, ``HC2'' for $\hat{\sigma}_{\text{HC}2}^{2}$ and ``HC3''  for $\hat{\sigma}_{\text{HC}3}^{2}$)}\label{fig:t2_coverage}
	\end{subfigure}
	\caption{Simulation with $\pi_{1} = 0.2$. $X$ is a realization of a random matrix with $t(2)$ entries, and $\eps(t)$ is a realization of a random vector with entries from a distribution corresponding to each column. }\label{fig::simulation01}
\end{figure}

From Figure \ref{fig:t2_bias}, $\hdtau$ does reduce the bias regardless of the distribution of potential outcomes, especially for moderately large $p$.
For standard deviation inflation ratios, the true sampling variances of $n^{1/2}\htau$ and $n^{1/2}\hdtau$ are almost identical and thus we set the sampling variance of $n^{1/2}\htau$ as the baseline variance $\hat{\sigma}_{*}^{2}$. Figure \ref{fig:t2_sdinflate} shows an interesting phenomenon that the theoretical asymptotic variance $\sigma_{n}^{2}$ tends to underestimate the true sampling variance for large $p$. Theorem \ref{cor:expansion} partially suggests this. The theoretical asymptotic variance is simply the variance of $\hat{\tau}_e$ while the finite sample variance also involves the remainder, which can be large in the presence of high dimensional or influential observations. All variance estimators overestimate $\sigma_{n}^{2}$ because they all ignore the third term of $\sigma_{n}^{2}$. However, all estimators, except the HC3 estimator, tend to underestimate the true sampling variance for large $p$. By contrast, the HC3 estimator does not suffer from anti-conservatism in this case.

Figures \ref{fig:t2_sdinflate} shows that HC0 and HC1 variance estimates lie between the theoretical asymptotic variance and the HC2 variance estimate. For better visualization, Figures \ref{fig:t2_coverage} only shows the $95\%$ coverage rates of $t$-statistics computed from $\sigma_{n}^{2}, \hat{\sigma}_{\textup{HC}2}^{2}$ and $\hat{\sigma}_{\textup{HC}3}^{2}$, based on which we  draw the following conclusions. First, as we pointed out previously, the coverage rates based on two estimates are almost identical because the relative bias is small in these scenarios. Second, as Figures \ref{fig:t2_sdinflate} suggests, the $t$-statistic with HC3 variance estimate has the best coverage rate, which is robust with covariates of an increasing dimension. By contrast, the theoretical asymptotic variance and the HC$j$ $(j=0,1,2)$ variance estimates yield significantly lower coverage rates for large $p$. We recommend $\hat{\sigma}_{\textup{HC}3}^{2}$.

\subsection{Effectiveness of debiasing}\label{subsec:simulation_debiase}

In the aforementioned settings, $\hdtau$ yields almost identical inference as $\htau$. This is not surprising because in the above scenarios the potential outcomes are generated from linear models and thus the regression-adjusted estimator has bias close to zero.
However, in practice, the potential outcomes might not have prefect linear relationships with the covariates. To illustrate the potential benefits of debiasing, we consider the worst-case situation which maximizes the bias. Specifically, we consider the case where $\eps(0) = \eps$ and $\eps(1) = 2\eps$ for some vector $\eps$ that satisfies \eqref{eq:error_key_property} with sample variance $1$. To maximize the bias term, we take $\eps$ as the solution of
\begin{equation}
  \label{eq:worst_pout}
  \max_{\eps\in \R^{n}}\bigg|\frac{n_{1}}{n_{0}}\Delta_{0} - \frac{n_{0}}{n_{1}}\Delta_{1}\bigg| 
  = 
    \max_{\eps\in \R^{n}}\lb\frac{2n_{0}}{n_{1}} - \frac{n_{1}}{n_{0}}\rb\bigg|\sum_{i=1}^{n}H_{ii}\eps_{i}\bigg|, 
\end{equation}
such that $\|\eps\|_{2}^{2} / n = 1$ and $ X\tran \eps = \one\tran \eps = 0$. Supplementary Material III gives more details of constructing $\eps$. From \eqref{eq:worst_pout}, the bias is amplified when the group sizes are unbalanced, and it effectively imposes a non-linear relationship between potential outcomes and covariates.

We perform simulation detailed in Section \ref{subsec:simulation} based on potential outcomes in \eqref{eq:worst_pout} and report the relative bias and coverage rate to demonstrate the effectiveness of debiasing. To save space, we only report the coverage rates based on $\hat{\sigma}_{\textup{HC}2}^{2}$ and $\hat{\sigma}_{\textup{HC}3}^{2}$. Figure \ref{fig::simulation02} summarizes the results. 
Unlike the previous settings, the relative bias in this setting is large enough to affect the coverage rate.  
The debiased estimator reduces a fair proportion of bias  and improves the coverage rate especially when the dimension is high. We provide experimental results in more settings in Supplementary Material III.

\begin{figure}[t]
	\centering
	\begin{subfigure}{0.46\textwidth}  
          \centering
		\includegraphics[width=0.8\textwidth]{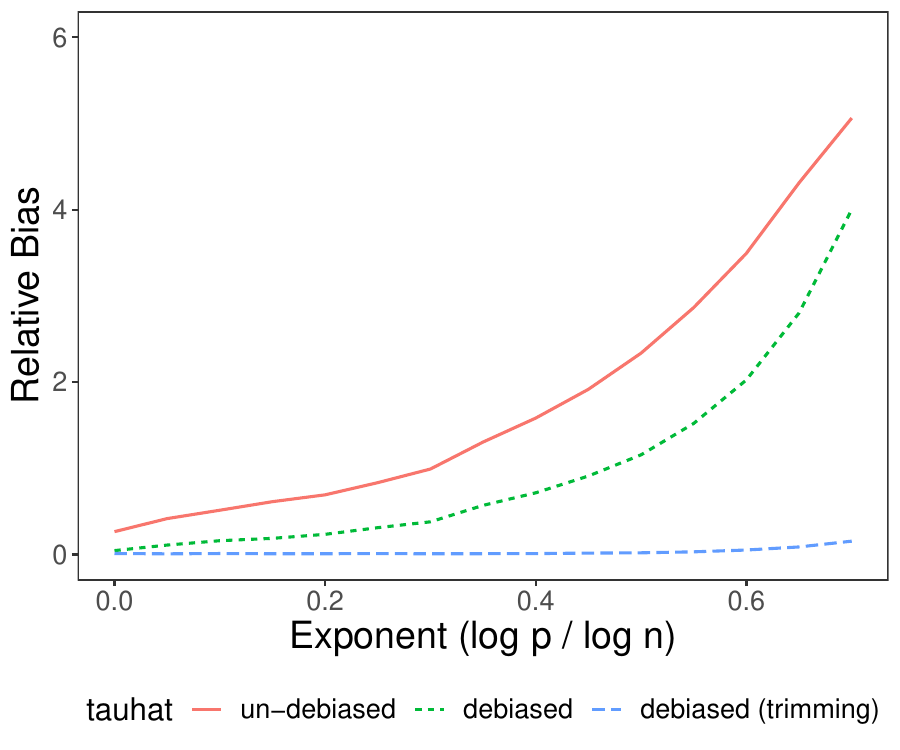}
		\caption{ Relative bias of $\hdtau$ and $\htau$.}\label{fig:t2_bias_worst}
	\end{subfigure}
        \begin{subfigure}{0.45\textwidth}  
          \centering
          \includegraphics[width=1\textwidth]{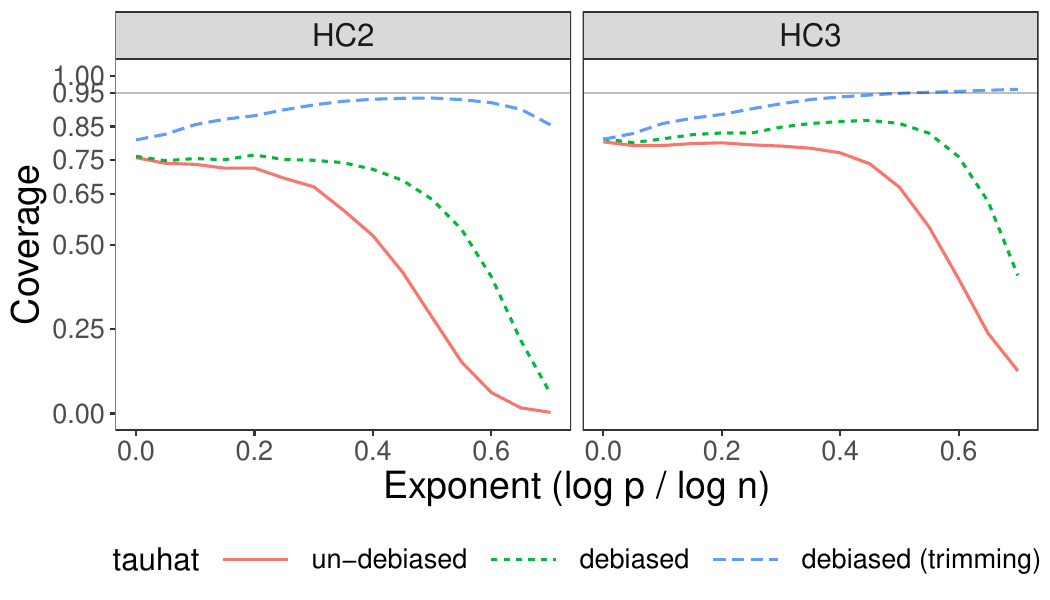}
		\caption{Empirical $95\%$ coverage rates of $t$-statistics derived from two estimators and two variance estimators (``HC2'' for $\hat{\sigma}_{\text{HC}2}^{2}$ and ``HC3''  for $\hat{\sigma}_{\text{HC}3}^{2}$)}\label{fig:t2_coverage_worst}
	\end{subfigure}
	\caption{Simulation. $X$ is a realization of a random matrix with $t(2)$ entries, $\pi_{1} = 0.2$ and $\eps(t)$ is defined in \eqref{eq:worst_pout}.}\label{fig::simulation02}
\end{figure}

\subsection{Trimming covariates}\label{subsec:simulation_regularization}

Because our theory holds even for mis-specified linear models, we can preprocess the covariate matrix $X$ arbitrarily without changing the estimand, provided that the preprocessing step does not involve $T$ or $Y^{\mathrm{obs}}$. This is a feature of our finite-population theory. Moreover, our asymptotic theory suggests that the maximum leverage score of the design matrix affects the properties of $\htau$ and $\hdtau$. When there are many influential observations, it is beneficial to reduce $\kappa$ by trimming the values of covariates before regression adjustment.  Importantly, trimming covariates should not use any information of $T$ or $Y^{\mathrm{obs}}$.

 For the cases considered in previous subsections, we consider trimming each covariate at its 2.5\% and 97.5\% quantiles. For the 50 design matrices used in Section \ref{sec:experiment} with $p = \lceil n^{2/3}\rceil$ and $n = 2000$, the average of $\kappa$ is $0.9558$ with standard error $0.0384$. After trimming, the average of $\kappa$ reduces dramatically to $0.0704$ with standard error $0.0212$. Figure \ref{fig::simulation02} shows that the bias is significantly reduced and the coverage rate gets drastically improved after trimming covariates.

\section*{Acknowledgment}
We thank the editor, associate editor, two reviewers, Cheng Gao, and Luke Miratrix for helpful suggestions. Peng Ding was partially supported by a grant from the U.S. National Science Foundation.

\bibliographystyle{biometrika}
\bibliography{high_dim_Neyman_Rubin}

\clearpage

 \appendix

\renewcommand {\thefigure} {S\arabic{figure}}

~\\
\begin{center}
  \begin{LARGE}
    \textbf{Supplementary Materials}
  \end{LARGE}
\end{center}

~\\

Supplementary Material I gives all the proofs, including Section \ref{sec:technicallemmas} for technical lemmas, Section \ref{sec:main_proofs} for the proofs of the main results, Section \ref{sec:concentration} for the concentration inequalities for sampling without replacement, Section \ref{app:mean_variance} for the proof of Lemma \ref{thm:mean},  and Section \ref{subapp:lemmas} for the proofs of other lemmas in Section \ref{subsec:lemmas2}. 

 Supplementary Material II discusses the assumptions, including Section \ref{sec::discussion-assumptions} for Propositions \ref{prop:kappa_iid} and \ref{prop:eq}, Section \ref{sec::useful-results-proofs} for the lemmas, and Sections \ref{section:proof-of-F1} and \ref{subapp:cure} for the proofs of Propositions \ref{prop:kappa_iid} and \ref{prop:eq}, respectively.

     Supplementary Material III contains additional experiments.

 \bigskip

\begin{center}
  \begin{Large}
    Supplementary Material I: Proofs
  \end{Large}
\end{center}
 
\section{Technical Lemmas}\label{sec:technicallemmas}
We start by stating several useful lemmas and proceed to prove main results in Section \ref{sec:main} in Section \ref{sec:main_proofs}. The lemmas will be proved later in Sections \ref{sec:concentration} and \ref{app:mean_variance}.

\subsection{Some general results for sampling without replacement}\label{subsec:lemmas1}

Completely randomized experiments have deep connections with sampling without replacement because the treatment and control groups are simple random samples from a finite population of $n$ units. Let $\T$ denote a random size-$m$ subset of $  \{1, \ldots, n\}$ over all $\binom{n}{m}$ subsets, and $\S^{p-1} = \{  (\omega_1, \ldots, \omega_p)\tran :  \omega_1^2+\cdots + \omega_p^2 =1 \}$ denote the $(p-1)$-dimensional unit sphere. 
The first lemma gives the mean and variance of the sample total from sampling without replacement. See \citet[][Theorem 2.2]{cochran07} for a proof. 

\begin{lemma}\label{lem:var_sum}
Let $(w_1, \ldots, w_n)$ be fixed scalars with mean $\bar{w} = n^{-1}\sum_{i=1}^{n}w_i$. Then $\sum_{i\in \T} w_i$ has mean $m\bar{w}$ and variance
\[\Var\lb\sum_{i\in \T} w_i  \rb = \frac{m(n-m)}{n(n-1)}\sum_{i=1}^{n}(w_i - \bar{w})^{2}.\]
\end{lemma}

The second lemma gives the Berry--Esseen-type bound for the finite population central limit theorem. See \citet{bikelis69} and \citet{hoglund78} for proofs. 

\begin{lemma}\label{prop:finite_CLT}
  Let $(w_{1}, \ldots, w_{n})$ be fixed scalars with
$\bar{w} = n^{-1}\sum_{i=1}^{n}w_{i}$ and $\quad S_w^{2} = \sum_{i=1}^{n}(w_{i} - \bar{w})^{2}.$
Let $m = n f$ for some $f \in (0, 1)$. Then 
\begin{align*}
d_{\textup{K}}\lb \frac{\sum_{i\in \T  }(w_{i} - \bar{w})}{ S_w (f(1-f))^{1/2}}, N(0, 1)\rb
\le \frac{C}{(f(1-f))^{1/2}}\frac{\sum_{i=1}^{n}|w_{i} - \bar{w}|^{3}}{S_w^{3}} 
 \le \frac{C}{(f(1-f))^{1/2}}\frac{\max_{1\leq i\leq n}|w_{i} - \bar{w}|}{S_w},
\end{align*}
where $d_\textup{K}$ denotes the Kolmogorov distance between two distributions, and $C$ is a universal constant. 
\end{lemma}

The following two lemmas give novel vector and matrix concentration inequalities for sampling without replacement.

\begin{lemma}\label{thm:vector_concentration}
Let $(u_{1}, \ldots, u_{n})$ be a finite population of $p$-dimensional vectors with 
$\sum_{i=1}^{n}u_{i} = 0.$
Then for any $\delta\in (0, 1)$, with probability at least $1 - \delta$,
\[\left\|\sum_{i\in \T}u_{i}\right\|_{2}\le \|U\|_{F}\lb\frac{m(n - m)}{n(n - 1)}\rb^{1/2} + \|U\|_{\op}\lb 8\log \frac{1}{\delta}\rb^{1/2}\]
where $u_{i}\tran$ is the $i$-th row of the matrix $U\in\R^{n\times p}$. 
\end{lemma}

\begin{lemma}\label{thm:matrix_concentration}
Let $(V_{1}, \ldots, V_{n})$ be a finite population of $(p\times p)$-dimensional Hermittian matrices with
$\sum_{i=1}^{n}V_{i} = 0.$
Let $C(p) = 4(1 + \lceil 2\log p\rceil)$, and
\[\nu^{2} = \left\|\frac{1}{n}\sum_{i=1}^{n}V_{i}^{2}\right\|_{\op}, \quad \nu_{-}^{2} = \sup_{\omega\in \S^{p-1}}\frac{1}{n}\sum_{i=1}^{n}(\omega\tran V_{i}\omega)^{2}, \quad \nu_{+} = \max_{1\leq i \leq n}\|V_{i}\|_{\op}. \] 
Then for any $\delta \in (0, 1)$, with probability at least $1 - \delta$, 
\[\left\|\sum_{i\in \T}V_{i}\right\|_{\op}\le (nC(p))^{1/2}\nu + C(p)\nu_{+} + \lb 8n\log \frac{2}{\delta}\rb^{1/2}\,\nu_{-}.\]
\end{lemma}

The following lemma gives the mean and variance of the summation over randomly selected rows and columns from a deterministic matrix $Q \in \R^{n\times n}.$ 

\begin{lemma}\label{thm:mean}
Let $Q\in \R^{n\times n}$ be a deterministic matrix, and $Q_{\T} \equiv \sum_{i,j\in \T} Q_{ij}$. Assume $n\ge 4$. Then
 \[\E Q_{\T} = \frac{m(n - m)}{n(n - 1)}\tr(Q) + \frac{m(m - 1)}{n(n - 1)} \one\tran Q\one.\] 
 If $Q$ further satisfies $
\one\tran Q = Q\one = 0,
$
then 
\[\Var(Q_{\T}) \le \frac{m(n - m)}{n(n - 1)}\|Q\|_{F}^{2}.\] 
\end{lemma}

Lemmas \ref{thm:vector_concentration}--\ref{thm:mean} are critical for our proofs. The proofs of Theorem \ref{thm:vector_concentration} and \ref{thm:matrix_concentration} are presented in Section \ref{sec:concentration} and the proof of Theorem \ref{thm:mean} is presented in Section \ref{app:mean_variance}. They are novel tools to the best of our knowledge and potentially useful in other contexts such as survey sampling, matrix sketching, and transductive learning.

\subsection{Some results particularly useful for our setting} \label{subsec:lemmas2}

We first give an implication of Assumption \ref{assumption::A3}, a lower bound on $\sigma_{n}^{2}$ under Assumption \ref{assumption::A1}. 
\begin{lemma}\label{lem:A3}
  Under Assumptions \ref{assumption::A1} and \ref{assumption::A3}, 
$
\sigma_{n}^{2} \ge \eta\min\left\{ n_{1} / n_{0},  n_{0} / n_{1} \right\}\cure_{2}.
$
\end{lemma}

Recall $H_t= X_{t}(X_{t}\tran X_{t})^{-1}X_{t}\tran $ and define $\Sigma_t = n_t^{-1} X_{t}\tran X_{t}$ $(t=0,1)$.
The following explicit formula is the starting point of our proof.
\begin{lemma}\label{lem:tauhat}
We have
\begin{equation}
  \label{eq:ATE}
  \htau - \tau = \frac{\oneone \tran e_{1}(1) / n_{1} - \oneone \tran H_1 e_{1}(1) / n_{1}}{1 - \oneone \tran H_1 \oneone  / n_{1}} - \frac{\onezero \tran e_{0}(0) / n_{0} - \onezero \tran H_0e_{0}(0) / n_{0}}{1 - \onezero \tran H_0\onezero  / n_{0}}.
\end{equation}
\end{lemma}

As in the  main text, we assume that the covariate matrix has centered columns and full column rank. 
The quantities $\mu_{t}$, $e(t)$, and our estimators $(\htau, \hdtau)$ are all invariant if $X$ is transformed to $X Z$ for any full rank matrix $Z \in \R^{p\times p}$.
Thus, without loss of generality, we further assume
\begin{equation}
  \label{eq:orthonormal}
n^{-1} X\tran X = \id .
\end{equation}
Otherwise, suppose $X$ has the singular value decomposition $U\Sigma V\tran $ with $U\in \R^{n\times p}, \Sigma, V\in \R^{p\times p}$, then we can replace $X$ by $n^{1/2}U = X(n^{1/2}V \Sigma^{-1})$ to ensure \eqref{eq:orthonormal}. We can verify that the key properties in \eqref{eq:error_key_property} still hold. Assuming \eqref{eq:orthonormal}, we can rewrite the hat matrix and the leverage scores as
\begin{equation}\label{eq:hatmatrix}
H = n^{-1} XX\tran , \quad H_{ii} =  n^{-1}   \|x_{i}\|_{2}^{2} ,\quad H_{ij} = n^{-1} x_i \tran x_j .
\end{equation}
Note that the invariance property under the standardization \eqref{eq:orthonormal} is a feature of the regression adjustment based on  ordinary least squares. It does not hold for many other estimators \citep[e.g.,][]{bloniarz16, wager16}.

We will repeatedly use the following results to obtain the stochastic orders of the terms in \eqref{eq:ATE}. They are consequences of Lemmas \ref{thm:vector_concentration} and \ref{thm:matrix_concentration}.

\begin{lemma}\label{lem:vector_concentration}
Under Assumption \ref{assumption::A1}, for $t = 0, 1$,
\begin{gather*} 
\frac{\onet\tran e_t(t)  }{n_{t}} = O_\P\lb \lb\frac{\cure_{2}}{n}\rb^{1/2}\rb, \quad \left\|\frac{X_{t}\tran \onet }{n_{t}}\right\|_{2} = O_\P\lb \lb\frac{p}{n}\rb^{1/2}\rb ,  \quad 
\left\|\frac{X_{t}\tran e_{t}(t)}{n_{t}}\right\|_{2} = O_\P\lb \lb\cure_{2}\kappa\rb^{1/2}\rb .
\end{gather*}
\end{lemma}

\begin{lemma}\label{lem:matrix_concentration}
Under Assumptions \ref{assumption::A1}, \ref{assumption::A2} and \eqref{eq:orthonormal}, for $t = 0, 1$, 
\begin{gather*} 
\left\|\Sigma_t - \id \right\|_{\op} = O_\P\lb \lb\kappa\log p\rb^{1/2}\rb, \quad \left\|\Sigma_t^{-1}\right\|_{\op} = O_\P(1), \quad 
\left\|\Sigma_t^{-1} - \id \right\|_{\op} = O_\P\lb (\kappa\log p)^{1/2}\rb. 
\end{gather*}
\end{lemma}

The following lemma states some key properties of an intermediate quantity, which will facilitate our proofs.
\begin{lemma}\label{lemma:Mt}
Define $Q(t)   = H\diag\lb e(t)\rb = \lb H_{ij}e_{j}(t)    \rb_{i, j = 1}^{n}$. It satisfies
\begin{eqnarray*}
\one\tran Q(t) =0, \quad Q(t) \one =0,\quad 
\one\tran Q(t) \one = 0, \quad 
\textup{tr}(Q(t) ) =n  \Delta_t,\quad    \|Q(t) \|_{F}^{2} = \sum_{i=1}^{n}e_{i}^2(t)   H_{ii} \le n \cure_{2}\kappa .
 \end{eqnarray*}
\end{lemma}

\section{Proofs of the main results}\label{sec:main_proofs}
\subsection{Proof of the asymptotic expansions}\label{subsec:proof_expansion}

We first prove a more refined expansion than Theorem \ref{cor:expansion}.
  \begin{theorem}\label{thm:expansion}
  Under Assumptions \ref{assumption::A1} and \ref{assumption::A2},
  \begin{align}
   \htau - \tau  = \lb\frac{\oneone \tran e_{1}(1) }{n_{1}} - \frac{\onezero \tran e_{0}(0)}{n_{0}}\rb + \lb\frac{n_{1}}{n_{0}}\Delta_{0} - \frac{n_{0}}{n_{1}}\Delta_{1}\rb  
 + O_\P\lb \lb\frac{\cure_{2} (\kappa^{2} p \log p + \kappa) }{n}\rb^{1/2}\rb.\label{eq:expansion}
  \end{align}
\end{theorem}

\begin{proof}[Proof of Theorem \ref{thm:expansion}]
We need to analyze the terms in \eqref{eq:ATE}. First, by Lemmas \ref{lem:vector_concentration} and \ref{lem:matrix_concentration},
\begin{align*}
  \frac{\onet \tran H_t \onet }{n_{t}} &= \frac{\onet \tran X_{t}}{n_{t}}\Sigma_t^{-1}\frac{X_{t}\tran \onet }{n_{t}}
  \leq    \left\| \Sigma_t^{-1}  \right\|_{\op}  \left \|  \frac{X_{t}\tran \onet }{n_{t}}  \right \|_2^2
  = O_\P\lb \frac{p}{n}\rb.
\end{align*}
Using \eqref{eq:pn} that $p = o(n)$, we obtain that 
\begin{equation}\label{eq:expansion1}
\frac{1}{1 - \onet \tran H_t \onet  / n_{t}} = 1 + O_\P\lb\frac{p}{n}\rb.
\end{equation}
Second,  
\begin{align}
  \frac{\onet \tran H_t e_{t}(t)}{n_{t}} = \frac{\onet \tran X_{t}}{n_{t}}\Sigma_t^{-1}\frac{X_{t}\tran e_{t}(t)}{n_{t}} 
=  \frac{\onet \tran X_{t}}{n_{t}}\frac{X_{t}\tran e_{t}(t)}{n_{t}} 
+ \frac{\onet \tran X_{t}}{n_{t}}\lb \Sigma_t^{-1} - \id \rb\frac{X_{t}\tran e_{t}(t)}{n_{t}} \equiv   R_{t1} + R_{t2} . \label{eq:R1+R2} 
\end{align}
Note that here we do not use the naive bound for $\onet \tran H_t e_{t}(t) / n_t$ as for $\onet \tran H_t \onet / n_t$ in \eqref{eq:expansion1} because this gives weaker results. Instead, we bound $R_{t1}$ and $ R_{t2}$ separately. Lemmas \ref{lem:vector_concentration} and \ref{lem:matrix_concentration} imply
\begin{align}\label{eq:R2}
  R_{t2}  \le \left\|  \Sigma_t^{-1} - \id  \right\|_\op  
  \left\| \frac{X_{t}\tran \onet }{n_{t}} \right\|_2  
  \left\| \frac{X_{t}\tran e_{t}(t)}{n_{t}} \right\|_2  
  = O_\P\lb \lb\frac{\cure_{2}\kappa^{2} p \log p}{n}\rb^{1/2}\rb.
\end{align}
We apply Chebyshev's inequality to obtain that
\begin{equation}
  \label{eq:Markov}
  R_{t1} = \E R_{t1} + O_\P\lb\Var(R_{t1})^{1/2}\rb.
\end{equation}
Therefore, to bound $R_{t1}$, we need to calculate its first two moments. 
Recalling \eqref{eq:hatmatrix} and the definition of $Q(t)$ in Lemma \ref{lemma:Mt}, we have
\begin{align}
  R_{t1} &= \frac{1}{n_{t}^{2}}\lb\sum_{i\in \T_{t}}x_{i}\tran \rb\lb \sum_{i\in \T_{t}} x_{i}e_{i}(t)\rb  
  = \frac{1}{n_{t}^{2}} \sum_{i\in \T_{t}} \sum_{j\in \T_{t}} x_{i}\tran x_{j}e_{j}(t)  \nonumber\\
  &= \frac{1}{n_{t}^{2}} \sum_{i\in \T_{t}} \sum_{j\in \T_{t}} nH_{ij}e_{j}(t)   = \frac{n}{n_{t}^{2}}\sum_{i\in \T_{t}} \sum_{j\in \T_{t}} Q_{ij}(t) .\label{eq:R1}
\end{align}
Lemmas \ref{thm:mean} and \ref{lemma:Mt} imply the expectation of $R_{t1}$:
\begin{align}
  \E R_{t1} = \frac{n}{n_{t}^{2}}\lb\frac{n_{1}{n_{0}}}{n(n - 1)} \tr\lb Q(t) \rb + \frac{n_{t}(n_{t} - 1)}{n (n - 1)}\one\tran  Q(t) \one \rb 
 = \frac{n n_{1}n_{0}}{n_{t}^{2}(n - 1)}\Delta_{t} 
 = \frac{n_{1}n_{0}}{n_{t}^{2}}\Delta_{t} + O\lb\frac{|\Delta_{t}|}{n}\rb.\label{eq:ER1}
\end{align}
We then bound the variance of $ R_{t1} $:  
\begin{eqnarray}
  \Var(R_{t1}) &=& \frac{n^{2}}{n_{t}^{4}}\Var\lb\sum_{i,j\in \T_{t}}Q_{ij}(t) \rb 
\leq    \frac{n^{2}}{n_{t}^{4}}   \frac{n_1 n_0}{  n(n-1) } \| Q(t) \|_F^2  \label{eq::varianceofQ} \\
&\leq &  \frac{n^{2}}{n_{t}^{4}}   \frac{n_1 n_0}{  n(n-1) } n\mathcal{E}_2 \kappa 
=  O\lb\frac{\cure_{2}\kappa}{n}\rb, \label{eq:VarR1}
\end{eqnarray}
 where \eqref{eq::varianceofQ} follows from Lemma \ref{thm:mean}, \eqref{eq:VarR1} follows from Lemma \ref{lemma:Mt} and Assumption \ref{assumption::A1}. 
Putting \eqref{eq:R1+R2}--\eqref{eq:ER1} and \eqref{eq:VarR1} together, we obtain that
\begin{align}
  \frac{\onet \tran H_t e_{t}(t)}{n_{t}} &= \frac{n_{1}n_{0}}{n_{t}^{2}}\Delta_{t} + O_\P\lb \lb\frac{\cure_{2}\kappa^{2} p \log p}{n}\rb^{1/2} + \frac{|\Delta_{t}|}{n} + \lb\frac{\cure_{2}\kappa}{n}\rb^{1/2}\rb.\label{eq:annoyingterm1}
\end{align}
By \eqref{eq:pn} and \eqref{eq:Deltat}, \eqref{eq:annoyingterm1} further simplifies to 
\begin{equation}
  \label{eq:annoyingterm2}
  \frac{\onet \tran H_t e_{t}(t)}{n_{t}} = \frac{n_{1}n_{0}}{n_{t}^{2}}\Delta_{t} + O_\P\lb \lb\frac{\cure_{2}\kappa^{2} p\log p}{n}\rb^{1/2} + \lb\frac{\cure_{2}\kappa}{n}\rb^{1/2}\rb.
\end{equation}
Using Lemma \ref{lem:vector_concentration}, \eqref{eq:annoyingterm2}, and the fact that $\kappa \le 1$, we have 
\begin{align}
  &\frac{\onet \tran e_{t}(t)}{n_{t}} - \frac{\onet \tran H_t  e_{t}(t)}{n_{t}}= O_\P\lb \lb\frac{\cure_{2}}{n}\rb^{1/2} + \Delta + \lb\frac{\cure_{2}\kappa^{2} p\log p}{n}\rb^{1/2}\rb.\label{eq:annoyingterm3}
\end{align}

Finally, putting \eqref{eq:expansion1}, \eqref{eq:annoyingterm2} and \eqref{eq:annoyingterm3} together into \eqref{eq:ATE}, we obtain that
\begin{align}
 &\htau - \tau   \nonumber\\
  =& \lb \frac{\oneone \tran e_{1}(1)}{n_{1}} - \frac{\oneone \tran H_1 e_{1}(1)}{n_{1}}\rb \lb 1 + O_\P\lb\frac{p}{n}\rb\rb
  - \lb \frac{\onezero \tran e_{0}(0)}{n_{0}} - \frac{\onezero \tran H_0e_{0}(0)}{n_{0}}\rb \lb 1 + O_\P\lb\frac{p}{n}\rb\rb
\nonumber\\
 =& \frac{\oneone \tran e_{1}(1)}{n_{1}} - \frac{\onezero \tran e_{0}(0)}{n_{0}} + \frac{\onezero \tran H_0e_{0}(0)}{n_{0}} - \frac{\oneone \tran H_1 e_{1}(1)}{n_{1}} 
 + O_\P\lb \lb\frac{p^{2}\cure_{2}}{n^{3}}\rb^{1/2} + \frac{p\Delta}{n} + \lb\frac{\cure_{2}\kappa^{2} p^{3}\log p}{n^{3}}\rb^{1/2}\rb
\nonumber\\
 =& \frac{\oneone \tran e_{1}(1)}{n_{1}} - \frac{\onezero \tran e_{0}(0)}{n_{0}} + \frac{n_{1}}{n_{0}}\Delta_{0} - \frac{n_{0}}{n_{1}}\Delta_{1} 
 + O_\P\lb \lb\frac{p^{2}\cure_{2}}{n^{3}}\rb^{1/2} + \frac{p\Delta}{n} + \lb\frac{\cure_{2}\kappa^{2} p\log p}{n}\rb^{1/2}   + \lb\frac{\cure_{2}\kappa}{n}\rb^{1/2}\rb. \label{eq::Opterm}
\end{align}
where \eqref{eq::Opterm} uses \eqref{eq:pn} that $p = o(n)$. The fourth terms dominates the first term in \eqref{eq::Opterm} because $1 \ge \kappa\ge p / n$. The third term dominates the second term in \eqref{eq::Opterm}  because, by \eqref{eq:Deltat}, 
\[\frac{p\Delta}{n} \le \kappa \Delta \le \kappa^{1/2}\Delta = O\lb \lb\frac{\cure_{2}\kappa^{2} p}{n}\rb^{1/2}\rb.\]
Deleting the first two terms in \eqref{eq::Opterm}, we complete the proof. 
\end{proof}

\begin{proof}[Proof of Theorem \ref{cor:expansion}]
Assumption \ref{assumption::A1} implies that
$
  \frac{n_{1}}{n_{0}}\Delta_{0} - \frac{n_{0}}{n_{1}}\Delta_{1} = O\lb \Delta\rb,
$
which, coupled with Theorem \ref{thm:expansion}, implies
\eqref{eq:expansion_htau}. 
The key is to prove the result for the debiased estimator. By definition,
\begin{align*}
   \hdtau - \tau  &= \frac{\oneone \tran e_{1}(1)}{n_{1}} - \frac{\onezero \tran e_{0}(0)}{n_{0}} + \frac{n_{1}}{n_{0}}(\Delta_{0} - \hat{\Delta}_{0}) - \frac{n_{0}}{n_{1}}(\Delta_{1} - \hat{\Delta}_{1})\\
  & \quad + O_\P\lb \lb\frac{\cure_{2}\kappa^{2} p\log p}{n}\rb^{1/2}  + \lb\frac{\cure_{2}\kappa}{n}\rb^{1/2}\rb,
\end{align*}
and therefore, the key is to bound $|\Delta_{t} - \hat{\Delta}_{t}|$. 

We introduce an intermediate quantity 
$
\td{\Delta}_{t} =  n_{t}^{-1}  \sum_{i\in \T_{t}}H_{ii}e_{i}(t).
$
It has mean
$
\E \td{\Delta}_{t} = \Delta_{t} 
$
and variance
\begin{equation}
  \label{eq:VartdDeltat}
\Var(\td{\Delta}_{t}) \le  \frac{1}{n_{t}^{2}}\frac{n_{1}n_{0}}{n(n - 1)}\sum_{i=1}^{n}H_{ii}^{2}e_{i}^2(t)  \le \frac{n\cure_{2}\kappa^{2}}{n_{t}^{2}} = O\lb\frac{\cure_{2}\kappa^{2}}{n}\rb,
\end{equation}
from Lemma \ref{lem:var_sum} and Assumption \ref{assumption::A1}. Equipped with the first two moments, we use Chebyshev's inequality to obtain  
\begin{equation}
  \label{eq:tdDeltat-Deltat}
  |\td{\Delta}_{t} - \Delta_{t}| = O_\P\lb \lb\frac{\cure_{2}\kappa^{2}}{n}\rb^{1/2}\rb.
\end{equation}
Next we bound $|\hat{\Delta}_{t} - \td{\Delta}_{t}|$. The Cauchy--Schwarz inequality implies 
\begin{equation}
  \label{eq:hatDeltat-tdDeltat1}
  |\hat{\Delta}_{t} - \td{\Delta}_{t}| \le \frac{1}{n_{t}}\sum_{i\in \T_{t}} H_{ii}|\hat{e}_{i} - e_{i}(t)|\le \lb\frac{1}{n_{t}}\sum_{i\in \T_{t}}H_{ii}^{2}\rb^{1/2}\lb\frac{1}{n_{t}}\sum_{i\in \T_{t}}(\hat{e}_{i} - e_{i}(t))^{2}\rb^{1/2}.
\end{equation}
First, 
\begin{equation}
  \label{eq:hatDeltat-tdDeltat3}
  \frac{1}{n_{t}}\sum_{i\in \T_{t}}H_{ii}^{2}\le \frac{n\kappa}{n_{t}}\lb\frac{1}{n}\sum_{i=1}^{n}H_{ii}\rb = O\lb\frac{\kappa p}{n}\rb.
\end{equation}
Second, using the fact
$
\hat{e}_{t} = (\id - H_{t})e_{t}(t),
$
we have 
\begin{align}
  \frac{1}{n_{t}}\sum_{i\in \T_{t}}(\hat{e}_{i} - e_{i}(t))^{2} &= \frac{1}{n_{t}}\|\hat{e}_{t} - e_{t}(t)\|_{2}^{2} =
\frac{1}{n_t}e_{t}\tran (t)  H_t e_{t}(t)\nonumber\\
  & =  \lb\frac{X_{t}\tran e_{t}(t)}{n_{t}}\rb\tran \Sigma_t^{-1}\frac{X_{t}\tran e_{t}(t)}{n_{t}} \le  \left\|\Sigma_t\right\|_{\op}^{-1}\left\|\frac{X_{t}\tran e_{t}(t)}{n_{t}}\right\|_{2}^{2}= O_\P(\cure_{2}\kappa),\label{eq:hatDeltat-tdDeltat2}
\end{align}
where the last line follows from Lemma \ref{lem:vector_concentration}. 
Putting \eqref{eq:hatDeltat-tdDeltat3} and  \eqref{eq:hatDeltat-tdDeltat2} into \eqref{eq:hatDeltat-tdDeltat1}, we obtain  
\begin{equation}
  \label{eq:hatDeltat-tdDeltat}
  |\hat{\Delta}_{t} - \td{\Delta}_{t}| = O_\P\lb \lb\frac{\cure_{2}\kappa^{2} p}{n}\rb^{1/2}\rb.
\end{equation}
Combining \eqref{eq:tdDeltat-Deltat} and \eqref{eq:hatDeltat-tdDeltat} together, we have 
$
|\hat{\Delta}_{t} - \Delta_{t}| = O_\P\lb \lb \cure_{2}\kappa^{2} p / n\rb^{1/2} \rb.
$
We complete the proof by invoking Theorem \ref{thm:expansion}.
\end{proof}

\subsection{Proof of asymptotic normality}

\begin{proof}[Proofs of Theorem \ref{thm:asym_normality_htau} and \ref{thm:asym_normality_hdtau}]
We first prove the asymptotic normality of $\taue $.
Recalling $0 = \one\tran e(0) =  \one \tran e_1(0)  + \one \tran  e_0(0) $, we obtain that
\begin{eqnarray}
n^{1/2} \taue 
=  \frac{n^{1/2}}{n_{1}}\oneone \tran e_{1}(1) + \frac{n^{1/2}}{n_{0}}\oneone \tran e_{1}(0) 
=\sum_{i\in \T_{1}}\lb \frac{n^{1/2}}{n_{1}} e_{i}(1)+ \frac{n^{1/2}}{n_{0}}e_{i}(0)\rb. \label{eq::srsrepresentation}
\end{eqnarray}
Let $w_{i} = \frac{n^{1/2}}{n_{1}} e_{i}(1)+ \frac{ n^{1/2} }{n_{0}}e_{i}(0)$. Based on \eqref{eq:sigman}, we can verify that
\begin{align*}
&S_w^{2} \equiv  \sum_{i=1}^{n}(w_{i} - \bar{w})^{2} = \sum_{i=1}^{n}w_{i}^{2} 
=  n \sum_{i=1}^{n}  \lb \frac{e_{i}(1)}{n_{1}} + \frac{ e_{i}(0) }{n_{0}} \rb^2 
= \frac{n^{2}}{n_{1}n_{0}}\sigma_{n}^{2}.
\end{align*}
Applying Lemma \ref{prop:finite_CLT} to \eqref{eq::srsrepresentation}, we have 
$
d_{\text{K}} (  n^{1/2}\sigma_{n} \taue, N(0, 1) )= O(   \max_{1\leq i \leq n}|w_{i}| / S_w ) .
$
Lemma \ref{lem:A3} and Assumption \ref{assumption::A4} imply 
$
S_w^{-1} = O( \cure_{2}^{-1 / 2} ) 
$
and
$ \max_{1\leq i \leq n}|w_{i}| = O( \cure_{\infty} / n^{1/2 })= o( \cure_{2}^{1 / 2} ). 
$
Therefore, $\taue$ is asymptotically normal because convergence in Kolmogorov distance implies weak convergence.

We then prove the asymptotic normality of the two estimators.
Theorem \ref{cor:expansion} and Lemma \ref{lem:A3} imply
  \begin{align*}
  \frac{n^{1/2}(\htau - \tau)}{\sigma_{n}} 
= &\frac{n^{1/2} \taue }{\sigma_{n}}  + O_\P\lb \frac{(\cure_{2}\kappa^{2} p\log p)^{1/2}}{\sigma_{n}} + \frac{n^{1/2}\Delta}{\sigma_{n}} + \frac{(\cure_{2}\kappa)^{1/2}}{\sigma_{n}}\rb\\
= & \frac{n^{1/2} \taue}{\sigma_{n}}  + O_\P\lb (\kappa^{2} p\log p)^{1/2} + \lb\frac{n}{\cure_{2}}\rb^{1/2}\Delta + \kappa^{1/2}\rb.
  \end{align*}
  We complete the proof  by noting that $\kappa = o(1)$ under Assumption \ref{assumption::A2}. The same proof carries over to $\hdtau$.
\end{proof}

\subsection{Proof of asymptotic conservatism of variance estimators}\label{subsec:proof_var}

\begin{proof}[Proof of Theorem \ref{thm:variance_est}]
First, we prove the result for $j = 0$. Recalling $\hat{e}_{t} = (\id  - H_t )e_{t}(t)$, we have
\begin{align}
\frac{1}{n_{t}}\sum_{i\in \T_{t}}\hat{e}_{i}^{2} = \frac{1}{n_{t}} e_{t}(t) \tran   (\id  - H_t )e_{t}(t) 
=  \frac{1}{n_{t}}\sum_{i\in \T_{t}}e_{i}^2(t)  
 -
 \lb\frac{X_{t}\tran e_{t}(t)}{n_{t}}\rb\tran 
 \Sigma_t^{-1}
 \frac{X_{t}\tran e_{t}(t)}{n_{t}}
 \triangleq S_{t1}   - S_{t2}  .\label{eq:S1-S2}
\end{align}
Lemma \ref{lem:vector_concentration} and Assumption \ref{assumption::A2} together imply a bound for $S_{t2}$:
\begin{equation}
  \label{eq:S2t}
  S_{t2} \leq  \left\|   \Sigma_t^{-1} \right\|_\op  \left\|  \frac{X_{t}\tran e_{t}(t)}{n_{t}} \right\|_2^2    
  = O_\P\lb \cure_{2}\kappa\rb = o_\P\lb \cure_{2} \rb.
\end{equation}
The first term, $S_{t1}$, has mean $\E S_{t1}   = n^{-1}\sum_{i=1}^{n}e_{i}^2(t)$ and variance
\begin{eqnarray}
  \Var(S_{t1}  ) & \le & \frac{1}{n_{t}^{2}}\frac{n_{1}n_{0}}{n(n  - 1)}\sum_{i=1}^{n}e_{i}^4(t) \label{eq::variance-St1} \\
  & \le& \frac{n}{n_{t}^{2}}\cure_{\infty}^{2} \cure_{2} = O\lb\frac{\cure_{\infty}^{2}\cure_{2}}{n}\rb \label{eq:VarS1t} \\
  &=& o_\P(\cure_{2}^{2}), \label{eq::finalboundvarSt1}
\end{eqnarray}
where \eqref{eq::variance-St1} follows from Lemma \ref{lem:var_sum}, \eqref{eq:VarS1t} follows from the definitions of $\cure_{2}$ and $\cure_{\infty}$ and Assumption \ref{assumption::A1}, and \eqref{eq::finalboundvarSt1} follows from Assumption \ref{assumption::A4} that $\cure_{\infty}^{2} = o(n\cure_{2}  )$. Therefore, Chebyshev's inequality implies 
\begin{equation}\label{eq:SMarkov}
S_{t1}   = \E S_{t1}   + O_\P\lb \Var(S_{t1})^{1/2}\rb 
= \frac{1}{n}\sum_{i=1}^{n}e_{i}^2(t) + o_\P(\cure_{2}) . 
\end{equation}
Combining the bounds for $S_{t1} $ in \eqref{eq:SMarkov} and $S_{t2} $ in \eqref{eq:S2t}, we have 
\begin{equation}\label{eq:avgeit}
\frac{1}{n_{t}}\sum_{i\in \T_{t}}\hat{e}_{i}^{2} = \frac{1}{n}\sum_{i=1}^{n}e_{i}^2(t)   + o_\P\lb \cure_{2}\rb.
\end{equation}
Using the formula of $\hat{\sigma}^{2}$ in \eqref{eq:sigman2hat_simple} and Assumption \ref{assumption::A1}, we have
\begin{align*}
  \hat{\sigma}_{\HC 0}^{2} 
&= \frac{n}{n_1 - 1} \lb     \frac{1}{n}\sum_{i=1}^{n}e_{i}^2(1)   + o_\P\lb \cure_{2}\rb \rb 
+ \frac{n}{n_0 - 1} \lb \frac{1}{n}\sum_{i=1}^{n}e_{i}^2(0)   + o_\P\lb \cure_{2}\rb \rb \\
 &  = \frac{1}{n_{1}}\sum_{i=1}^{n} e_i^2(1) + \frac{1}{n_{0}}\sum_{i=1}^{n}e_i^2(0) + o_\P(\cure_{2}).
\end{align*}
Using the formula of $\sigma_n^2$ in \eqref{eq:sigman}, we have
$$
 \hat{\sigma}_{\HC 0}^{2}  \ge \sigma_{n}^{2} + \frac{1}{n}\sum_{i=1}^{n}(e_{i}(1)- e_{i}(0))^{2} + o_\P(\cure_{2}) \ge \sigma_{n}^{2} + o_\P(\cure_{2}),
$$
which, coupled with Lemma \ref{lem:A3}, implies the result on $\hat{\sigma}^{2}_{\HC 0}$.

Next we prove that the $\hat{\sigma}^{2}_{\HC j}$'s are asymptotically equivalent. It suffices to show   
$
\min_{1\leq i \leq n }|\td{e}_{i, j}| / |\hat{e}_{i}| = 1 + o_\P(1)
$
for $j=1,2,3$. 
The proof for $j = 1$ follows from $p/n = o(1)$ in \eqref{eq:pn} and Assumption \ref{assumption::A1}. The proofs for $j=2,3$ follow from
$
\max_{t=0,1} \max_{i\in \T_{t}}H_{t,ii} = o_\P(1),
$
which holds by Lemma \ref{lem:matrix_concentration} and Assumption \ref{assumption::A2}: 
\[
\max_{i\in \T_{t}}H_{t,ii} = \max_{i\in \T_{t}}  n_{t}^{-1}x_{i}\tran \Sigma_t^{-1}x_{i} 
= O_\P\left(  n_t^{-1} \max_{1\leq i \leq n}\|x_{i}\|_{2}^{2} \right) = O_\P\lb\kappa\rb = o_\P(1).
\]
\end{proof}

\section{Concentration Inequalities for Sampling Without Replacement}\label{sec:concentration}

\subsection{Some existing tools}
The proofs rely on concentration inequalities for sampling without replacement. \citet[][Theorem 4]{hoeffding63} proved the following result that sampling without replacement is more concentrated in convex ordering than sampling with replacement. 

\begin{proposition}\label{prop:hoeffding}
Let $C = (c_{1}, \ldots, c_{n})$ be a finite population of fixed elements. Let $Z_{1}, \ldots, Z_{m}$ be a random sample with replacement from $C$ and $W_{1}, \ldots, W_{m}$ be a random sample without replacement from $C$. If the function $f(x)$ is continuous and convex, then
$
\E f\lb\sum_{i=1}^{m}Z_{i}\rb\ge \E f\lb\sum_{i=1}^{m}W_{i}\rb.
$
\end{proposition}

From Proposition \ref{prop:hoeffding}, most concentration inequalities for independent sampling carry over to sampling without replacement. Later a line of works, in different contexts, showed an even more surprising phenomenon that sampling without replacement can have strictly better concentration than independent sampling \citep[e.g.,][]{serfling74, diaconis87, lee98, bobkov04, cortes09, el09, bardenet15, tolstikhin17}. In particular, \citet[][Theorem 9]{tolstikhin17} proved a useful concentration inequality for the empirical processes for sampling without replacement. 

\begin{proposition}\label{prop:emp_process}
Let $C = (c_{1}, \ldots, c_{n})$ be a finite population of fixed elements, and $W_{1}, \ldots, W_{m}$ be a random sample without replacement from $C$. Let $\F$ be a class of functions on $C$, and 
\[S(\F) = \sup_{f\in \F} \sum_{i=1}^{m}f(W_{i}),\quad 
\nu(\F)^{2} = \sup_{f\in \F}\Var(f(W_{1})). \]
Then 
\[\P(S(\F) - \E[S(\F)]\ge t)\le \exps{-\frac{(n + 2)t^{2}}{8n^{2}\nu(\F)^{2}}} . \]
\end{proposition}

Proposition \ref{prop:emp_process} gives a sub-gaussian tail of $S(\F)$ with the sub-gaussian parameter depending solely on the variance. In contrast, the concentration inequalities in the standard empirical process theory for independent sampling usually requires the functions in $\F$ to be uniformly bounded and the tail is either sub-gaussian with the sub-gaussian parameter being the uniform bound on $\F$ or sub-exponential with Bernstein-style behaviors; see \cite{boucheron13} for instance. Therefore, Proposition \ref{prop:emp_process} provides a more precise statement that sampling without replacement is more concentrated than independent sampling for a large class of statistics.

We need the following result from \citet[][Theorem 5.1.(2)]{tropp16} to prove the matrix concentration inequality.
\begin{proposition}\label{prop:norm_mean}
Let $\td{V}_{1}, \ldots, \td{V}_{m}$ be independent $p\times p$ random matrices with $\E \td{V}_{i} = 0$ for all $i$. Let $C(p) = 4(1 + \lceil 2\log p\rceil)$. Then 
\[\lb\E \left\|\sum_{i=1}^{m}\td{V}_{i}\right\|_{\op}^{2}\rb^{\frac{1}{2}}\le C(p)^{1/2} \left\|\sum_{i=1}^{m}\E \td{V}_{i}^{2}\right\|_{\op}^{\frac{1}{2}} + C(p)\lb \E \max_{1\leq i \leq n}\|\td{V}_{i}\|_{\op}^{2}\rb^{\frac{1}{2}}.\]
\end{proposition}

\subsection{Proofs of Lemma \ref{thm:vector_concentration} and \ref{thm:matrix_concentration}}

We will use the facts that for any $u\in \R^{p}$ and Hermitian $V\in \R^{p\times p}$,
\[\|u\|_{2} = \sup_{\omega\in \S^{p-1}}u\tran \omega, \quad \|V\|_{\op} = \sup_{\omega\in \S^{p-1}}|\omega\tran V\omega|.\]

\begin{proof}[Proof of Lemma \ref{thm:vector_concentration}]
Let 
$C = (u_{1}, \ldots, u_{n}) 
$
and
$ \F = \{f_{\omega}(u) = u\tran \omega: \omega\in \S^{p-1}\}.
$
Let $u$ be a vector that is randomly sampled from $C$. Then 
\begin{align*}
  \nu^{2}(\F) &= \sup_{\omega\in \S^{p-1}}\Var(u\tran \omega) \le \sup_{\omega\in \S^{p-1}}\E(u\tran \omega)^{2} \\
& = \sup_{\omega\in \S^{p-1}}  \frac{1}{n}\sum_{i=1}^{n}(u_{i}\tran \omega)^{2} 
= \sup_{\omega\in \S^{p-1}}  \omega\tran \lb\frac{1}{n}\sum_{i=1}^{n}u_{i}u_{i}\tran \rb\omega 
  = \sup_{\omega\in \S^{p-1}} \omega\tran \lb\frac{U\tran U}{n}\rb\omega \le \frac{\|U\|_{\op}^{2}}{n}.
\end{align*}
By Proposition \ref{prop:emp_process}, 
\begin{align*}
  \P\lb \left\|\sum_{i\in \T}u_{i}\right\|_{2} \ge \E\left\|\sum_{i\in \T}u_{i}\right\|_{2} + t \rb\le \exps{-\frac{(n + 2)t^{2}}{8n \|U\|_{\op}^{2}}}\le \exps{-\frac{t^{2}}{8\|U\|_{\op}^{2}}},
\end{align*}
or, equivalently, with probability $1 - \delta$,
\begin{equation}
  \label{eq:vector1}
  \left\|\sum_{i\in \T}u_{i}\right\|_{2} \le \E\left\|\sum_{i\in \T}u_{i}\right\|_{2} + \|U\|_{\op}\, \lb 8\log \frac{1}{\delta}\rb^{1/2}.
\end{equation}
By the Cauchy--Schwarz inequality, 
\[\lb\E\left\|\sum_{i\in \T}u_{i}\right\|_{2}\rb^{2}\le \E\left\|\sum_{i\in \T}u_{i}\right\|_{2}^{2} = \sum_{j=1}^{p}\E\lb\sum_{i\in \T}u_{ij}\rb^{2}.\]
Lemma \ref{lem:var_sum} implies 
\[\E\lb\sum_{i\in \T}u_{ij}\rb^{2} = \frac{m(n - m)}{n(n - 1)}\sum_{i=1}^{n}u_{ij}^{2}.\]
As a result,
\begin{equation}\label{eq:vector2}
\lb\E\left\|\sum_{i\in \T}u_{i}\right\|_{2}\rb^{2}\le \frac{m(n - m)}{n(n - 1)}\sum_{i=1}^{n}\|u_{i}\|_{2}^{2} = \|U\|_{F}^{2}\,\frac{m(n - m)}{n(n - 1)}.
\end{equation}
We complete the proof by using \eqref{eq:vector1} and \eqref{eq:vector2}.
\end{proof}

\begin{proof}[Proof of Lemma \ref{thm:matrix_concentration}]

Let $V$ be a matrix that is randomly sampled from a set of matrices $C = (V_{1}, \ldots, V_{n})$ and $ \F = \{f_{\omega}(V) = \omega\tran V\omega: \omega\in \S^{p-1}\}$. Then
$$
  \nu^{2}(\F) = \sup_{\omega\in \S^{p-1}}\Var(\omega\tran V\omega) \le \sup_{\omega\in \S^{p-1}}\E(\omega\tran V\omega)^{2} 
 =  \sup_{\omega\in \S^{p-1}} \frac{1}{n}\sum_{i=1}^{n}(\omega\tran V_{i}\omega)^{2} = \nu^{2}_{-}.
$$ 
By Proposition \ref{prop:emp_process}, 
\begin{align*}
  \P\lb \sup_{\omega\in \S^{p-1}}\omega\tran \lb\sum_{i\in\T_{t}}V_{i}\rb\omega \ge \E\left[ \sup_{\omega\in \S^{p-1}}\omega\tran \lb\sum_{i\in\T_{t}}V_{i}\rb\omega \right] + t \rb\le \exps{-\frac{(n + 2)t^{2}}{8n^{2}\nu^{2}_{-}}}\le \exps{-\frac{t^{2}}{8n \nu^{2}_{-}}}.
\end{align*}
Since $\sup_{\omega\in \S^{p-1}}\omega\tran V \omega\le \|V\|_{\op}$, it implies
\begin{align*}
  \P\lb \sup_{\omega\in \S^{p-1}}\omega\tran \lb\sum_{i\in\T_{t}}V_{i}\rb\omega \ge \E\left[\left\|\sum_{i\in\T_{t}}V_{i}\right\|_{\op}\right] + t \rb\le \exps{-\frac{(n + 2)t^{2}}{8n^{2}\nu^{2}_{-}}}\le \exps{-\frac{t^{2}}{8n \nu^{2}_{-}}}.
\end{align*}
Similarly,
\begin{align*}
  \P\lb -\sup_{\omega\in \S^{p-1}}\omega\tran \lb\sum_{i\in\T_{t}}V_{i}\rb\omega \ge \E\left[\left\|\sum_{i\in\T_{t}}V_{i}\right\|_{\op}\right] + t \rb\le \exps{-\frac{t^{2}}{8n \nu^{2}_{-}}}.
\end{align*}
As a consequence,
\begin{align*}
  \P\lb  \left\|\sum_{i\in\T_{t}}V_{i}\right\|_{\op} \ge \E\left[\left\|\sum_{i\in\T_{t}}V_{i}\right\|_{\op}\right] + t \rb\le 2\exps{-\frac{t^{2}}{8n \nu^{2}_{-}}}.
\end{align*}
or, equivalently,  with probability $1 - \delta$,
\begin{equation}
  \label{eq:matrix1}
  \left\|\sum_{i\in \T}V_{i}\right\|_{\op} \le \E\left\|\sum_{i\in \T}V_{i}\right\|_{\op} + \lb 8n\log \frac{2}{\delta}\rb^{1/2}\,\nu_{-}.
\end{equation}

We then bound $\E\left\|\sum_{i\in \T}V_{i}\right\|_{\op} .$
Let $\td{V}_{1}, \ldots, \td{V}_{m}$ be an i.i.d. random sample with replacement from $C$. We have
\begin{align}
  &\E\left\|\sum_{i\in \T}V_{i}\right\|_{\op} \le \E\left\|\sum_{i=1}^{m}\td{V}_{i}\right\|_{\op}
  \le \ \lb  \E\left\|\sum_{i=1}^{m}\td{V}_{i}\right\|_{\op}^{2}\rb^{\frac{1}{2}} \le (nC(p))^{1/2}\nu + C(p)\nu_{+},\label{eq:matrix2}
\end{align}
where the first inequality follows from Proposition \ref{prop:hoeffding} due to the convexity of $\|\cdot\|_{\op}$, the second inequality follows from the Cauchy--Schwarz inequality, and the third inequality follows from Proposition \ref{prop:norm_mean}.

Combining \eqref{eq:matrix1} and \eqref{eq:matrix2}, we complete the proof.
\end{proof}

\section{ Proof of Lemma \ref{thm:mean}}
\label{app:mean_variance}

When $m = 0$ or $m = n$, $Q_{\T}$ is deterministic with zero variance and the inequality holds automatically. Thus we assume $1\le m\le n-1$.
Let $\sum_{[i_1, \ldots, i_k]}$ denote the sum over all $(i_1, \ldots, i_k)$ with mutually distinct elements in $\{1,\ldots, n\}$. For instance, $\sum_{[i, j]}$ denotes the sum over all pairs $(i, j)$ with $i\not = j$. 
We first state a basic result for sampling without replacement.

\begin{lemma}\label{lem:trivial_comb}
  Let $i_{1}, \ldots, i_{k}$ be distinct indices in $\{1,\ldots, n  \}$ and $\T$ be a uniformly random subset of $\{ 1,\ldots, n\}$ with size $m$. Then
\[\P\lb i_{1}, \ldots, i_{k}\in \T\rb = \frac{m\cdots (m - k + 1)}{n\cdots (n - k + 1)}.\]
\end{lemma}

  By definition,
\begin{equation}\label{eq:QT}
Q_{\T} = \sum_{i=1}^{n}Q_{ii}I(i\in \T) + \sum_{[i, j]} Q_{ij}I(i,j\in \T).
\end{equation}
The mean of $Q_\T$ follows directly from Lemma \ref{lem:trivial_comb}:  
\begin{align*}
\E Q_{\T} = \sum_{i=1}^{n}Q_{ii}\cdot \frac{m}{n} + \sum_{[i, j]} Q_{ij}\cdot \frac{m(m-1)}{n(n-1)} 
  = \frac{m(n - m)}{n(n - 1)}\tr(Q) + \frac{m(m - 1)}{n(n - 1)}(\one\tran Q\one).
\end{align*}

The rest of this section proves the result of the variance. 
Let 
\begin{align*}
&  c_1 = \frac{m(n - m)}{n(n - 1)}, \quad  c_2 = \Var\lb I(1, 2\in \T)\rb = c_1 \frac{(m-1)(n + m - 1)}{n(n-1)},\\
&  c_3 = \Cov\lb I(1, 2\in \T), I(1, 3\in \T)\rb = c_1\frac{(m - 1)(mn - 2m - 2n + 2)}{n(n-1)(n-2)},\\
  &  c_4 = \Cov\lb I(1, 2\in \T), I(3, 4\in \T)\rb = c_1\frac{(m - 1)(-4mn + 6n + 6m - 6)}{n(n-1)(n - 2)(n - 3)},\\
& c_5 = \Cov\lb I(1 \in \T), I(1, 2\in \T)\rb = c_1\frac{m - 1}{n},\\
& c_6 = \Cov\lb I(1 \in \T), I(2, 3\in \T)\rb = -c_1\frac{2(m - 1)}{n(n - 2)}.
\end{align*}
Using \eqref{eq:QT}, we have 
\begin{align}
\Var(Q_{\T})  &=  \underbrace{  \Var\lb \sum_{i=1}^{n}Q_{ii}I(i\in \T)\rb }_{V_\textup{I}}
+ \underbrace{  \Var\lb \sum_{[i, j]} Q_{ij}I(i, j\in \T)\rb}_{V_\textup{II}} \nonumber\\
 & \quad + 2  \underbrace{  \Cov\lb \sum_{i=1}^{n}Q_{ii}I(i\in \T), \sum_{[i, j]} Q_{ij}I(i, j\in \T)   \rb  }_{V_\textup{III}}.  
 \label{eq:var_decompose}
\end{align}

The next subsection deals with the three terms in \eqref{eq:var_decompose}, separately. 

\subsection{Simplifying  \eqref{eq:var_decompose}}

\paragraph{Term $V_\textup{I}$}
Lemma \ref{lem:var_sum} implies
\begin{align}
V_\textup{I}=  \Var\lb \sum_{i=1}^{n}Q_{ii}I(i\in \T)\rb  = \frac{m(n - m)}{n(n-1)}\sum_{i=1}^{n}\lb Q_{ii} - \frac{1}{n}\sum_{i=1}^{n}Q_{ii}\rb^{2} 
  = c_1\sum_{i=1}^{n}Q_{ii}^{2} - \frac{c_1}{n}(\tr(Q))^{2}. \label{eq:var_decompose1}
\end{align}

\paragraph{Term $V_\textup{II}$}
We expand $V_\textup{II}$ as
\begin{align*}
  &V_\textup{II} = \Var\lb \sum_{[i, j]}Q_{ij}I(i, j\in \T)\rb = \Cov\lb \sum_{[i, j]}Q_{ij}I(i, j\in \T), \sum_{[i', j']}Q_{i'j'}I(i', j'\in \T)\rb \\
&  =\sum_{[i, j]} \lb Q_{ij}^{2} + Q_{ij}Q_{ji}\rb\Var(I(i, j\in \T)) + \sum_{[i, j, k, \ell]}Q_{ij}Q_{k\ell} \Cov\lb I(i, j\in \T), I(k, \ell\in \T)\rb  \\
& \quad + \sum_{[i, j, k]}\lb Q_{ij}Q_{ik} + Q_{ij}Q_{ki}\rb \Cov\lb I(i, j\in \T), I(i, k\in \T)\rb  \\
& \quad + \sum_{[i, j, k]}\lb Q_{ij}Q_{jk} + Q_{ij}Q_{kj}\rb \Cov\lb I(i, j\in \T), I(j, k\in \T)\rb  \\
 &   =  c_2 \sum_{[i, j]}\lb Q_{ij}^{2} + Q_{ij}Q_{ji}\rb + c_4 \sum_{[i, j, k, \ell]}Q_{ij}Q_{k\ell} 
 + c_3\sum_{[i, j, k]}\lb Q_{ij}Q_{ik} + Q_{ij}Q_{ki} + Q_{ij}Q_{jk} + Q_{ij}Q_{kj}\rb . 
\end{align*}
We then reduce the summation over $[i,j,k,l]$ to summations over fewer indices.
First, 
\begin{align*}
  \lb \sum_{[i, j]} Q_{ij}\rb^{2} = \sum_{[i, j]}\lb Q_{ij}^{2} + Q_{ij}Q_{ji}\rb + \sum_{[i, j, k, \ell]}Q_{ij}Q_{k\ell} 
 + \sum_{[i, j, k]}\lb Q_{ij}Q_{ik} + Q_{ij}Q_{ki} + Q_{ij}Q_{jk} + Q_{ij}Q_{kj}\rb.
\end{align*}
Second, $\one\tran Q\one = 0$ implies
$\sum_{[i, j]} Q_{ij} = -\sum_{i=1}^{n}Q_{ii} = -\tr(Q),$
which further implies   
\begin{align*}
\sum_{[i, j, k, \ell]}Q_{ij}Q_{k\ell}  = \lb\tr(Q)\rb^{2} - \sum_{[i, j]}\lb Q_{ij}^{2} + Q_{ij}Q_{ji}\rb 
 - \sum_{[i, j, k]}\lb Q_{ij}Q_{ik} + Q_{ij}Q_{ki} + Q_{ij}Q_{jk} + Q_{ij}Q_{kj}\rb.
\end{align*}
The above two facts simplify $V_{\textup{II}}$ to
\begin{align}
V_\textup{II} &=   c_{4}(\tr(Q))^{2} + (c_2 - c_4)\sum_{[i, j]}\lb Q_{ij}^{2} + Q_{ij}Q_{ji}\rb  \nonumber\\
& + (c_3 - c_4)\sum_{[i, j, k]}\lb Q_{ij}Q_{ik} + Q_{ij}Q_{ki} + Q_{ij}Q_{jk} + Q_{ij}Q_{kj}\rb\label{eq:var_decompose2_3}.
\end{align}
We then reduce the summation over $[i,j,k]$ to summations over fewer indices.
Note that $\one\tran Q = Q\one = 0$ implies 
$  \sum_{j=1}^{n}Q_{ij} = \sum_{i=1}^{n}Q_{ij} = 0,$
which further implies 
\begin{align*}
  \sum_{[i, j, k]} Q_{ij}Q_{ik} &= \sum_{[i, j]}Q_{ij}\sum_{k\not= i,j}Q_{ik} = -\sum_{[i, j]}Q_{ij}(Q_{ii} + Q_{ij}) \\
  & = -\sum_{i=1}^nQ_{ii}\sum_{j\not = i}Q_{ij} - \sum_{[i, j]}Q_{ij}^{2} = \sum_{i=1}^{n}Q_{ii}^{2} - \sum_{[i, j]}Q_{ij}^{2}.
\end{align*}
Similarly, 
\[\sum_{[i, j, k]} Q_{ij}Q_{kj} = \sum_{i=1}^{n}Q_{ii}^{2} - \sum_{[i, j]}Q_{ij}^{2},\quad 
 \sum_{[i, j, k]} Q_{ij}Q_{ki} = \sum_{[i, j, k]} Q_{ij}Q_{jk} = \sum_{i=1}^{n}Q_{ii}^{2} - \sum_{[i, j]}Q_{ij}Q_{ji}.\]
Using the above three identities to simplify the third term in \eqref{eq:var_decompose2_3}, we obtain
\begin{align}
V_\textup{II}  =  c_{4}(\tr(Q))^{2} + 4(c_3 - c_4) \sum_{i=1}^n Q_{ii}^{2} 
 + (c_2 - 2c_3 + c_4)\sum_{[i, j]}\lb Q_{ij}^{2} + Q_{ij}Q_{ji}\rb . \label{eq:var_decompose2_4}
\end{align}

\paragraph{Term $V_\textup{III}$}
The covariance term is
\begin{eqnarray*}
V_\textup{III} &=   &\Cov\lb \sum_{i=1}^{n}Q_{ii}I(i\in \T), \sum_{[i, j]} Q_{ij}I(i, j\in \T)\rb \\
 & =& \sum_{[i, j]}Q_{ii}(Q_{ij} + Q_{ji})\Cov\lb I(i\in \T), I(i, j\in \T)\rb  
 + \sum_{[i, j, k]}Q_{ii}Q_{jk}\Cov\lb I(i\in \T), I(j, k\in \T)\rb \\
&  =& c_5 \sum_{[i, j]}Q_{ii}(Q_{ij} + Q_{ji}) + c_6 \sum_{[i, j, k]}Q_{ii}Q_{jk}. 
\end{eqnarray*}
Similar to previous arguments,
\begin{eqnarray*}
  \sum_{[i, j]}Q_{ii}(Q_{ij} + Q_{ji}) & = & \sum_{i=1}^nQ_{ii}\sum_{j\not = i}(Q_{ij} + Q_{ji}) = -2\sum_{i=1}^nQ_{ii}^{2} ,\\
  \sum_{[i, j, k]}Q_{ii}Q_{jk} & =& \sum_{[i, j]}Q_{ii}\sum_{k\not = i,j}Q_{jk} = -\sum_{[i, j]}Q_{ii}(Q_{jj} + Q_{ji})\\
& =& -\sum_{i=1}^nQ_{ii} \sum_{j\not = i}(Q_{jj} + Q_{ji}) = -\sum_{i=1}^nQ_{ii} \lb \tr(Q) - Q_{ii} - Q_{ii}\rb \\
& = &-(\tr(Q))^{2} + 2\sum_{i=1}^nQ_{ii}^{2}.
\end{eqnarray*}
Using the above two identities, we can simplify $V_\textup{III}$ to
\begin{align}
 V_\textup{III}   = -c_6 (\tr(Q))^{2} - 2(c_5 - c_6)\sum_{i=1}^nQ_{ii}^{2}.  \label{eq:var_decompose3_2}
\end{align}

Putting \eqref{eq:var_decompose1}, \eqref{eq:var_decompose2_4} and \eqref{eq:var_decompose3_2} together, we obtain that 
\begin{align}
\Var(Q_\T)  &=   \underbrace{  (c_1 + 4c_3 - 4c_4 - 4c_5 + 4c_6)}_{C_\textup{I}}  \sum_{i=1}^{n}Q_{ii}^{2}  
+  \underbrace{   \lb c_4 - \frac{c_1}{n} - 2c_6\rb  }_{C_\textup{II}} (\tr(Q))^{2} \nonumber\\
 & + \underbrace{   (c_2 - 2c_3 + c_4)}_{C_\textup{III}}  \sum_{[i, j]}(Q_{ij}^{2} + Q_{ij}Q_{ji}).\label{eq:exact_var}
\end{align}
We simplify \eqref{eq:exact_var} in the next subsection by deriving bounds for the coefficients.

\subsection{Bounding the coefficients $C_\textup{I}, C_\textup{II}$ and $C_\textup{III}$ in \eqref{eq:exact_var}}
\paragraph{Bounding $C_\textup{I}$}
We have 
\begin{align}
C_\textup{I} &= c_1 + 4c_3 - 4c_4 - 4c_5 + 4c_6 \nonumber\\
 & =  c_1 + 4c_1\frac{m - 1}{n}\lb \frac{mn - 2m - 2n + 2}{(n - 1)(n - 2)} + \frac{4mn - 6m - 6n + 6}{(n - 1)(n - 2)(n - 3)} - 1 - \frac{2}{n - 2}\rb.\nonumber
\end{align}
Through a tedious calculation, we obtain that
\begin{align*}
  &\frac{mn - 2m - 2n + 2}{(n - 1)(n - 2)} + \frac{4mn - 6m - 6n + 6}{(n - 1)(n - 2)(n - 3)} - 1 - \frac{2}{n - 2} = -\frac{(n - m - 1)n}{(n - 2)(n - 3)}.
\end{align*}
Thus,
$
  C_\textup{I} = c_1\lb 1 - \frac{4(m - 1)(n - m - 1)}{(n - 2)(n - 3)}\rb.
$

\paragraph{Bounding $C_\textup{II}$}
We have 
\begin{align*}
C_\textup{II} = & \,\, c_4 - \frac{c_1}{n} - 2c_6 = -\frac{c_1}{n} + c_1\frac{m - 1}{n(n - 2)}\lb \frac{-4mn + 6m + 6n - 6}{(n - 1)(n - 3)} + 4\rb\\
= & -\frac{c_1}{n} + c_1\frac{(m - 1)(4n^2 - 4mn + 6m - 10n + 6)}{n(n - 1)(n - 2)(n - 3)}\\
= & -\frac{c_1}{n}\lb 1 - \frac{(m - 1)(n - m - 1)(4n - 6)}{(n - 1)(n - 2)(n - 3)}\rb\\
\le & \,\, c_1 \frac{(m - 1)(n - m - 1)(4n - 6)}{n(n - 1)(n - 2)(n - 3)} \le  \,\, \frac{c_1}{n} \frac{4(m - 1)(n - m - 1)}{(n - 2)(n - 3)}.
\end{align*}





\paragraph{Bounding $C_\textup{III}$}
We consider four cases.

\noindent $\bullet$ If $m = 1$, then $c_2 = c_3 = c_4 = 0$ and
$C_{\textup{III}}\le \frac{c_{1}}{2}.$

\noindent $\bullet$ If $m = 2$, since $n\ge 4$, we have
\begin{align*}
  C_{\textup{III}} &= c_1\lb\frac{n + 1}{n(n - 1)} - \frac{-4}{n(n - 1)(n - 2)} - \frac{2}{n(n - 1)(n - 2)}\rb \\
& = c_1\lb \frac{n + 1}{n(n - 1)} + \frac{2}{n(n - 1)(n - 2)}\rb \le \frac{c_1}{2}.
\end{align*}

\noindent $\bullet$ If $m = 3$, then
\begin{align*}
  C_{\textup{III}} &= c_1\lb\frac{2(n + 2)}{n(n - 1)} - \frac{4(n - 4)}{n(n - 1)(n - 2)} - \frac{12n - 24}{n(n - 1)(n - 2)(n - 3)}\rb\\
& = c_1\lb\frac{2(n + 2)}{n(n - 1)} - \frac{4(n - 4)}{n(n - 1)(n - 2)} - \frac{12}{n(n - 1)(n - 3)}\rb.
\end{align*}
If $n\ge 7$,
$C_{\textup{III}}\le c_1\frac{2(n + 2)}{n (n - 1)}\le \frac{c_1}{2}.$
For $n = 4, 5, 6$, we can also verify that
$C_{\textup{III}}\le\frac{c_1}{2}$.

\noindent $\bullet$ If $m \ge 4$, then 
\[4mn - 6m - 6n + 6 = (2m - 6)(n - 3) + 2(mn - 6)\ge (2m + 2)(n - 3).\]
and thus
\[c_{4}\le -c_{1}\frac{(2m + 2)(m - 1)}{n(n - 1)(n - 2)}.\]
Then we have
\begin{align}
  C_{\textup{III}} & \le c_{1}\frac{m - 1}{n(n - 1)}\lb n + m - 1 - \frac{2(mn - 2m - 2n + 2)}{n - 2} - \lihua{\frac{2m + 2}{n - 2}}\rb\nonumber\\
& = c_{1}\frac{m - 1}{n(n - 1)}\lb n + m - 1 - \lihua{\frac{2mn - 4n - 2m + 6}{n - 2}}\rb\nonumber\\
& = c_{1}\frac{m - 1}{n(n - 1)}\lb n - m + 3 - \lihua{\frac{2m - 2}{n - 2}}\rb\nonumber\\
& = c_{1}\lb \frac{(m - 1)(n - m + 3)}{n(n - 1)} - \frac{2(m - 1)^{2}}{n(n - 1)(n - 2)}\rb \nonumber\\
& \le c_{1}\lb \frac{(n + 2)^{2}}{4n(n - 1)} - \frac{2(m - 1)^{2}}{n(n - 1)(n - 2)}\rb \nonumber\\
& \le c_{1}\lb \frac{(n + 2)^{2}}{4n(n - 1)} - \frac{18}{n(n - 1)(n - 2)}\rb.
\end{align}
If $n\ge 7$,
$C_{\textup{III}}\le c_{1}\frac{(n + 2)^{2}}{4n(n - 1)}\le \frac{81c_{1}}{168}\le \frac{c_1}{2}.$
For $n = 4, 5, 6$, we can also verify that 
$C_{\textup{III}}\le \frac{c_1}{2}.$

Therefore, we always have 
$
  C_{\textup{III}}\le \frac{c_{1}}{2}.
$

\medskip 
Using the above bounds for $(C_{\textup{I}}, C_{\textup{II}}, C_{\textup{III}})$ in \eqref{eq:exact_var}, we obtain that
\begin{align*}
\Var\lb Q_\T \rb &\le c_1\lb 1 - \frac{4(m - 1)(n - m - 1)}{(n - 2)(n - 3)}\rb\sum_{i=1}^{n}Q_{ii}^{2} 
  \\
  & \quad + c_1\frac{4(m - 1)(n - m - 1)}{(n - 2)(n - 3)} \frac{(\tr(Q))^{2}}{n} + \frac{c_1}{2}\sum_{[i, j]}(Q_{ij}^{2} + Q_{ij}Q_{ji}).
\end{align*}
Because $(\tr(Q))^{2} \le n\sum_{i=1}^{n}Q_{ii}^{2}$ and $Q_{ij}Q_{ji}\le  (Q_{ij}^{2} + Q_{ji}^{2} ) / 2$, we conclude that 
$\Var\lb Q_\T\rb \le c_1 \|Q\|_{F}^{2}.$

\section{Proofs of other lemmas in Section \ref{subsec:lemmas2}}\label{subapp:lemmas}
 
\begin{proof}[Proof of Lemma \ref{lem:A3}]
  Using the definitions of $ \sigma_{n}^{2}$ and $\rho_{e}$, we have
  \begin{align*}
    \sigma_{n}^{2} &= \lb\frac{1}{n_{1}} - \frac{1}{n}\rb\sum_{i=1}^{n}e_i^2(1) + \lb\frac{1}{n_{0}} - \frac{1}{n}\rb\sum_{i=1}^{n}e_i^2(0) + \frac{2}{n}\sum_{i=1}^{n}e_{i}(1)   e_{i}(0)\\
& = \frac{n_{0}}{n_{1}n}\sum_{i=1}^{n}e_i^2(1) + \frac{n_{1}}{n_{0}n}\sum_{i=1}^{n}e_i^2(0) + \frac{2\rho_{e}}{n}\lb\sum_{i=1}^{n}e_i^2(1)\rb^{1/2}\lb\sum_{i=1}^{n}e_i^2(0)\rb^{1/2}.
  \end{align*}
If $\rho_{e} \ge 0$, then 
\[\sigma_{n}^{2}\ge \frac{n_{0}}{n_{1}n}\sum_{i=1}^{n}e_i^2(1) + \frac{n_{1}}{n_{0}n}\sum_{i=1}^{n}e_i^2(0)\ge \min\left\{\frac{n_{1}}{n_{0}}, \frac{n_{0}}{n_{1}}\right\}\cure_{2}.\]
If $\rho_{e} < 0$, then using the fact  
\[
\lb   a\lb\frac{n_0}{n_1}\rb^{1/2} - b\lb\frac{n_1}{n_0}\rb^{1/2}   \rb^2    \geq 0 \Longleftrightarrow
2ab \le \frac{n_{0}}{n_{1}}a^{2} + \frac{n_{1}}{n_{0}}b^{2},
\]
we obtain that
\begin{align*}
  \sigma_{n}^{2} &\ge (1 + \rho_{e})\lb\frac{n_{0}}{n_{1}n}\sum_{i=1}^{n}e_i^2(1) + \frac{n_{1}}{n_{0}n}\sum_{i=1}^{n}e_i^2(0)\rb \ge \eta\min\left\{\frac{n_{1}}{n_{0}}, \frac{n_{0}}{n_{1}}\right\}\cure_{2}.
\end{align*}
Putting the pieces together, we complete the proof.
\end{proof}

\begin{proof}[Proof of Lemma \ref{lem:tauhat}]   
Recall that $\hat{\mu}_{t}$ is the intercept from the  ordinary least squares fit of $Y_{t}^{\mathrm{obs}}$ on $\onet $ and $X_{t}.$ From the Frisch--Waugh--Lovell Theorem \citep{ding2020frisch}, it is identical to the coefficient of the  ordinary least squares fit of the residual $ (\id - H_t ) Y_{t}^{\mathrm{obs}}$ on the residual $ (\id - H_t ) \onet $, after projecting onto $X_t$:  
\[\hat{\mu}_{t} 
= \frac{  \onet \tran  (\id - H_t ) \tran     (\id - H_t ) Y_{t}^{\mathrm{obs}}    }{  \|  (\id - H_t ) \onet  \|_2^2  }
= \frac{\onet \tran  (\id - H_t ) Y_{t}^{\mathrm{obs}} }{\onet \tran  (\id - H_t ) \onet }.\]
Using the definition \eqref{eq:et} and the fact that $ (\id - H_t ) X_{t} = 0$, we have 
\begin{eqnarray*}
(\id - H_t ) Y_{t}^{\mathrm{obs}} =  (\id - H_t ) (\mu_{t}\onet  + X_{t}\beta_{t} + e_{t}(t)) = \mu_{t} (\id - H_t ) \onet  +  (\id - H_t ) e_{t}(t),\\
\Longrightarrow
 \hat{\mu}_{t} = \mu_{t} + \frac{\onet \tran  (\id - H_t )   e_t(t)  }{\onet \tran  (\id - H_t ) \onet } = \mu_{t} + \frac{\onet \tran e_{t}(t) / n_{t} - \onet \tran H_t e_{t}(t) / n_{t}}{1 - \onet \tran H_t \onet  / n_{t}}.
\end{eqnarray*}  
Recalling that $\tau = \mu_1 - \mu_0$, we complete the proof. 
\end{proof}

\begin{proof}[Proof of Lemma \ref{lem:vector_concentration}]
Because $\|  U \|_\op \leq \|  U \|_F$, Lemma \ref{thm:vector_concentration} implies that with probability $1-\delta$,
\[\left\|\sum_{i\in \T}u_{i}\right\|_{2}  \Big /  \|U\|_{F} \le \lb\frac{m(n - m)}{n(n - 1)}\rb^{1/2} +  \lb 8\log \frac{1}{\delta}\rb^{1/2},
\]
which further implies
$
\left\|\sum_{i\in \T}u_{i}\right\|_{2} = O_\P \lb  \|U\|_{F}  \rb.
$

Let $u_{i} = e_{i}(t)$ with $\sum_{i=1}^{n} u_{i} = 0$, $U = (u_{1}, \ldots,u_{n})\tran \in \R^{n\times 1}$, and
$ \|U\|_{F}^{2} = \sum_{i=1}^{n}u_{i}^{2} =\sum_{i=1}^{n} e_{i}^2(t) .$
Therefore,  
\[ \onet\tran e_{t}(t)   = \left\|   \sum_{i\in \mathcal{T}_t} u_i \right\|_{2} =O_\P( \|  U \|_F)
= O_\P\lb    \lb\sum_{i=1}^{n} e_{i}^2(t)\rb^{1/2} \rb = O_\P\lb (n\cure_{2})^{1/2}\rb.\]

Let $u_{i} = x_{i}$ with $\sum_{i=1}^{n} u_{i} = 0$, $U=X$, and $\| U \|_F =  \|X\|_{F}  = (\tr(X\tran X))^{1/2} = \tr(n \id ) = np$. Therefore,
$$
\|X_{t}\tran \onet \|_{2} = \left \|   \sum_{i\in \mathcal{T}_t}  u_i  \right  \|_2   = O_\P\lb \|U\|_{F}\rb = O_\P\lb (np)^{1/2}\rb.
$$

Let $u_{i} = x_{i}e_{i}(t)$ with $\sum_{i=1}^{n} u_{i} = 0$ due to \eqref{eq:error_key_property}. Therefore,
\[\|X_{t}\tran e_{t}(t)\|_{2}  
= \left \|   \sum_{i\in \mathcal{T}_t}  u_i  \right  \|_2   
= O_\P\lb \lb\sum_{i=1}^{n}\|x_{i}\|^{2}e_{i}^2(t)\rb^{1/2}\rb.\]
Recalling \eqref{eq:hatmatrix} that
$  \|x_{i}\|^{2}_2 = n H_{ii}\le n\kappa,$
we have 
$\|X_{t}\tran e_{t}(t)\|_{2} = O_\P\lb n(\cure_{2}\kappa)^{1/2}\rb.$

\end{proof}

We need the following proposition to prove Lemma \ref{lem:matrix_concentration}.
\begin{proposition}\label{prop:matrix_inv_diff}
$A$ and $ B$ are two symmetric matrices. $A$ is positive definite,  and $A + B$ is invertible. Then
\[\|(A + B)^{-1} - A^{-1}\|_{\op}\le \frac{\|A^{-1}\|_{\op}^{2}\cdot \|B\|_{\op}}{1 - \min\{1, \|A^{-1}\|_{\op}\cdot \|B\|_{\op}\}}.\]
\end{proposition}

\begin{proof}[Proof of Proposition \ref{prop:matrix_inv_diff}]
Let $M = A^{-\frac{1}{2}}BA^{-\frac{1}{2}}$ and $\Lambda(M)$ be its spectrum. By definition, $\|M\|_{\op}\le \|A^{-1}\|_{\op}\cdot \|B\|_{\op}.$ If $\|A^{-1}\|_{\op}\cdot \|B\|_{\op} \ge 1$, the inequality is trivial because the right-hand side of it is $\infty$. Without loss of generality, we assume $\|A^{-1}\|_{\op}\cdot \|B\|_{\op} < 1$, which implies $\|M\|_{\op} < 1.$
Proposition \ref{prop:matrix_inv_diff} follows by combining the following two results: 
  \begin{align*}
\|(A + B)^{-1} - A^{-1}\|_{\op} & = \|A^{-\frac{1}{2}}((\id  + M)^{-1} - \id )A^{-\frac{1}{2}}\|_{\op} 
  \le \|A^{-1}\|_{\op}\cdot \|\id  - (\id  + M)^{-1}\|_{\op}  , \\
  \|\id  - (\id  + M)^{-1}\|_{\op} &\le \sup_{\lambda\in \Lambda(M)}\bigg|\frac{\lambda}{1 + \lambda}\bigg| =  \frac{\|M\|_{\op}}{1 - \|M\|_{\op}} \leq \frac{\|A^{-1}\|_{\op} \cdot \|B\|_{\op}}{1 - \|A^{-1}\|_{\op}\cdot \|B\|_{\op}} .
  \end{align*}
\end{proof}

\begin{proof}[Proof of Lemma \ref{lem:matrix_concentration}]
  Let $V_{i} = x_{i}x_{i}\tran  - \id $, then $\sum_{i=1}^{n}V_{i} = 0$. In the following, we will repeatedly use the basic facts: $n^{-1} X\tran X = \id$, $\| x_i\|_2^2 = nH_{ii}$, and $\sum_{i=1}^{n}x_{i}x_{i}\tran = XX\tran = nH$.
  Recalling the definitions of $\nu, \nu_{+}$ and $\nu_{-}$ in Lemma \ref{thm:matrix_concentration}, we have
\begin{align*}
   \nu^{2} &= \left\|\frac{1}{n}\sum_{i=1}^{n}V_{i}^{2}\right\|_{\op} 
   = \left\|\frac{1}{n}\sum_{i=1}^{n}\lb \|x_{i}\|_{2}^{2} x_{i}x_{i}\tran  - 2x_{i}x_{i}\tran  + \id \rb \right\|_{\op} 
  = \left\|\lb\frac{1}{n}\sum_{i=1}^{n}\|x_{i}\|_{2}^{2} x_{i}x_{i}\tran \rb - \id \right\|_{\op}  \\
&  = \left\|\lb \sum_{i=1}^{n}H_{ii} x_{i}x_{i}\tran \rb - \id \right\|_{\op}
  \le \left\|\sum_{i=1}^{n}H_{ii} x_{i}x_{i}\tran \right\|_{\op} + 1 \le \kappa \left\|\sum_{i=1}^{n}x_{i}x_{i}\tran \right\|_{\op} + 1  
= n \kappa  \| H\|_\op + 1 = n \kappa + 1,\\
  \nu_{+} &= \max_{1\leq i \leq n}\|x_{i}x_{i}\tran  - \id \|_{\op} \le \max_{1\leq i \leq n}\|x_{i}\|_{2}^{2} + 1 
  =  n\max_{1\leq i \leq n} H_{ii}  +1  = n\kappa + 1,\\
  \nu_{-}^{2} & = \sup_{\omega\in \S^{p-1}}\frac{1}{n}\sum_{i=1}^{n}(\omega\tran V_{i}\omega)^{2} = \sup_{\omega\in \S^{p-1}}\frac{1}{n}\sum_{i=1}^{n}((x_{i}\tran \omega)^{2} - 1)^{2}
 = \sup_{\omega\in \S^{p-1}}\frac{1}{n}\sum_{i=1}^{n}\left[(x_{i}\tran \omega)^{4} - 2(x_{i}\tran \omega)^{2} + 1\right]\\
& = \sup_{\omega\in \S^{p-1}}\frac{1}{n}\sum_{i=1}^{n}(x_{i}\tran \omega)^{4} - 2\omega\tran \lb\frac{X\tran X}{n}\rb\omega + 1
 = \sup_{\omega\in \S^{p-1}}\frac{1}{n}\sum_{i=1}^{n}(x_{i}\tran \omega)^{4} - 1 \le \sup_{\omega\in \S^{p-1}}\frac{1}{n}\sum_{i=1}^{n}(x_{i}\tran \omega)^{4}\\
& \le \sup_{\omega\in \S^{p-1}}\frac{1}{n}\sum_{i=1}^{n}\|x_{i}\|_{2}^{2}(x_{i}\tran \omega)^{2} = \left\|\sum_{i=1}^{n}H_{ii}x_{i}x_{i}\tran \right\|_{\op} \le n\kappa .
\end{align*}
By Lemma \ref{thm:matrix_concentration},
\begin{align*}
  \left\|\Sigma_t - \id \right \|_{\op} 
  &= \frac{1}{n_{t}}\left\|\sum_{i \in \mathcal{T}_t} V_{i}\right\|_{\op}
  =  O_\P\lb \frac{1}{n_{t}}\left[n (C(p)\kappa)^{1/2} + nC(p)\kappa + n \kappa^{1/2}\right]\rb \\
  & =   O_\P\lb (\kappa\log p)^{1/2} + \kappa\log p\rb.
\end{align*}
By Assumption \ref{assumption::A2}, $\kappa \log p = o(1)$, and therefore the first result holds:
\begin{equation}\label{eq:inv1}
\left\|\Sigma_t - \id \right \|_{\op} = O_\P\lb (\kappa\log p)^{1/2} \rb = o_\P(1).
\end{equation}

 Thus with probability $1 - o(1)$, by Weyl's inequality, 
\begin{equation}\label{eq:original_concentration}
\left\|\Sigma_t - \id \right \|_{\op} \le \frac{1}{2}\Longrightarrow \lambda_{\min}(\id) - \lambda_{\min}(\Sigma_t) \le \frac{1}{2} \Longrightarrow \lambda_{\min}(\Sigma_t)\ge \frac{1}{2}
\end{equation}
where $\lambda_{\min}$ denotes the minimum eigenvalue. Note that for any Hermitian matrix $A$, $\|A^{-1}\|_{\op} = \lambda_{\min}(A)^{-1}$. Thus with probability $1 - o(1)$,
\begin{equation}\label{eq:le2}
\left\|\Sigma_t^{-1}\right\|_{\op}\le 2.
\end{equation}
Therefore, the second result holds:
$
\left\|\Sigma_t^{-1}\right\|_{\op} = O_\P(1).
$

To prove the third result, we apply Proposition \ref{prop:matrix_inv_diff} with $A = \id $ and $B = \Sigma_t  -  \id $. By \eqref{eq:original_concentration} and \eqref{eq:le2}, with probability $1 - o(1)$, $A + B$ is invertible and
$\|B\|_{\op} \le 1/2.$
Together with \eqref{eq:inv1}, we have
\[\left\|\Sigma_t^{-1} - \id \right \|_{\op} = O_\P\lb\frac{\|B\|_{\op}}{1 - \|B\|_{\op}}\rb = O_\P(\|B\|_{\op}) = O_\P((\kappa \log p)^{1/2}).\]
\end{proof}

\begin{proof}[Proof of Lemma \ref{lemma:Mt}]
First, \eqref{eq:error_key_property} implies
\[\one\tran Q(t)   = \one\tran H \diag(e(t)) = \one\tran X(X\tran X)^{-1}X\tran \diag(e(t)) = 0,\]
\[Q(t) \one = H \diag(e(t))\one = He(t) = X(X\tran X)^{-1}X\tran e(t) = 0,\]
which further imply $\one\tran Q(t) \one = 0.$
Second, \eqref{eq:hatDeltat} implies 
$
\text{tr}(Q(t) ) =n  \Delta_t . 
$
Third, 
$$
  \|Q(t) \|_{F}^{2} = \sum_{i=1}^{n} \sum_{j=1}^{n} H_{ij}^{2}e_{j}^2(t)    = \sum_{j=1}^{n}e_{j}^2(t)    \lb\sum_{i=1}^{n}H_{ij}^{2}\rb.
$$
Because $H$ is idempotent, 
$ H\tran H = H  \Longrightarrow \sum_{i=1}^{n}H_{ij}^{2} = H_{jj}$ for all $j$. Thus,
$
  \|Q(t) \|_{F}^{2} = \sum_{j=1}^{n}e_{j}^2(t)   H_{jj} \le n \cure_{2}\kappa.
$
\end{proof}

\clearpage

~\\
\begin{center}
  \begin{Large}
    Supplementary Material II: Discussion of the Assumptions
  \end{Large}
\end{center}

\section{Discussion of Assumptions}\label{sec::discussion-assumptions}
 Although Assumptions \ref{assumption::A2} and \ref{assumption::A4} are about fixed quantities in the finite population, it is helpful to consider the case where the quantities are realizations of random variables. This approach connects the assumptions to more comprehensible conditions on the data generating process. See \cite{portnoy84, portnoy85, lei16} for examples in other contexts. We emphasize that we do not need the assumptions in this subsection for our main theory but use them to aid interpretation. 

For Assumption \ref{assumption::A2}, we consider the case where $(x_{i})_{i=1}^{n}$ are realizations of i.i.d. random vectors. \cite{anatolyev17} show that under mild conditions each leverage score concentrates around $p/n$. Here we  further consider the magnitude of the maximum leverage score $\kappa$.

\begin{proposition}\label{prop:kappa_iid}
  Let $Z_{i}$ be i.i.d. random vectors in $\R^{p}$ with arbitrary mean. Assume that $Z_{i}$ has independent entries with
$\max_{1\leq j \leq p}\E |Z_{ij} - \E Z_{ij}|^{\delta} \le M = O(1)$
for some $\delta > 2$. 
  Define $Z = (Z_{1}\tran , \ldots, Z_{n}\tran )\tran \in \R^{n\times p}$ and $X =\projone  Z$ where $\projone$ is defined in Section \ref{sec:intro} so that $X$ has centered columns. If $p = O(n^{\gamma})$ for some $\gamma < 1$, then over the randomness of $Z$, 
$$
\kappa_0 = O_\P\lb \frac{p^{2 / \min\{\delta, 4\}}}{n^{(\delta - 2) / \delta}} + \frac{p^{3/2}}{n^{3/2}}\rb,
\quad 
\kappa = O_\P\lb\frac{p}{n} + \frac{p^{2 / \min\{\delta, 4\}}}{n^{(\delta - 2) / \delta}}\rb.
$$ 
\end{proposition}

When $\delta > 4$, Proposition \ref{prop:kappa_iid} implies that $\kappa = O_\P(p / n + n^{-(\delta - 4) / 2\delta}(p / n)^{1/2})$. In this case, Assumption \ref{assumption::A2} holds with high probability if $p = O(n^{\gamma})$ for any $\gamma < 1$. In particular, the fixed-$p$ regime corresponds to $\gamma =0.$
Under $ p = o(n^{2/3} /(\log n)^{1/3})$, we can use Proposition \ref{prop:kappa_iid} to verify that the condition (ii) of Corollary \ref{cor:asym_normality_htau} holds with high probability if entries of $X$ are independent and have finite $12$-th moments, and that the condition $\kappa^{2} p \log p = o(1)$ holds if $p = o(n^{2/3} / (\log n)^{1/3})$ and entries of $X$ are independent and have finite $(6+\eps)$-th moments.

The hat matrix of $X$ is invariant to any nonsingular linear transformation of the columns. Consequently, $X$ and $XA$ have the same leverage scores for any invertible $A\in \R^{n\times n}$. Thus we can extend Proposition \ref{prop:kappa_iid} to random matrices with correlated columns in the form of $\projone ZA$. In particular, when $Z_{i} \stackrel{\text{i.i.d.}}{\sim} N(\mu, \id)$ and $A = \Sigma^{1/2}$, $Z_{i}\tran A\stackrel{\text{i.i.d.}}{\sim} N(\Sigma^{1/2}\mu, \Sigma)$. The previous argument implies that Proposition \ref{prop:kappa_iid} holds for $X = \projone ZA$.

For Assumption \ref{assumption::A4}, we consider the case where the $Y_{i}(t)$'s are realizations of i.i.d. random variables, and make a connection with the usual moment conditions. This helps to understand the growth rates of $\cure_{2}$ and $\cure_{\infty}$.

\begin{proposition}\label{prop:eq}
  Let $Y(t)\in \R^{n}$ be a non-constant random vector with i.i.d. entries, and $X$ be any fixed matrix with centered columns. If for some $\delta > 0$,
$
\E |Y_{i}(t) - \E Y_{i}(t)|^{\delta} < \infty
$
for $t=0,1$,
then

\[\cure_{2} = \left\{
    \begin{array}{ll}
      O_\P(1) & (\delta \ge 2)\\
      o_\P(n^{2 / \delta - 1}) & (\delta < 2)
    \end{array}
\right., \quad \cure_{\infty} = O_{\P}(n^{1/\delta}).\]
Furthermore, $\cure_{2}^{-1} = O_\P(1)$ if $Y_{i}(1)$ or $Y_{i}(0)$ is not a constant.
\end{proposition}

When $\delta > 2$, Proposition \ref{prop:eq} implies $\cure_{\infty}^{2} / (n\cure_{2}) = O_\P(n^{2/\delta - 1}) = o_\P(1)$, and thus Assumption \ref{assumption::A4} holds with high probability.
From Proposition \ref{prop:eq}, the condition $\cure_{2} = o(n)$ corresponds to the existence of finite first moment under a super-population i.i.d sampling.

\section{Some useful results for the proofs} \label{sec::useful-results-proofs}
 
The proofs rely on the following results.

\begin{proposition}\label{prop:yaskov}[modified version of Corollary 3.1 of \cite{yaskov14}]
  Let $Z_{i}$ be i.i.d. random vectors in $\R^{p}$ with mean $0$ and covariance $\id $. Suppose 
\[L(\delta)\triangleq   \sup_{ \nu \in \mathcal{S}^{p-1} } 
\E |\nu\tran Z_{i}|^{\delta} < \infty\]
for some $\delta > 2$. For any constant $C > 0$, with probability $1 - e^{-Cp}$,
\[\lambda_{\min}\lb \frac{Z\tran Z}{n}\rb\ge 1 - 5\lb\frac{pC}{n}\rb^{\frac{\delta}{\delta + 2}}L(\delta)^{\frac{2}{\delta + 2}}\lb 1+ \frac{1}{C}\rb.\]
\end{proposition}

\begin{proof}
Write $y = p / n$ and $L = L(\delta)$. The proof of Corollary 3.1 of \citet[][page 6]{yaskov14} showed that for any $a > 0$,
\[\P\lb \lambda_{\min}\lb \frac{Z\tran Z}{n}\rb < 1 - 4La^{- \delta / 2} - 5ay\rb\le \exps{-La^{-1 - \delta / 2}n}.\]
Let $a = \lb Cy / L\rb^{-2/ (\delta + 2)}$. Then the right-hand side is $1 - e^{-Cp}$. Thus with probability $1 - e^{-Cp}$, 
\begin{align*}
 \lambda_{\min}\lb \frac{Z\tran Z}{n}\rb  
 \ge 1 - y^{\frac{\delta}{\delta + 2}}L^{\frac{2}{\delta + 2}}\lb 5 C^{-\frac{2}{\delta + 2}} + 4C^{\frac{\delta}{\delta + 2}}\rb 
  \ge 1 - 5\lb Cy\rb^{\frac{\delta}{\delta + 2}}L^{\frac{2}{\delta + 2}}\lb 1+ \frac{1}{C}\rb.
\end{align*}
\end{proof}

\begin{proposition}[Theorem 1 of \cite{tikhomirov17}]\label{prop:operator_norm}
  Let $Z_{i}$ be i.i.d. random vectors in $\R^{p}$ with mean $0$ and covariance $\id$. Suppose 
\[L(\delta)\triangleq \sup_{ \nu \in \mathcal{S}^{p-1} } 
\E |\nu\tran Z_{i}|^{\delta} < \infty\]
for some $\delta > 2$. Then with probability at least $1 - 1/n$,
\[\nu(\delta)^{-1}\left\|\frac{Z\tran Z}{n} - \id \right \|_{\op} \le \frac{\max_{1\leq i \leq n}\|Z_{i}\|_{2}^{2}}{n} + L(\delta)^{\frac{2}{\delta}}\left\{\lb \frac{p}{n}\rb^{\frac{\delta - 2}{\delta}}\log^{4}\lb \frac{n}{p}\rb + \lb\frac{p}{n}\rb^{\frac{\min\{\delta - 2, 2\}}{\min\{\delta, 4\}}}\right\},\]
for some constant $\nu(\delta)$ depending only on $\delta$.
\end{proposition}

\begin{proposition}[Theorem 2 of \cite{vonbahr65}]\label{prop:vonbahresseen}
  Let $Z_{i}$ be independent mean-zero random variables. Then for any $r \in [1, 2)$,
\[\E \bigg|\sum_{i=1}^{n}Z_{i}\bigg|^{r}\le 2\sum_{i=1}^{n}\E |Z_{i}|^{r}.\]
\end{proposition}

\section{Proof of Proposition \ref{prop:kappa_iid}}\label{section:proof-of-F1}

\subsection{A lemma}

First we prove a more general result. 
\begin{lemma}\label{prop:kappa}
Let $Z_{i}$ be i.i.d. random vectors in $\R^{p}$ with mean $\mu \in \R^p $ and covariance matrix $\Sigma\in \R^{p\times p}$. Let $\td{Z}_{i} = \Sigma^{-1/2}(Z_{i} - \mu)$, and assume
\[\sup_{\nu\in \S^{p-1}}\E |\nu\tran \td{Z}_{i}|^{\delta} = O(1), \quad \text{and}\quad 
\max_{1\leq i \leq n}\big|\|\td{Z}_{i}\|_{2}^{2} - \E \|\td{Z}_{i}\|_{2}^{2}\big| = O_\P(\omega(n, p)),\]
for some $\delta> 2$ and some function $\omega(n, p)$ increasing in $n$ and $p$. Further let $Z = (Z_{1}\tran , \ldots, Z_{n}\tran )\tran $ and $X = \projone  Z$ so that $X$ has centered columns. If $p = O(n^{\gamma})$ for some $\gamma < 1$, then over the randomness of $Z$, 
\[
\kappa_0 = O_\P\lb \frac{\omega(n, p)}{n} + \lb\frac{p}{n}\rb^{\frac{2\delta - 2}{\delta}}\log^{4}\lb \frac{n}{p}\rb + \lb\frac{p}{n}\rb^{\frac{\min\{2\delta - 2, 6\}}{\min\{\delta, 4\}}}\rb.\]
\end{lemma}

\begin{proof}[Proof of Lemma \ref{prop:kappa}]
Let $\td{Z} = (\td{Z}_{1}\tran , \ldots, \td{Z}_{n})\tran $ and 
$\td{X} = \projone  \td{Z}.$
Then 
$\td{X} = \projone  \lb Z - \one  \mu\tran \rb\Sigma^{-\frac{1}{2}} = \projone  Z\Sigma^{-\frac{1}{2}},$
and thus
\begin{align*}
  \td{X}(\td{X}\tran \td{X})^{-1}\td{X}\tran  &= \projone  Z\lb Z\tran \projone  Z\rb^{-1}Z\tran \projone 
 = X(X\tran X)^{-1}X\tran .
\end{align*}
Therefore, we can assume $\mu = 0$ and $ \Sigma = \id $ without loss of generality, in which case $Z_{i} = \td{Z}_{i}$ has mean 0 and covariance matrix $\id$.

By definition,
$H_{ii} = x_{i}\tran (X\tran X)^{-1}x_{i}$, and therefore
\begin{align*}
  H_{ii}  = \frac{1}{n}x_{i}\tran \lb \lb\frac{X\tran X}{n}\rb^{-1} - \id \rb  x_{i} + \frac{\|x_{i}\|_{2}^{2}}{n} 
 \le \frac{\|x_{i}\|_{2}^{2}}{n}\lb 1 + \left\|\lb\frac{X\tran X}{n}\rb^{-1} - \id \right \|_{\op}\rb.
\end{align*}
As a result,
\begin{equation}
  \label{eq:kappaHii1}
\kappa_0 \le \max_{1\le i\le n}\frac{|\|x_{i}\|_{2}^{2} - p|}{n} +  \left\|\lb\frac{X\tran X}{n}\rb^{-1} - \id \right \|_{\op}\max_{1\le i\le n}\frac{\|x_{i}\|_{2}^{2}}{n}.
\end{equation}
To bound $\kappa$, we need to bound two key terms below.

\paragraph{Bounding $\left\|\lb  n^{-1} X\tran X \rb^{-1} - \id \right \|_{\op}$} Let $\bar{Z} = n^{-1}  \sum_{i=1}^{n}Z_{i}$. Note that
\[\E \|\bar{Z}\|_{2}^{2} = \frac{1}{n^{2}}\sum_{i=1}^{n}\E \|Z_{i}\|_{2}^{2} = \frac{1}{n}\E \|Z_{1}\|_{2}^{2} = \frac{p}{n}.\]
By Markov's inequality,
\begin{equation}\label{eq:barZ}
  \|\bar{Z}\|_{2}^{2} = O_{\P}\lb\frac{p}{n}\rb,
\end{equation}
and more precisely, 
\begin{equation}\label{eq:E1}
  \P\lb \|\bar{Z}\|_{2}^{2} \le \lb\frac{p}{n}\rb^{1/2}\rb = 1 - \P\lb \|\bar{Z}\|_{2}^{2} > \lb\frac{p}{n}\rb^{1/2}\rb \ge 1 - \lb\frac{p}{n}\rb^{1/2}.
\end{equation} 
Let $\mathcal{A}_{1}$ denote the above event that $\|\bar{Z}\|_{2}^{2} \le (p / n)^{1/2}$. Then
$ 
\P(\mathcal{A}_{1}) \ge 1 - (p / n)^{1/2}.
$

By Proposition \ref{prop:operator_norm}, 
\[\left\|\frac{Z\tran Z}{n} - \id \right \|_{\op}= O_\P\lb \frac{\max_{1\leq i \leq n}\|Z_{i}\|_{2}^{2}}{n} + \lb\frac{p}{n}\rb^{\frac{\delta - 2}{\delta}}\log^{4}\lb \frac{n}{p}\rb + \lb\frac{p}{n}\rb^{\frac{\min\{\delta - 2, 2\}}{\min\{\delta, 4\}}}\rb  .\]
By the condition of Lemma \ref{prop:kappa}, 
\begin{equation}\label{eq:max2norm}
\frac{\max_{1\leq i \leq n}\|Z_{i}\|_{2}^{2}}{n} = \frac{p}{n}  +    \max_{1\leq i \leq n}\frac{\|Z_{i}\|_{2}^{2}    - \E \|Z_{i}\|_{2}^{2}}{n}  = \frac{p}{n} + O_\P\lb \frac{\omega(n, p)}{n}\rb.
\end{equation}
Combining the above three equations, we have 
\begin{align}
  & \left\|\frac{X\tran X}{n} - \id \right \|_{\op} = \left\|\frac{Z\tran Z}{n} - \id - \bar{Z}\bar{Z}\tran \right\|_{\op}\le \left\|\frac{Z\tran Z}{n} - \id \right \|_{\op} + \|\bar{Z}\|_{2}^{2}\nonumber\\
= & O_\P\lb \frac{p}{n} + \frac{\omega(n, p)}{n} + \lb\frac{p}{n}\rb^{\frac{\delta - 2}{\delta}}\log^{4}\lb \frac{n}{p}\rb + \lb\frac{p}{n}\rb^{\frac{\min\{\delta - 2, 2\}}{\min\{\delta, 4\}}}\rb\nonumber\\
= & O_\P\lb\frac{\omega(n, p)}{n} + \lb\frac{p}{n}\rb^{\frac{\delta - 2}{\delta}}\log^{4}\lb \frac{n}{p}\rb + \lb\frac{p}{n}\rb^{\frac{\min\{\delta - 2, 2\}}{\min\{\delta, 4\}}}\rb,\label{eq:XTX-I}
\end{align}
where the last line uses the fact that the third term dominates the first term due to $p / n \rightarrow 0$. On the other hand, by Proposition \ref{prop:yaskov} with $C = (n / p)^{1/2}$, with probability $1 - e^{-(np)^{1/2}}$, 
\begin{align}
  \lambda_{\min}\lb\frac{Z\tran Z}{n}\rb&\ge 1 - 5\lb\lb\frac{p}{n}\rb^{1/2}\rb^{\frac{\delta}{\delta + 2}}L(\delta)^{\frac{2}{\delta + 2}}\lb 1 + \lb\frac{p}{n}\rb^{1/2}\rb  \ge 1 - 10\lb\frac{p}{n}\rb^{\frac{\delta}{2(\delta + 2)}}L(\delta)^{\frac{2}{\delta + 2}}. \label{eq::A2}
\end{align}
Let $\mathcal{A}_{2}$ denote the event in \eqref{eq::A2}. Then
$
\P(\mathcal{A}_{2})\ge 1 - e^{-(np)^{1/2}}.
$

Note that for any Hermitian matrices $A$ and $B$, the convexity of $\|\cdot\|_{\op}$ implies that
\[|\lambda_{\min}(A) - \lambda_{\min}(B)| = |\lambda_{\max}(-A) - \lambda_{\max}(-B)| \le \|-A - (-B)\|_{\op} = \|A - B\|_{\op}.\]
 We have
\[\lambda_{\min}\lb\frac{X\tran X}{n}\rb\ge \lambda_{\min}\lb\frac{Z\tran Z}{n}\rb - \|\bar{Z}\bar{Z}\tran \|_{\op} = \lambda_{\min}\lb\frac{Z\tran Z}{n}\rb - \|\bar{Z}\|_{2}^{2}.\]
Let $\mathcal{A} = \mathcal{A}_{1}\cup \mathcal{A}_{2}$. Then on $\mathcal{A}$, 
\begin{equation*}
  \lambda_{\min}\lb\frac{X\tran X}{n}\rb \ge 1 - 10\lb\frac{p}{n}\rb^{\frac{\delta}{2(\delta + 2)}}L(\delta)^{\frac{2}{\delta + 2}} - \lb\frac{p}{n}\rb^{1/2}.
\end{equation*}
Since $p / n\rightarrow 0$, for sufficiently large $n$, 
\[\lambda_{\min}\lb\frac{X\tran X}{n}\rb \ge \frac{1}{2}\]
with probability 
\[\P(\mathcal{A})\ge \P(\mathcal{A}_{1}) + \P(\mathcal{A}_{2}) - 1\ge 1 - e^{-(np)^{1/2}} - \lb\frac{p}{n}\rb^{1/2} = 1 - o(1).\]

Finally, using Proposition \ref{prop:matrix_inv_diff} with $A=\id$ and $B = n^{-1} X\tran X - \id$, by Slutsky's lemma, we have that
\begin{equation}
  \label{eq:XTXinv-I}
  \left\|\lb\frac{X\tran X}{n}\rb^{-1} - \id \right \|_{\op} = O_\P\lb\frac{\omega(n, p)}{n} + \lb\frac{p}{n}\rb^{\frac{\delta - 2}{\delta}}\log^{4}\lb \frac{n}{p}\rb + \lb\frac{p}{n}\rb^{\frac{\min\{\delta - 2, 2\}}{\min\{\delta, 4\}}}\rb.
\end{equation}
Because  $p = O(n^{\gamma})$ for some $\gamma < 1$, 
\begin{equation}
  \label{eq:XTXinv-I2}
 \lb\frac{p}{n}\rb^{\frac{\delta - 2}{\delta}}\log^{4}\lb \frac{n}{p}\rb + \lb\frac{p}{n}\rb^{\frac{\min\{\delta - 2, 2\}}{\min\{\delta, 4\}}} = o(1).
\end{equation}

\paragraph{Bounding $  \max_{1\leq i\leq n} \|x_{i}\|_{2}^{2} $}
Because  $x_{i} = Z_{i} - \bar{Z}$, the Cauchy--Schwarz inequality implies 
\[\|x_{i}\|_{2}^{2} = \|Z_{i}\|_{2}^{2} - 2Z_{i}\tran \bar{Z} + \|\bar{Z}\|_{2}^{2}\le \|Z_{i}\|_{2}^{2} + 2\|Z_{i}\|_{2}\|\bar{Z}\|_{2} + \|\bar{Z}\|_{2}^{2}.\]
By \eqref{eq:max2norm}, \eqref{eq:barZ} and the fact that $p = \E \|Z_{i}\|_{2}^{2}$, we have
\begin{align}
  \frac{\max_{1\leq i \leq n}|\|x_{i}\|_{2}^{2} - p|}{n}& \le \frac{\max_{i}\big|\|Z_{i}\|_{2}^{2} - \E \|Z_{i}\|_{2}^{2}+ 2 \|Z_{i}\|_{2}\|\bar{Z}\|_{2} + \|\bar{Z}\|_{2}^{2}\big|}{n} \nonumber\\
& = O_\P\lb \frac{\omega(n, p)}{n} + \lb\frac{(p + \omega(n, p))p}{n^{3}}\rb^{1/2} + \frac{p}{n^{2}}\rb\nonumber\\
& = O_\P\lb \frac{\omega(n, p)}{n} + \lb\frac{\omega(n, p)p}{n^{3}}\rb^{1/2} + \frac{p}{n^{3/2}} \rb\nonumber.
\end{align}
Because  $\omega(n, p)$ is increasing and $p / n\rightarrow 0$, we have
\[\lb\frac{\omega(n, p)p}{n^{3}}\rb^{1/2} = O\lb\frac{\omega(n, p)}{n}\lb\frac{p}{n}\rb^{1/2}\rb = o\lb\frac{\omega(n, p)}{n}\rb.\]
Thus, we obtain that
\begin{equation}
  \label{eq:max2normx0}
  \frac{\max_{1\leq i \leq n}|\|x_{i}\|_{2}^{2} - p|}{n} = O_\P\lb \frac{\omega(n, p)}{n}  + \frac{p}{n^{3/2}} \rb,
\end{equation}
and thus
\begin{equation}
  \label{eq:max2normx}
  \frac{\max_{1\leq i \leq n}\|x_{i}\|_{2}^{2}}{n} = O_\P\lb \frac{\omega(n, p)}{n}  + \frac{p}{n} \rb.
\end{equation}

Putting \eqref{eq:kappaHii1} and  \eqref{eq:XTXinv-I}--\eqref{eq:max2normx} together and using some tedious cancellations, we have
\begin{align}
\kappa_0
  =  O_\P\lb \frac{\omega(n, p)}{n}  + \frac{p}{n^{3/2}} +\frac{\omega^{2}(n, p)}{n^{2}} + \lb\frac{p}{n}\rb^{1 + \frac{\delta - 2}{\delta}}\log^{4}\lb \frac{n}{p}\rb + \lb\frac{p}{n}\rb^{1 + \frac{\min\{\delta - 2, 2\}}{\min\{\delta, 4\}}}\rb.\label{eq:Hii1}
\end{align}
Because
\[
\lb\frac{p}{n}\rb^{1 + \frac{\min\{\delta - 2, 2\}}{\min\{\delta, 4\}}}\ge \lb\frac{p}{n}\rb^{3/2}\ge \frac{p}{n^{3/2}},
\]
\eqref{eq:Hii1} further simplifies to
\begin{align*}
\kappa_0 &= O_\P\lb \frac{\omega(n, p)}{n} + \frac{\omega^{2}(n, p)}{n^{2}} + \lb\frac{p}{n}\rb^{\frac{2\delta - 2}{\delta}}\log^{4}\lb \frac{n}{p}\rb + \lb\frac{p}{n}\rb^{\frac{\min\{2\delta - 2, 6\}}{\min\{\delta, 4\}}}\rb\\
& = O_\P\lb \frac{\omega(n, p)}{n}\max\left\{ \frac{\omega(n, p)}{n}, 1\right\} + \lb\frac{p}{n}\rb^{\frac{2\delta - 2}{\delta}}\log^{4}\lb \frac{n}{p}\rb + \lb\frac{p}{n}\rb^{\frac{\min\{2\delta - 2, 6\}}{\min\{\delta, 4\}}}\rb .
\end{align*}
If $\omega(n, p) / n\ge 1$, then the lemma is proved by the fact that $\kappa\le 1$. Otherwise, the lemma is also proved. 
\end{proof}

\subsection{Using Lemma \ref{prop:kappa} to prove Proposition \ref{prop:kappa_iid}}

We have argued in the proof of Proposition \ref{prop:kappa} that we can assume $\mu = 0$ without loss of generality. Because the hat matrix is invariant to rescaling, we further assume $\E Z_{ij}^{2} = 1$ without loss of generality. Based on Proposition \ref{prop:kappa}, it suffices to verify
\begin{eqnarray}
  \label{eq:kappa_goal1}
\sup_{\nu\in \S^{p-1}}\E |\nu\tran Z_{i}|^{\delta} &=& O(1), 
\\
  \label{eq:kappa_goal2}
\max_{1\leq i \leq n} \big| \|Z_{i}\|_{2}^{2} - \E \|Z_{i}\|_{2}^{2}\big| &=& O_\P\lb n^{\frac{2}{\delta}}p^{\frac{2}{\min\{\delta, 4\}}}\rb.
\end{eqnarray}
If \eqref{eq:kappa_goal1} and \eqref{eq:kappa_goal2} hold, by Proposition \ref{prop:kappa}, we have that 
\[
\kappa_0 = O_\P\lb \frac{p^{2/\min\{\delta, 4\}}}{n^{(\delta - 2) / \delta}} + \lb\frac{p}{n}\rb^{\frac{2\delta - 2}{\delta}}\log^{4}\lb \frac{n}{p}\rb + \lb\frac{p}{n}\rb^{\frac{\min\{2\delta - 2, 6\}}{\min\{\delta, 4\}}}\rb.\]
Then we can prove Proposition \ref{prop:kappa_iid} by considering two cases.

\paragraph{Case 1} 
If $\delta > 4$, then $\frac{2\delta - 2}{\delta} > \frac{3}{2} = \frac{\min\{2\delta - 2, 6\}}{\min\{\delta, 4\}}$. Thus the third term dominates the second term in the above $O_\P(\cdot)$, implying   
\begin{align*}
\kappa_0= O_\P\lb \frac{p^{1/2}}{n^{(\delta - 2) / \delta}} + \lb\frac{p}{n}\rb^{\frac{3}{2}}\rb.
\end{align*}

\paragraph{Case 2}
If $\delta \le 4$, then 
\[
\kappa_0 = O_\P\lb\frac{p^{2 / \delta}}{n^{(\delta - 2) / \delta}} + \lb\frac{p}{n}\rb^{\frac{2\delta - 2}{\delta}} \log^{4}\lb\frac{n}{p}\rb\rb.\]
Because $p = O(n^{\gamma})$ for some $\gamma < 1$ and
\[\lb\frac{p}{n}\rb^{\frac{2\delta - 2}{\delta}} = \frac{p^{2 / \delta}}{n^{(\delta - 2) / \delta}}\frac{p^{(2\delta - 4) / \delta}}{n}\le \frac{p^{2 / \delta}}{n^{(\delta - 2) / \delta}}\frac{p}{n} , \]
the first term dominates in the above $O_\P(\cdot)$, implying 
\[
\kappa_0  = O_\P\lb\frac{p^{2 / \delta}}{n^{(\delta - 2) / \delta}}\rb 
=  O_\P\lb\frac{p^{2 / \delta}}{n^{(\delta - 2) / \delta}} + \lb\frac{p}{n}\rb^{\frac{3}{2}}\rb .
\]
The last identity holds because $p^{3/2}/n^{3/2}$ is of smaller order and thus we can add it back.

We will prove \eqref{eq:kappa_goal1} and \eqref{eq:kappa_goal2} below. 

\paragraph{Proving \eqref{eq:kappa_goal1}}
By \citet{rosenthal70}'s inequality, 
\begin{align*}
  \E |\nu\tran Z_{i}|^{\delta} &= \E \bigg|\sum_{j=1}^{p}\nu_{j}Z_{ij}\bigg|^{\delta} \le C \lb \sum_{j=1}^{p}|\nu_{j}|^{\delta} \E |Z_{ij}|^{\delta} + \lb \sum_{j=1}^{p}\nu_{j}^{2} \E Z_{ij}^{2}\rb^{\delta / 2}\rb
\end{align*}
where $C$ is a constant depending only on $\delta$. Because  $\|\nu\|_{2} = 1$, we have $\max_{1\leq j \leq p}|\nu_{j}|\le 1$ and thus
\[\sum_{j=1}^{p}|\nu_{j}|^{\delta} \E |Z_{ij}|^{\delta}\le M \sum_{j=1}^{p}|\nu_{j}|^{\delta} \le M \sum_{j=1}^{p}|\nu_{j}|^{2} =  M.\]
H\"{o}lder's inequality implies
$\E Z_{ij}^{2} \le \lb \E |Z_{ij}|^{\delta}\rb^{2/\delta}\le M^{2 / \delta},$
which further implies 
\[\lb \sum_{j=1}^{p}\nu_{j}^{2} \E Z_{ij}^{2}\rb^{\delta / 2}\le \lb M^{2 / \delta}\rb^{\delta / 2} = M.\]
Because the above two bounds hold regardless of $\nu$, we have 
$
\sup_{\nu\in \S^{p-1}} \E |\nu\tran Z_{i}|^{\delta}\le 2CM = O(1).
$

\paragraph{Proving \eqref{eq:kappa_goal2}}
Let 
$W_{ij} = Z_{ij}^{2} - \E Z_{ij}^{2}.
$
Using Jensen's inequality that $\E |(X + Y) / 2|^{r}\le (\E |X|^{r} + \E |Y|^{r}) / 2$ for any random variables $X, Y$ and $r > 1$, we obtain that
\begin{align*}
  \E |W_{ij}|^{\delta / 2}\le 2^{\delta / 2 - 1}\lb\E |Z_{ij}|^{\delta} + (\E Z_{ij}^{2})^{\delta / 2}\rb 
 \le 2^{\delta / 2}\E |Z_{ij}|^{\delta}\le 2^{\delta / 2}M \triangleq \td{M}.
\end{align*}
We consider two cases.

\paragraph{Case 1: $\delta \ge 4$} 
By H\"{o}lder's inequality,
$\E W_{ij}^{2}\le \td{M}^{4 / \delta}.$
By \citet{rosenthal70}'s inequality,
\begin{align*}
  \E |\|Z_{i}\|_{2}^{2} - \E \|Z_{i}\|_{2}^{2}|^{\delta / 2}  &= \E \bigg|\sum_{j=1}^{p}W_{ij}\bigg|^{\delta / 2}
\le C\lb \sum_{j=1}^{p}\E |W_{ij}|^{\delta / 2} + \lb \sum_{j=1}^{p}\E W_{ij}^{2}\rb^{\delta / 4}\rb \\
  & \le  C\lb p\td{M} + p^{\delta / 4}\td{M}\rb \le C\td{M}p^{\delta / 4},
\end{align*}
which implies 
$
\E \big| \|Z_{i}\|_{2}^{2} - \E \|Z_{i}\|_{2}^{2}\big|^{\delta / 2} = O\lb p^{\delta / 4}\rb.
$
As a result,
\begin{align*}
  &\E   \left\{   \max_{1\leq i \leq n}\big| \|Z_{i}\|_{2}^{2} - \E \|Z_{i}\|_{2}^{2}\big|^{\delta / 2} \right\}
   \le \sum_{i=1}^{n}\E \big| \|Z_{i}\|_{2}^{2} - \E \|Z_{i}\|_{2}^{2}\big|^{\delta / 2} = O\lb n p^{\delta / 4}\rb.
\end{align*}
By Markov's inequality,
$
\max_{1\leq i \leq n}\big| \|Z_{i}\|_{2}^{2} - \E \|Z_{i}\|_{2}^{2}\big| = O_\P\lb n^{2/\delta}p^{1/2}\rb.
$

\paragraph{Case 2: $\delta < 4$} 
By Proposition \ref{prop:vonbahresseen}, with $\delta / 2\in (1, 2)$, 
\[\E |\|Z_{i}\|_{2}^{2} - \E \|Z_{i}\|_{2}^{2}|^{\delta / 2} = \E \bigg|\sum_{j=1}^{p}W_{ij}\bigg|^{\delta / 2}\le 2\sum_{j=1}^{p}\E |W_{ij}|^{\delta / 2}\le 2p\td{M}.\]
Similar to Case 1,
$
\max_{1\le i\le n}\big| \|Z_{i}\|_{2}^{2} - \E \|Z_{i}\|_{2}^{2}\big| = O_\P\lb n^{2 / \delta}p^{2 / \delta}\rb.
$

\section{Proof of Proposition \ref{prop:eq}}\label{subapp:cure}

Let $\bar{Y}(t) = n^{-1} \sum_{i=1}^{n}Y_{i}(t)$. Note that $H\one = X(X\tran X)^{-1}X\tran \one = 0$. By definition, 
$
e(t) = (\id  - H)\{  Y(t) - \bar{Y}(t)\one \}  = (\id  - H)\{  Y(t) - \E Y_{i}(t)\one \} .
$
Throughout the rest of the proof, we assume that $\E Y_{i}(t) = 0$ without loss of generality, and define 
$M(\delta)\triangleq \max_{t=0,1}\E |Y_{i}(t)|^{\delta}.$

\subsection{Bounding $\cure_2$}
Let $Z_{i} = Y_{i}(t)^{2}$. Then the moment condition reads
$\E |Z_{i}|^{\delta / 2} < \infty.$
The Kolmogorov--Marcinkiewicz--Zygmund strong law of large number \citep[][Theorem 4.23]{kallenberg06} implies 
\begin{eqnarray*}
\frac{1}{n}\sum_{i=1}^{n}Z_{i}\stackrel{\text{a.s.}}{\rightarrow} \E Z_{1}\Longrightarrow \frac{1}{n}\sum_{i=1}^{n}Z_{i} = O_\P(1), \quad \text{ if } \delta \ge 2, \\
\frac{1}{n^{2/\delta}}\sum_{i=1}^{n}Z_{i}\stackrel{\text{a.s.}}{\rightarrow} 0 \Longrightarrow \frac{1}{n}\sum_{i=1}^{n}Z_{i} = o_\P(n^{2/\delta - 1}),\quad 
\text{ if }  \delta < 2.
\end{eqnarray*}
The bound on $\cure_{2}$ then follows from
\[
\frac{1}{n}\|e(t)\|_{2}^{2} = \frac{1}{n}Y(t)\tran (\id  - H)Y(t)\le \frac{1}{n}\|Y(t)\|_{2}^{2} = \frac{1}{n}\sum_{i=1}^{n}Z_{i}.
\]

\subsection{Bounding $\cure_2^{-1}$}
Without loss of generality, we assume that $Y_{i}(1)$ is not a constant with probability $1$. First we show  
\[
\frac{Y(1)\tran H Y(1)}{Y(1)\tran Y(1)} = o_{\P}(1).
\]
For any permutation $\pi$ on $\{1, \ldots, n\}$, let $H(\pi)$ denote the matrix with
\[H(\pi)_{ij} = H_{\pi(i), \pi(j)}.\]
Because the $Y_{i}(1)$'s are i.i.d., for any $\pi$, 
\[(Y_{1}(1), \ldots, Y_{n}(1))\stackrel{\text{d}}{=}(Y_{\pi^{-1}(1)}(1), \ldots, Y_{\pi^{-1}(n)}(1)),\]
and thus
\begin{align*}
  &\frac{Y(1)\tran H(\pi) Y(1)}{Y(1)\tran Y(1)} = \frac{\sum_{i = 1}^{n}\sum_{j=1}^nH_{\pi(i), \pi(j)}Y_{i}(1) Y_{j}(1)}{\sum_{i=1}^{n}Y_{i}(1)^{2}} \\
  & =  \frac{\sum_{i = 1}^{n}\sum_{j=1}^nH_{i, j}Y_{\pi^{-1}(i)}(1) Y_{\pi^{-1}(j)}(1)}{\sum_{i=1}^{n}Y_{\pi^{-1}(i)}(1)^{2}}\stackrel{\text{d}}{=} \frac{Y(1)\tran H Y(1)}{Y(1)\tran Y(1)}.
\end{align*}
Furthermore, 
$
 Y(1)\tran H Y(1) / Y(1)\tran Y(1) \le 1,
$ so it has finite expectation. This implies that
\[
\E \frac{Y(1)\tran H Y(1)}{Y(1)\tran Y(1)} = \frac{1}{n!}\sum_{\pi} \frac{Y(1)\tran H(\pi) Y(1)}{Y(1)\tran Y(1)} = \frac{1}{n!}\frac{Y(1)\tran H^{*} Y(1)}{Y(1)\tran Y(1)},
\]
where $H^{*} = \sum_{\pi}H(\pi) / n!$ with the summation over all possible permutations. We can show that
\[H_{ii}^{*} = \frac{1}{n}\sum_{i=1}^{n}H_{ii} = \frac{p}{n}, \quad H_{ij}^{*} = \frac{1}{n(n - 1)}\sum_{i\not = j}H_{ij} = -\frac{1}{n(n - 1)}\sum_{i=1}^{n}H_{ii} = - \frac{p}{n (n - 1)},\]
where the last equality uses the fact that $\sum_{i = 1}^{n}\sum_{j=1}^nH_{ij} = 0$. Therefore,
\begin{align*}
 &\E \frac{Y(1)\tran H Y(1)}{Y(1)\tran Y(1)} = \E \frac{Y(1)\tran H^{*} Y(1)}{Y(1)\tran Y(1)}
= \E\frac{\frac{p}{n}Y(1)\tran Y(1) - \frac{p}{n(n - 1)}\sum_{i\not = j}Y_{i}(1)Y_{j}(1)}{Y(1)\tran Y(1)}  \\
= & \frac{p}{n-1} - \frac{p}{n(n - 1)}\E\frac{(\sum_{i=1}^{n}Y_{i}(1))^{2}}{Y(1)\tran Y(1)}\le \frac{p}{n - 1}.
\end{align*}
By Markov's inequality, with probability $1 - \frac{2p}{n - 1} = 1 - o(1)$, 
\[\frac{Y(1)\tran H Y(1)}{Y(1)\tran Y(1)} \le \frac{1}{2}.\]
Let $\mathcal{A}$ denote this event. Then 
$ \P(\mathcal{A}^{c}) = o(1),$ 
and on $\mathcal{A}$,
\[\frac{1}{n}\|e(1)\|_{2}^{2} = \frac{1}{n}Y(1)\tran (\id  - H)Y(1)\ge \frac{1}{2n}\|Y(1)\|_{2}^{2}.\]

~\\
\noindent  On the other hand, fix $k > 0$, and let 
$\td{Z}_{i} = Y_{i}(1)I(|Y_{i}(1)|\le k).$
For sufficiently large $k$,
$\E \td{Z}_{i} > 0.$
By the law of large numbers,
$
 n^{-1}\sum_{i=1}^{n}\td{Z}_{i} = \E \td{Z}_{i} \times  (1 + o_\P(1)).
$
Thus on $\mathcal{A}$,
\[
\cure_{2}\ge \frac{1}{2n}\sum_{i=1}^{n}Y_{i}(1)^{2} \ge \frac{1}{2n}\sum_{i=1}^{n}\td{Z}_{i} = \E \td{Z}_{i} \times (1 + o_\P(1)) 
\]
Since $\P(\mathcal{A}^{c}) = o(1)$, we conclude that
$ 
\cure_{2}^{-1} = O_\P(1).
$

\subsection{Bounding $\cure_{\infty}$}
We apply the triangle inequality to obtain 
$
\|e(t)\|_{\infty}\le \|Y(t)\|_{\infty} + \|HY(t)\|_{\infty}.
$
We bound the first term using a standard technique and Markov's inequality: 
\begin{equation}\label{eq:Ytinfty1}
\E \|Y(t)\|_{\infty}^{\delta}\le \sum_{i=1}^{n}\E |Y_{i}(t)|^{\delta} = n M(\delta)
\Longrightarrow
\|Y(t)\|_{\infty} = O_\P\lb n^{1/\delta}\rb.
\end{equation}
Next we bound the second term $\|HY(t)\|_{\infty}$. Define $\td{Y}(t) = HY(t)$ with 
\[\td{Y}_{i}(t) = \sum_{j=1}^{n}H_{ij}Y_{j}(t), \quad (i=1,\ldots, n).\]
Fix $\eps > 0$ and define  
$
D = \lb\frac{M(\delta)}{\eps}\rb^{1/\delta}.
$
 We decompose $\td{Y}_{i}(t)$ into two parts:
  \begin{align*}
  \td{Y}_{i}(t) = \sum_{j=1}^{n}H_{ij}Y_{j}(t)I(|Y_{j}(t)|\le Dn^{1/\delta}) + \sum_{j=1}^{n}H_{ij}Y_{j}(t)I(|Y_{j}(t)|> Dn^{1/\delta}) 
 \triangleq R_{1, i}(t) + R_{2, i}(t).  
  \end{align*}
The second term $R_{2, i}(t)$ satisfies 
\begin{align}
  \P\lb \exists i, R_{2, i}(t) \not = 0\rb &\le \P\lb \exists j, \,\, |Y_{j}(t)| > Dn^{1/\delta}\rb\le \sum_{j=1}^{n} \P\lb |Y_{j}(t)| > Dn^{1/\delta}\rb \nonumber\\
 & \le \sum_{j=1}^{n}\frac{1}{D^{\delta} n}\E |Y_{j}(t)|^{\delta} \le \frac{M(\delta)}{D^{\delta}} = \eps.\label{eq:R2i}
\end{align}
To deal with the first term $R_{1, i}(t)$, we define 
\[w_j  (t) = Y_j (t)I(|Y_j (t)|\le Dn^{1/\delta}) - \E \left\{  Y_j (t)I(|Y_j (t)|\le Dn^{1/\delta}) \right\}  ,\] 
with $\E w_j (t) = 0$. Because
\[\one\tran H = 0 \Longrightarrow 
\sum_{j=1}^{n}H_{ij} = 0\Longrightarrow 
\sum_{j=1}^{n}H_{ij} \E \left\{  Y_j (t)I(|Y_j (t)|\le Dn^{1/\delta}) \right\}  = 0.\]
we can rewrite $R_{1,i}(t)$ as 
\[R_{1,i}(t) = \sum_{j=1}^{n}H_{ij}w_{j}(t).\]
The rest of the proof proceeds based on two cases.

\paragraph{Case 1: $\delta < 2$}
First, the $w_{j}(t)$'s are i.i.d. with second moment bounded by
  \begin{align*}
    \E w_j (t)^{2} &\le \E \left\{   Y_j ^{2}(t)I(|Y_j (t)|\le Dn^{1/\delta}) \right\} \\
&\le (Dn^{1/\delta})^{2 - \delta}\E |Y_j (t)|^{\delta}\\
& \le n^{(2 - \delta) / \delta} D^{2-\delta} M(\delta) 
  = n^{(2 - \delta) / \delta}\eps^{-(2 - \delta) / \delta} M(\delta)^{2/\delta}.
  \end{align*}
Second, using the fact that $\sum_{j=1}^{n}H_{ij}^{2} = H_{ii}$, we obtain
\[\E R_{1, i}(t)^{2} = \sum_{j=1}^{n}H_{ij}^{2}\E w_{j}(t)^{2} = \E w_{1}(t)^{2}\lb\sum_{j=1}^{n}H_{ij}^{2}\rb = H_{ii}\E w_{1}(t)^{2}.\]
Let $R_{1}(t)$ denote the vector $(R_{1, i}(t))_{i=1}^{n}$. Then
\begin{align*}
  \E \|R_{1}(t)\|_{\infty}^{2} &\le \sum_{i=1}^{n} \E R_{1, i}(t)^{2}  = \lb\sum_{i=1}^{n}H_{ii}\rb \E w_{1}(t)^{2} 
 \le  pn^{(2 - \delta) / \delta}\eps^{-(2 - \delta) / \delta} M(\delta)^{2/\delta} .
\end{align*}
By Markov's inequality, with probability $1 - \eps$,
\begin{align}\label{eq::r1case1}
  \|R_{1}(t)\|_{\infty} \le \lb  \E \|  R_{1}(t)  \|_{\infty}^{2} / \eps \rb^{1/2}
 = p^{1/2}n^{(2 - \delta) / 2\delta}\eps^{-1 / \delta} M(\delta)^{1/\delta}.
\end{align}
Combining \eqref{eq:R2i} and \eqref{eq::r1case1}, we obtain that with probability $1 - 2\eps$,
\[\|HY(t)\|_{\infty}\le p^{1/2}n^{(2 - \delta) / 2\delta}\eps^{-1 / \delta} M(\delta)^{1/\delta}.\]
Because  this holds for arbitrary $\eps$, we conclude that if $\delta < 2$, 
\[\|HY(t)\|_{\infty} = O_\P\lb p^{1/2}n^{1 / \delta - 1 / 2}\rb = o_{\P}(n^{1/\delta}) . \]

\paragraph{Case 2: $\delta \ge 2$}
Using the convexity of the mapping $|\cdot|^{\delta}$, we have 
    \begin{align*}
      \E \bigg|\frac{w_{j}(t)}{2}\bigg|^{\delta} & \le \frac{\E \left\{  |Y_{j}(t)|^{\delta} I(|Y_{j}(t)|\le Dn^{1/\delta}) \right\} 
      + |\E \left\{  Y_{j}(t)I(|Y_{j}(t)|\le Dn^{1/\delta}) \right\}|^{\delta}}{2}.
    \end{align*}
Applying Jensen's inequality on the second term, we have
$$
  \E |w_{j}(t)|^{\delta}  \le 2^{\delta}\E \left\{   |Y_{j}(t)|^{\delta} I(|Y_{j}(t)|\le Dn^{1/\delta}) \right\} 
\le 2^{\delta}\E |Y_{j}(t)|^{\delta} \le 2^{\delta}M(\delta).
$$
By \citet{rosenthal70}'s inequality, there exists a constant $C$ depending only on $\delta$, such that
\begin{align*}
& \E |R_{1, i}(t)|^{\delta} \le C\lb\sum_{j=1}^{n} \E|H_{ij}w_j (t)|^{\delta} + \lb\sum_{j=1}^{n}\E |H_{ij}w_j (t)|^{2}\rb^{\delta / 2}\rb\\
& \le C\lb \lihua{2^{\delta}}M(\delta) \sum_{j=1}^{n} |H_{ij}|^{\delta} + \lb M(2)\sum_{j=1}^{n}H_{ij}^{2}\rb^{\delta / 2}\rb\\
& \le  C\lihua{2^{\delta}}\lb M(\delta)H_{ii}^{\delta / 2 - 1}\sum_{j=1}^{n}H_{ij}^{2} + M(2)^{\delta / 2}H_{ii}^{\delta / 2}\rb\\
&  = C2^{\delta}(M(\delta) + M(2)^{\delta / 2})H_{ii}^{\delta/2} \leq   C2^{\delta}(M(\delta) + M(2)^{\delta / 2})H_{ii}.
\end{align*}
where the last two lines use $\sum_{j=1}^{n}H_{ij}^{2} = H_{ii}$, $H_{ij}^{2}\le H_{ii}$, and $ H_{ii}^{\delta/2} \leq H_{ii}$ due to $H_{ii}\leq 1$ and $\delta/2>1$. As a result,
\begin{align*}
  \E \|R_{1}(t)\|_{\infty}^{\delta}  \le \sum_{i=1}^{n}\E |R_{1, i}(t)|^{\delta}  \le C\lihua{2^{\delta}}(M(\delta) + M(2)^{\delta / 2})\sum_{i=1}^{n}H_{ii} 
 = C\lihua{2^{\delta}}(M(\delta) + M(2)^{\delta / 2})p.
\end{align*}
Markov's inequality implies that with probability $1-\eps$,
\begin{align}\label{eq::r1case2}
  \|R_{1}(t)\|_{\infty} \le \lb  \E \|      R_1(t)    \|_{\infty}^{\delta} / \eps \rb^{1/\delta}
 = p^{1/\delta} \lb C2^{\delta}(M(\delta) + M(2)^{\delta / 2}) / \eps\rb^{1/\delta}.
\end{align}
Combining \eqref{eq:R2i} and \eqref{eq::r1case2}, we obtain that with probability $1 - 2\eps$,
\[\|HY(t)\|_{\infty}\le p^{1/\delta} \lb C2^{\delta}(M(\delta) + M(2)^{\delta / 2}) / \eps\rb^{1/\delta} . \]
Because  this holds for arbitrary $\eps$, we conclude that if $\delta \ge 2$,
\[\|HY(t)\|_{\infty} = O_\P\lb (p / \eps)^{1/\delta}\rb = o_{\P}(n^{1/\delta}).\]

\clearpage

~\\
\begin{center}
  \begin{Large}
    Supplementary Material III: Experiments
  \end{Large}
\end{center}

\section{Additional Numerical Experiments}\label{app:experiments}

\subsection{Derivation of worst-case residuals}
Using the following proposition, we know that the solution of $\eps$ in Section \ref{sec::dgp} is the rescaled  ordinary least squares
residual vector obtained by regressing the leverage scores $(H_{ii})_{i=1}^n$ on $X$ with an intercept.

\begin{proposition}\label{prop:lpqp}
Let $a\in \R^{n}$ be any vector, and $A\in \R^{n\times m}$ be any matrix with $H_{A} = A(A\tran A)^{-1} A\tran $ being its hat matrix. Define $e = (\id - H_{A})a$. Then $x^{*} = n^{1/2}e / \|e\|_{2} $ is the optimal solution of
\[
\max_{x\in \R^{n}}|a\tran x|\quad \mathrm{s.t.}\,\, \|x\|_{2}^{2} / n = 1, A\tran x = 0  .
\] 
\end{proposition}

\begin{proof}
  The constraint $A\tran x = 0$ implies $H_{A}x = 0$. Thus,
$
|a\tran x| =   |a\tran x - a\tran H_{A}x|  = |  a\tran (\id - H_A) x |
= |e\tran x|.
$
The Cauchy--Schwarz inequality implies 
$ |e\tran x|\le \|e\|_{2}\|x\|_{2} = n^{1/2}\|e\|_{2}$, with the maximum objective value achieved by $x = n^{1/2} e / \|e\|_{2}$.
\end{proof}

\subsection{Complementary experimental results on synthetic datasets}

We present more simulation results in the rest of this section and post the programs to replicate all the experimental results at 
 \texttt{https://github.com/lihualei71/RegAdjNeymanRubin/.}

Section \ref{sec:experiment} shows the results for $X$ contains i.i.d. $t(2)$ entries with $\pi_{1} = 0.2$. Here we plot the results for $X$ containing i.i.d. entries from $N(0, 1)$, $t(2)$ and $t(1)$ with $\pi_{i} = 0.2$ or $0.5$, analogous to the results in Sections \ref{subsec:simulation_results}--\ref{subsec:simulation_regularization}. The population residuals $\eps(1)$ and $\eps(0)$ are generated as realizations of random vectors with i.i.d. entries from $N(0, 1)$ or $t(2)$ or $t(1)$, or as the worst-case residuals derived above. 

But for completeness we plot the results in main text again for easy comparison. For all cases, we plotted the results with covariate trimming as well. The case with $N(0, 1)$ entries exhibits almost the same qualitative pattern; see Figure \ref{fig::simulation_normal_01} - Figure \ref{fig::simulation_normal_02}. However, for the case with $t(1)$ entries, the bias reduction is less effective and none of the variance estimates, including the HC3 estimate, is able to protect against undercoverage when $p > n^{1/2}$; see Figure \ref{fig::simulation_t1_01} - Figure \ref{fig::simulation_t1_02}. Nonetheless, for all challenging cases, covariate trimming drastically improves the coverage rate in all cases.

\begin{figure}[htp]
	\centering
	\begin{subfigure}{0.99\textwidth}  
          \centering
		\includegraphics[width=0.48\textwidth]{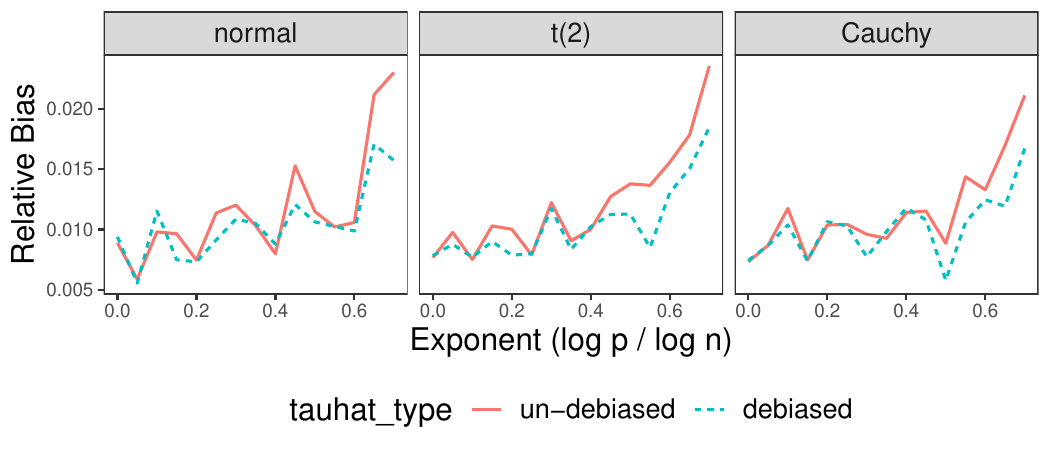}
		\includegraphics[width=0.48\textwidth]{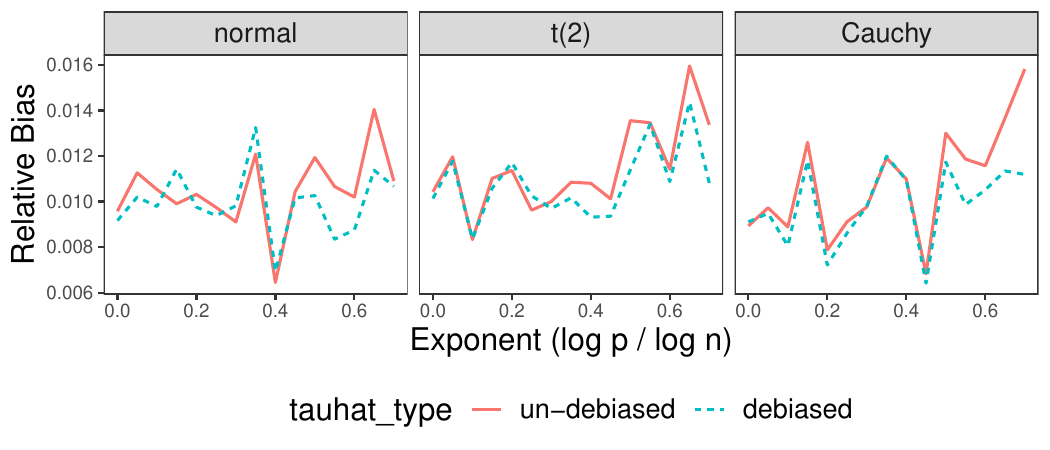}
		\caption{ Relative bias of $\hdtau$ and $\htau$.}
	\end{subfigure}
	\begin{subfigure}{0.99\textwidth}  
          \centering
		\includegraphics[width=0.48\textwidth]{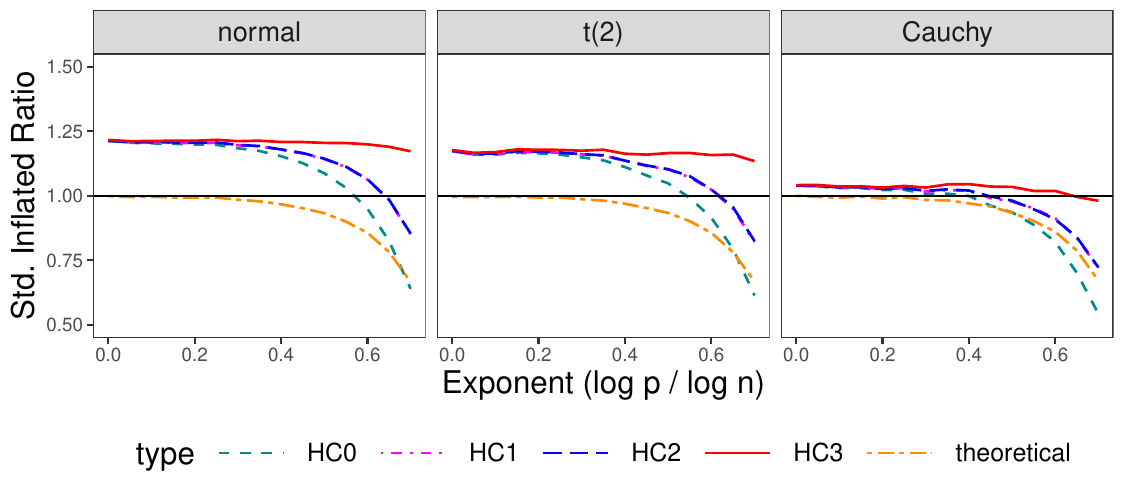}
		\includegraphics[width=0.48\textwidth]{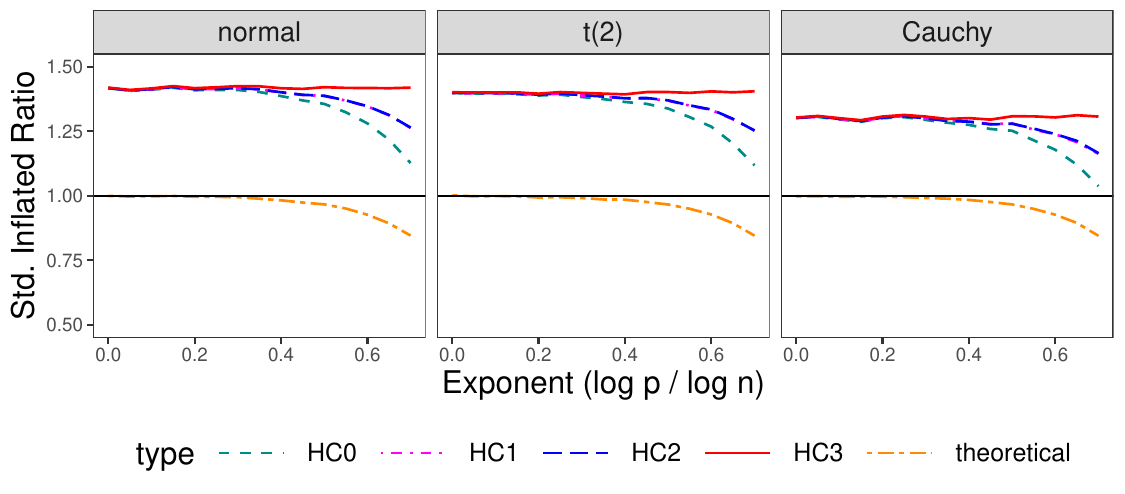}
		\caption{Ratio of standard deviation between five standard deviation estimates, $\sigma_{n}, \hat{\sigma}_{\textup{HC}0}, \hat{\sigma}_{\textup{HC}1}, \hat{\sigma}_{\textup{HC}2}, \hat{\sigma}_{\textup{HC}3}$, and the true standard deviation of $\htau$.}
	\end{subfigure}
		\begin{subfigure}{0.99\textwidth}  
                  \centering
		\includegraphics[width=0.48\textwidth]{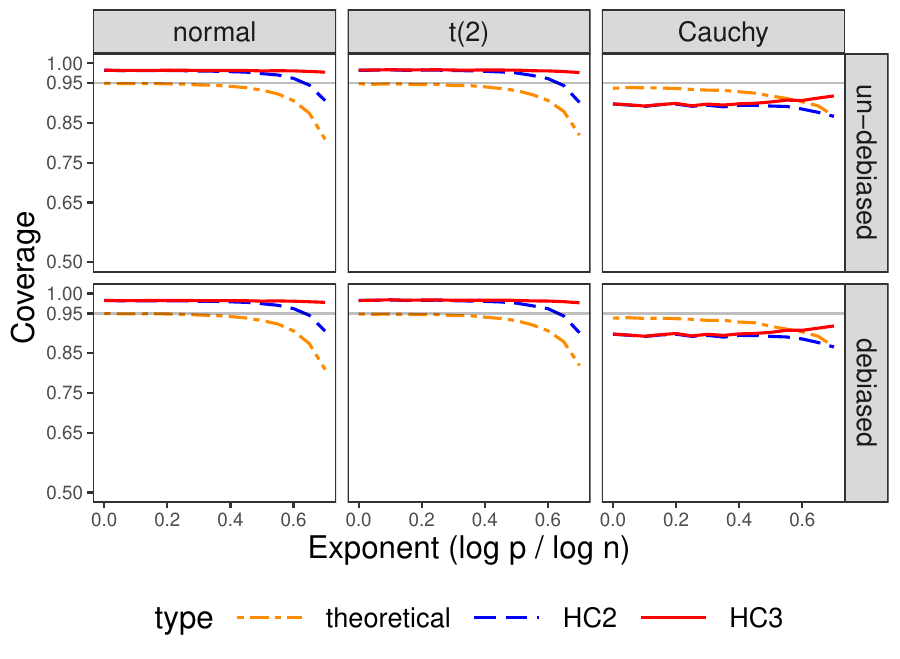}
		\includegraphics[width=0.48\textwidth]{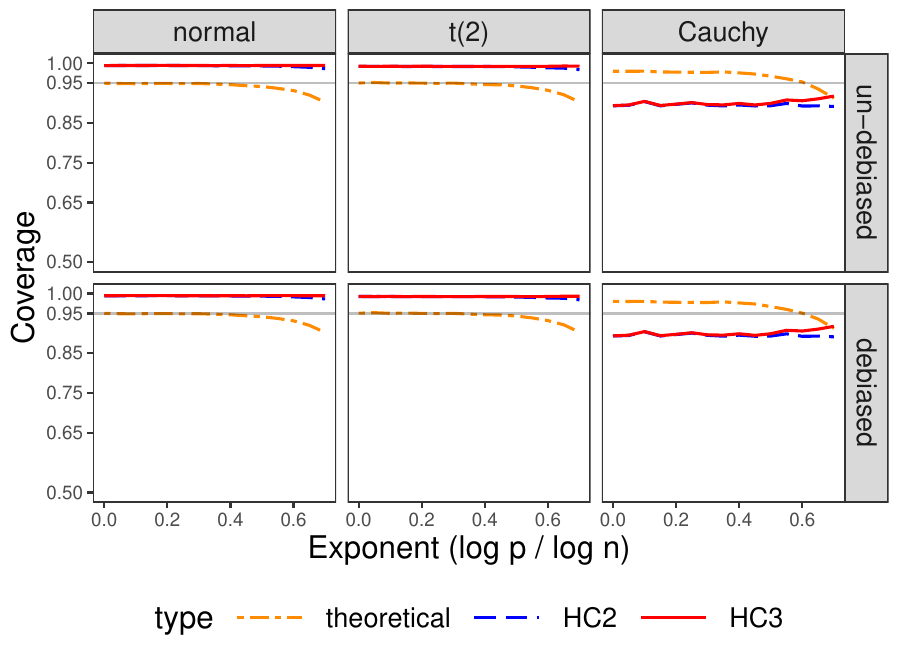}
		\caption{Empirical $95\%$ coverage rates of $t$-statistics derived from two estimators and four variance estimators (
                  ``theoretical'' for $\sigma_{n}^{2}$, ``HC2'' for $\hat{\sigma}_{\text{HC}2}^{2}$ and ``HC3''  for $\hat{\sigma}_{\text{HC}3}^{2}$)}
	\end{subfigure}
	\caption{Simulation without covariate trimming. $X$ is a realization of a random matrix with i.i.d. $N(0, 1)$ entries and $\eps(t)$ is a realization of a random vector with i.i.d. entries: (Left) $\pi_{1} = 0.2$; (Right) $\pi_{1} = 0.5$. Each column corresponds to a distribution of $\eps(t)$. }\label{fig::simulation_normal_01}
\end{figure}

\begin{figure}[htp]
	\centering
	\begin{subfigure}{0.99\textwidth}  
          \centering
		\includegraphics[width=0.48\textwidth]{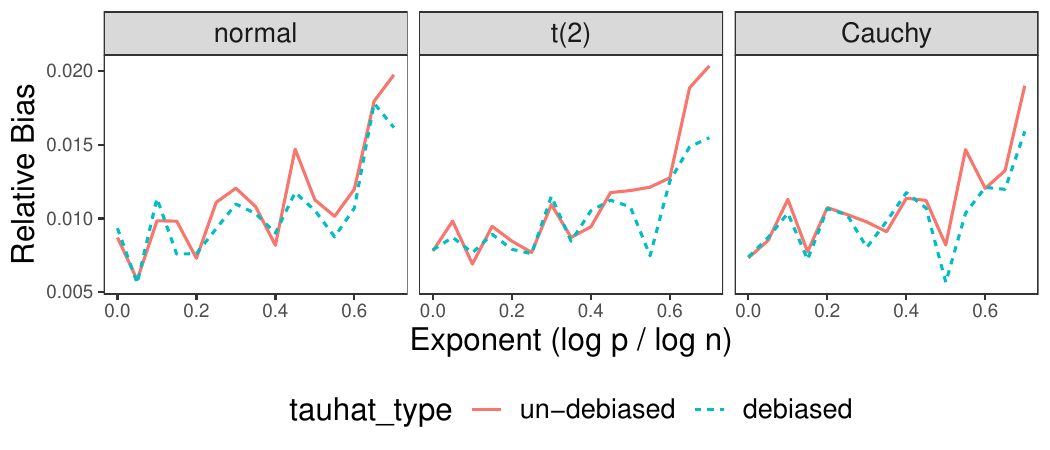}
		\includegraphics[width=0.48\textwidth]{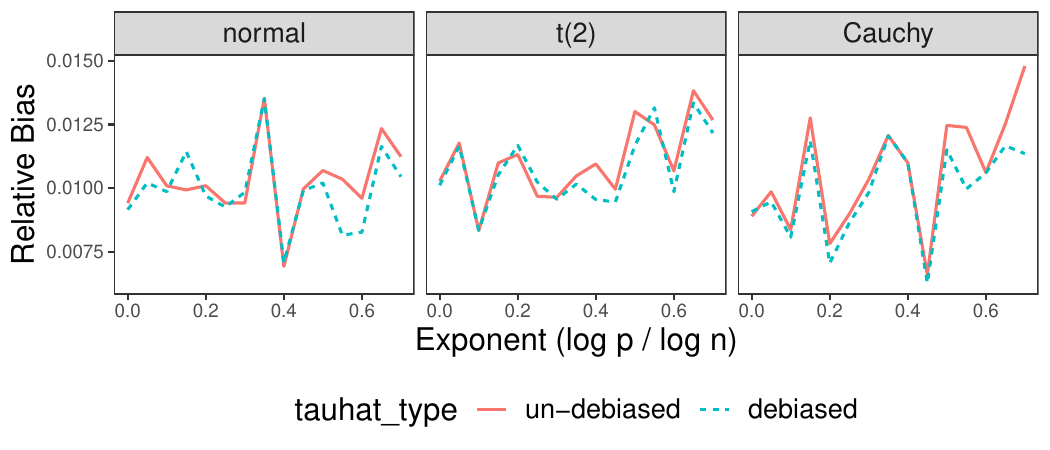}
		\caption{ Relative bias of $\hdtau$ and $\htau$.}
	\end{subfigure}
	\begin{subfigure}{0.99\textwidth}  
          \centering
		\includegraphics[width=0.48\textwidth]{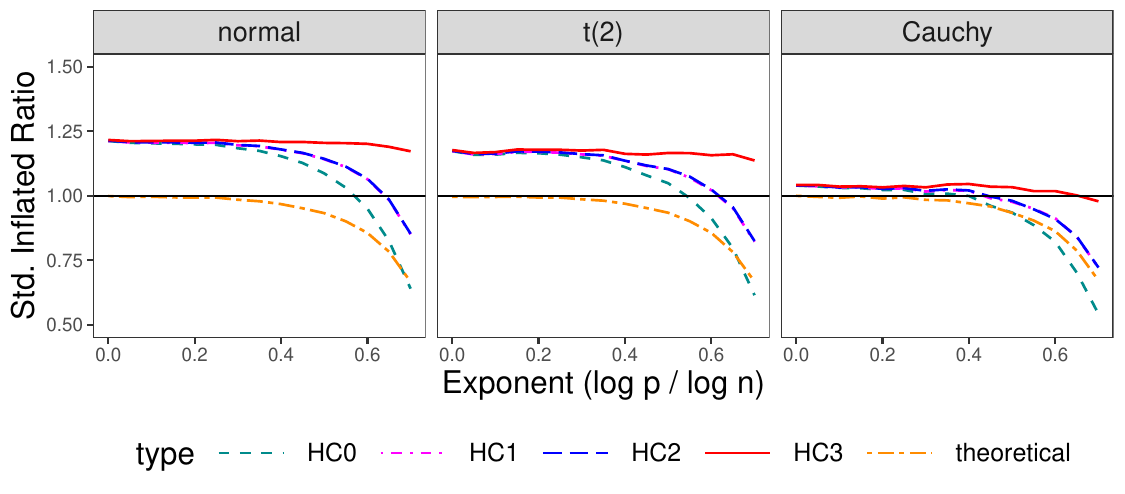}
		\includegraphics[width=0.48\textwidth]{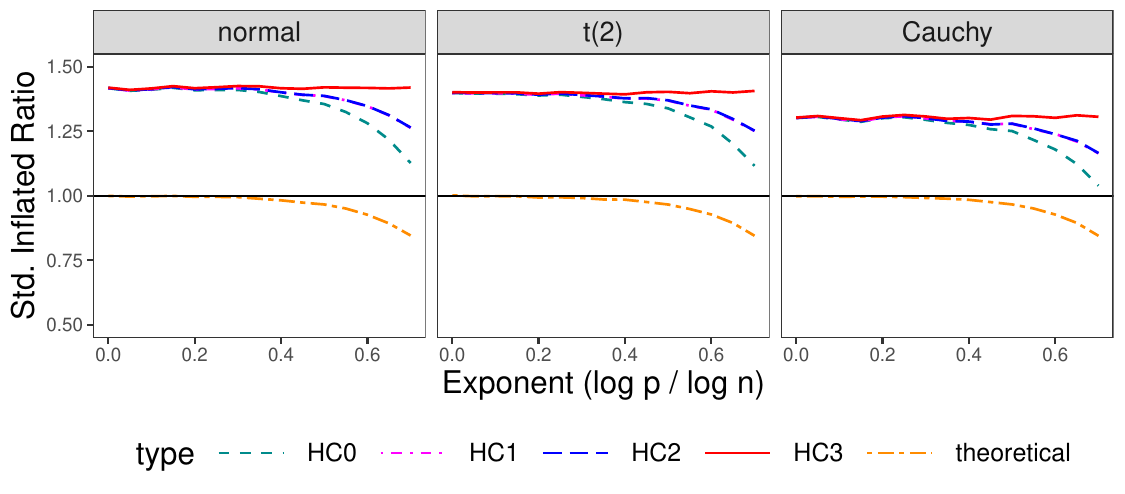}
		\caption{Ratio of standard deviation between five standard deviation estimates, $\sigma_{n}, \hat{\sigma}_{\textup{HC}0}, \hat{\sigma}_{\textup{HC}1}, \hat{\sigma}_{\textup{HC}2}, \hat{\sigma}_{\textup{HC}3}$, and the true standard deviation of $\htau$.}
	\end{subfigure}
		\begin{subfigure}{0.99\textwidth}  
                  \centering
		\includegraphics[width=0.48\textwidth]{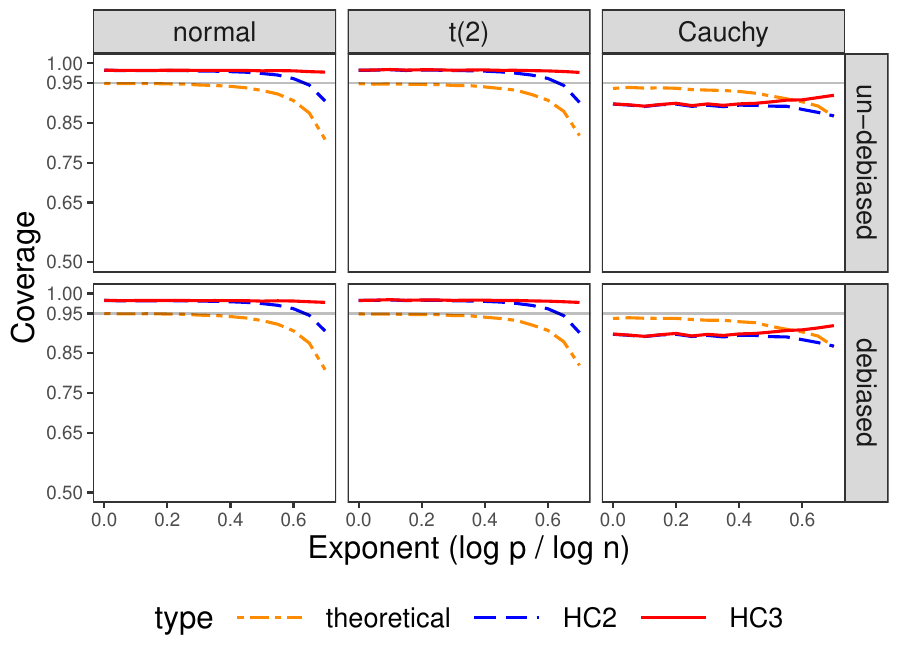}
		\includegraphics[width=0.48\textwidth]{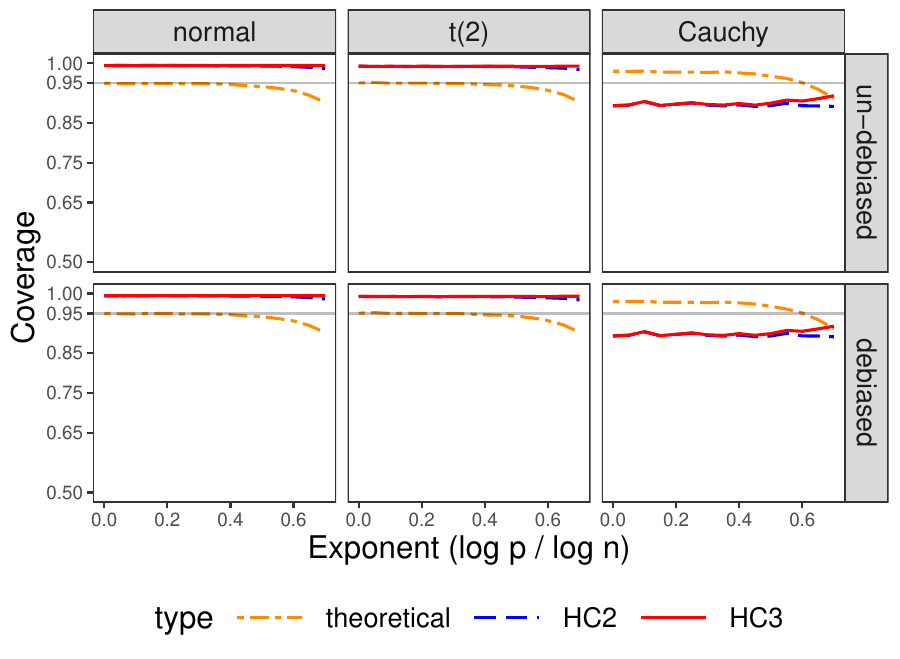}
		\caption{Empirical $95\%$ coverage rates of $t$-statistics derived from two estimators and four variance estimators (
                  ``theoretical'' for $\sigma_{n}^{2}$, ``HC2'' for $\hat{\sigma}_{\text{HC}2}^{2}$ and ``HC3''  for $\hat{\sigma}_{\text{HC}3}^{2}$)}
	\end{subfigure}
	\caption{Simulation with covariate trimming. $X$ is a realization of a random matrix with i.i.d. $N(0, 1)$ entries and $\eps(t)$ is a realization of a random vector with i.i.d. entries: (Left) $\pi_{1} = 0.2$; (Right) $\pi_{1} = 0.5$. Each column corresponds to a distribution of $\eps(t)$. }
\end{figure}

\begin{figure}
	\centering
	\begin{subfigure}{0.95\textwidth}  
          \centering
		\includegraphics[width=0.48\textwidth]{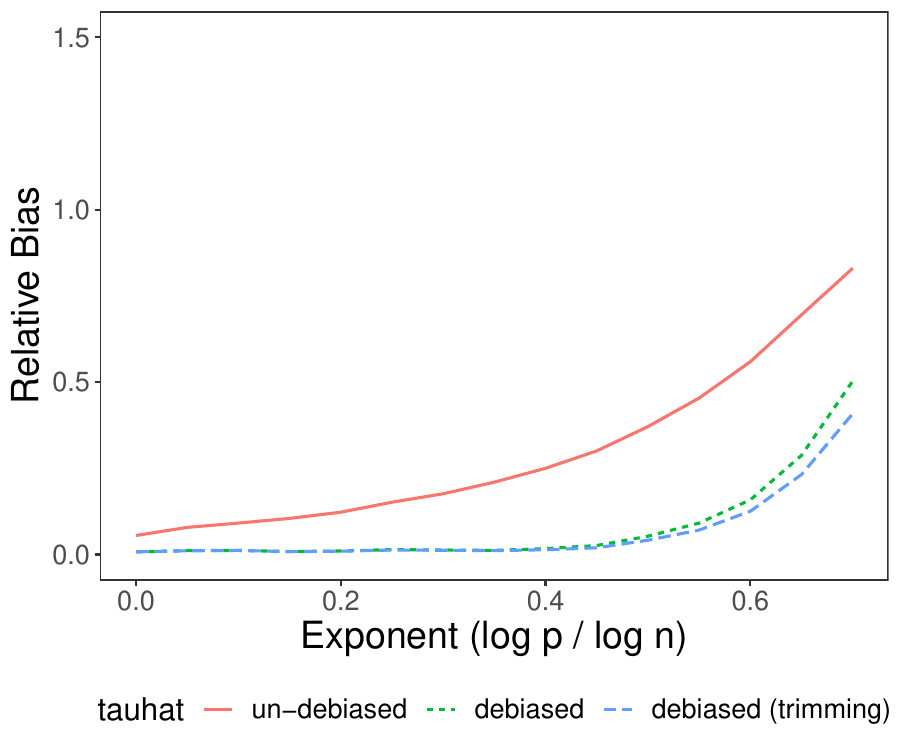}
		\includegraphics[width=0.48\textwidth]{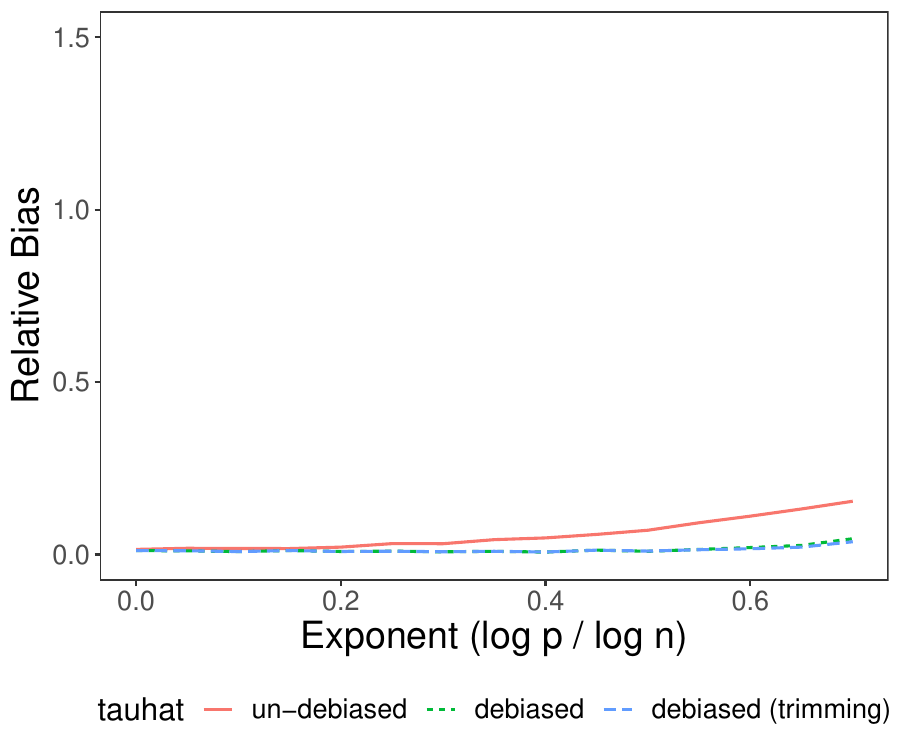}
		\caption{ Relative bias of $\hdtau$ and $\htau$.}\label{fig:normal_bias_worst}
	\end{subfigure}
		\begin{subfigure}{0.95\textwidth}  
                  \centering
		\includegraphics[width=0.48\textwidth]{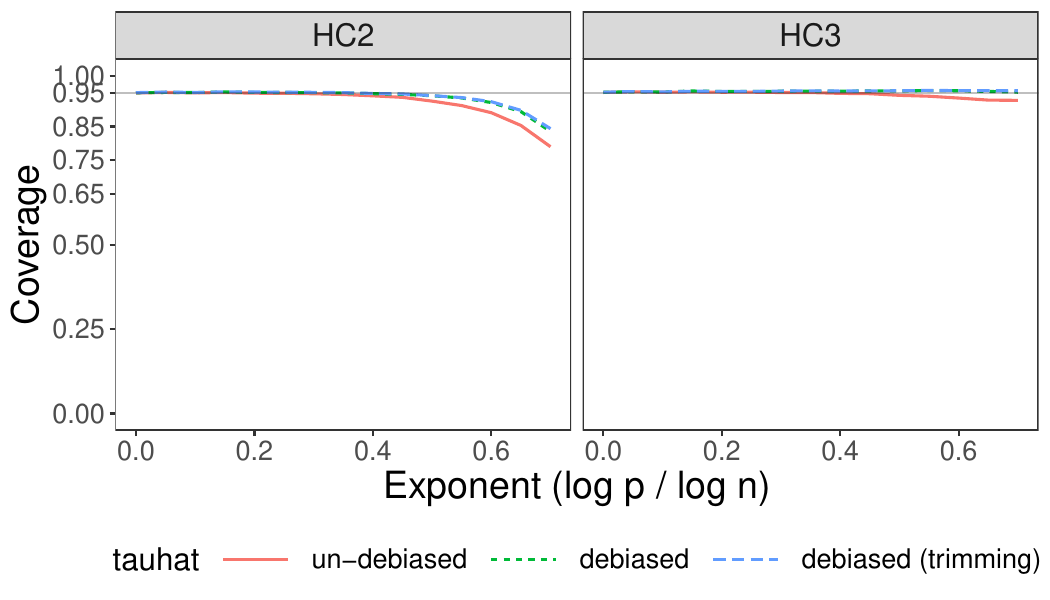}
		\includegraphics[width=0.48\textwidth]{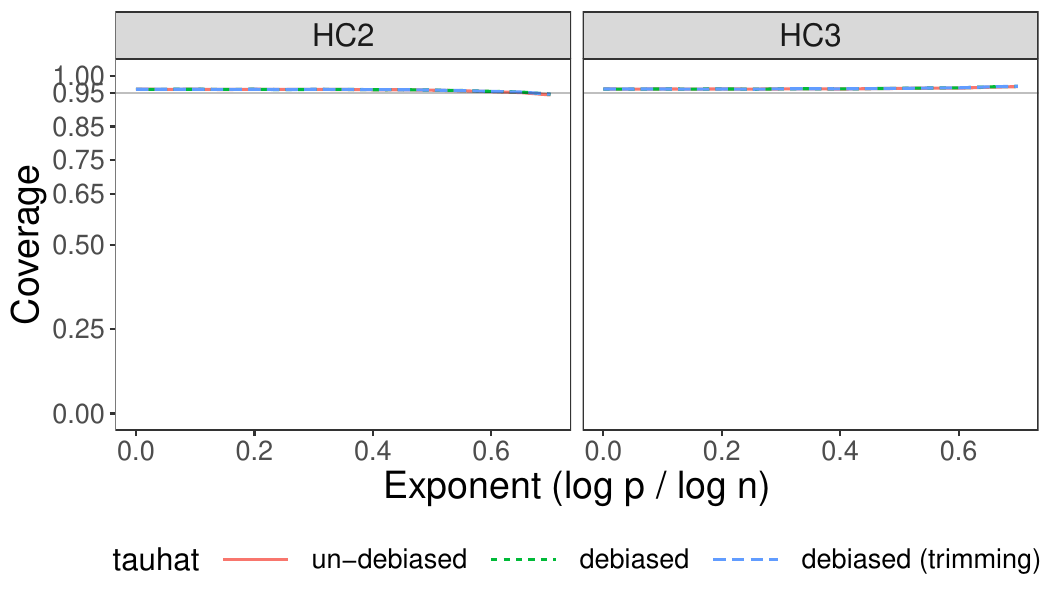}
		\caption{Empirical $95\%$ coverage rates of $t$-statistics derived from two estimators and two variance estimators (``HC2'' for $\hat{\sigma}_{\text{HC}2}^{2}$ and ``HC3''  for $\hat{\sigma}_{\text{HC}3}^{2}$)}
	\end{subfigure}
	\caption{Simulation with and without covariate trimming with $\eps(t)$ defined in \eqref{eq:worst_pout}. $X$ is a realization of a random matrix with i.i.d. $N(0, 1)$ entries: (Left) $\pi_{1} = 0.2$; (Right) $\pi_{1} = 0.5$. }\label{fig::simulation_normal_02}
\end{figure}

\begin{figure}[htp]
	\centering
	\begin{subfigure}{0.99\textwidth}  
          \centering
		\includegraphics[width=0.48\textwidth]{{synthetic-simul-t2-pi0.2-rho0-bias}.pdf}
		\includegraphics[width=0.48\textwidth]{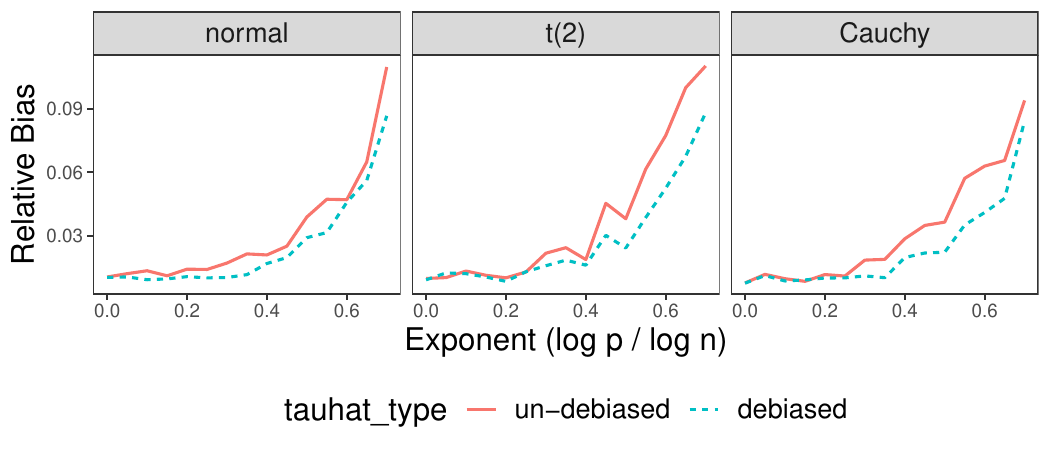}
		\caption{ Relative bias of $\hdtau$ and $\htau$.}
	\end{subfigure}
	\begin{subfigure}{0.99\textwidth}  
          \centering
		\includegraphics[width=0.48\textwidth]{{synthetic-simul-t2-pi0.2-rho0-sdinflate}.pdf}
		\includegraphics[width=0.48\textwidth]{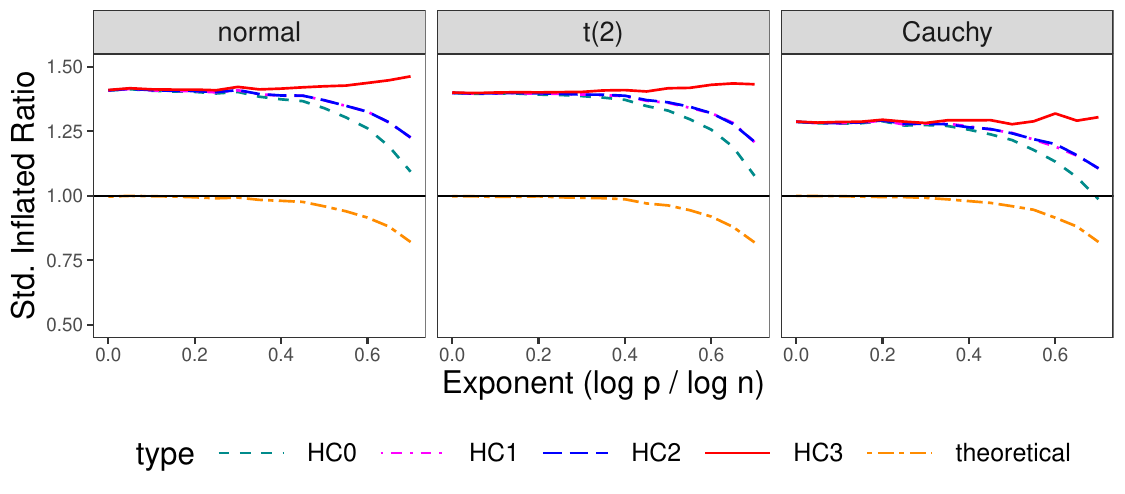}
		\caption{Ratio of standard deviation between five standard deviation estimates, $\sigma_{n}, \hat{\sigma}_{\textup{HC}0}, \hat{\sigma}_{\textup{HC}1}, \hat{\sigma}_{\textup{HC}2}, \hat{\sigma}_{\textup{HC}3}$, and the true standard deviation of $\htau$.}
	\end{subfigure}
		\begin{subfigure}{0.99\textwidth}  
                  \centering
		\includegraphics[width=0.48\textwidth]{{synthetic-simul-t2-pi0.2-rho0-coverage}.pdf}
		\includegraphics[width=0.48\textwidth]{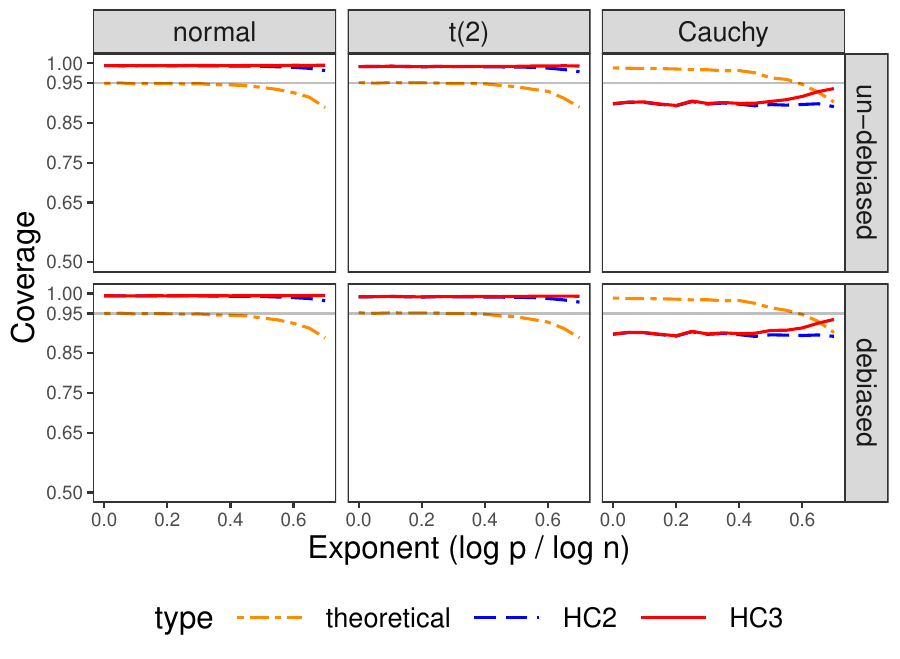}
		\caption{Empirical $95\%$ coverage rates of $t$-statistics derived from two estimators and four variance estimators (
                  ``theoretical'' for $\sigma_{n}^{2}$, ``HC2'' for $\hat{\sigma}_{\text{HC}2}^{2}$ and ``HC3''  for $\hat{\sigma}_{\text{HC}3}^{2}$)}
	\end{subfigure}
	\caption{Simulation without covariate trimming. $X$ is a realization of a random matrix with i.i.d. $t(2)$ entries and $\eps(t)$ is a realization of a random vector with i.i.d. entries: (Left) $\pi_{1} = 0.2$; (Right) $\pi_{1} = 0.5$. Each column corresponds to a distribution of $\eps(t)$. }\label{fig::simulation_t2_01}
\end{figure}

\begin{figure}[htp]
	\centering
	\begin{subfigure}{0.99\textwidth}  
          \centering
		\includegraphics[width=0.48\textwidth]{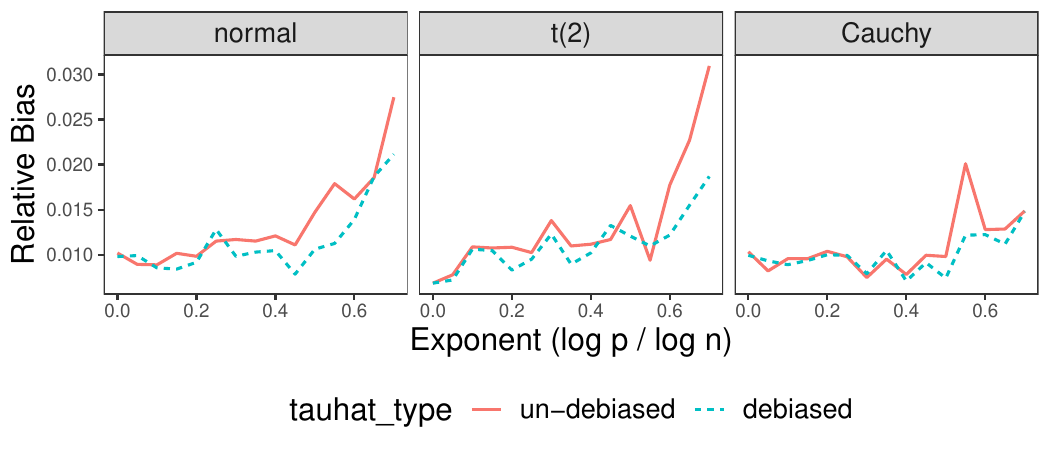}
		\includegraphics[width=0.48\textwidth]{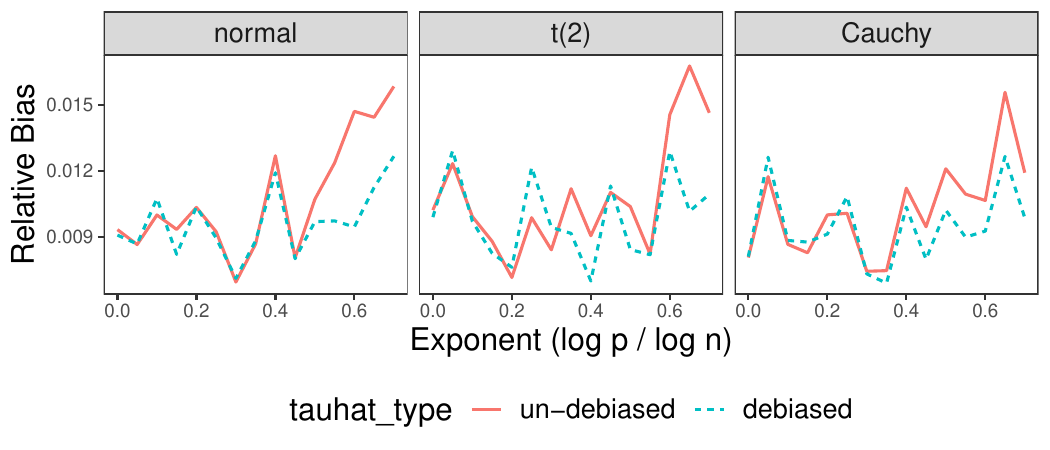}
		\caption{ Relative bias of $\hdtau$ and $\htau$.}
	\end{subfigure}
	\begin{subfigure}{0.99\textwidth}  
          \centering
		\includegraphics[width=0.48\textwidth]{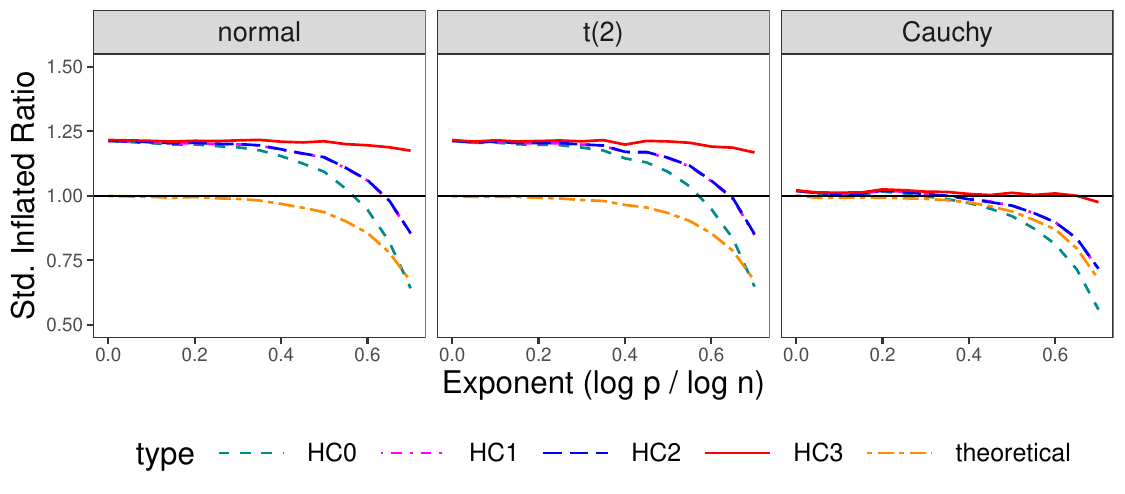}
		\includegraphics[width=0.48\textwidth]{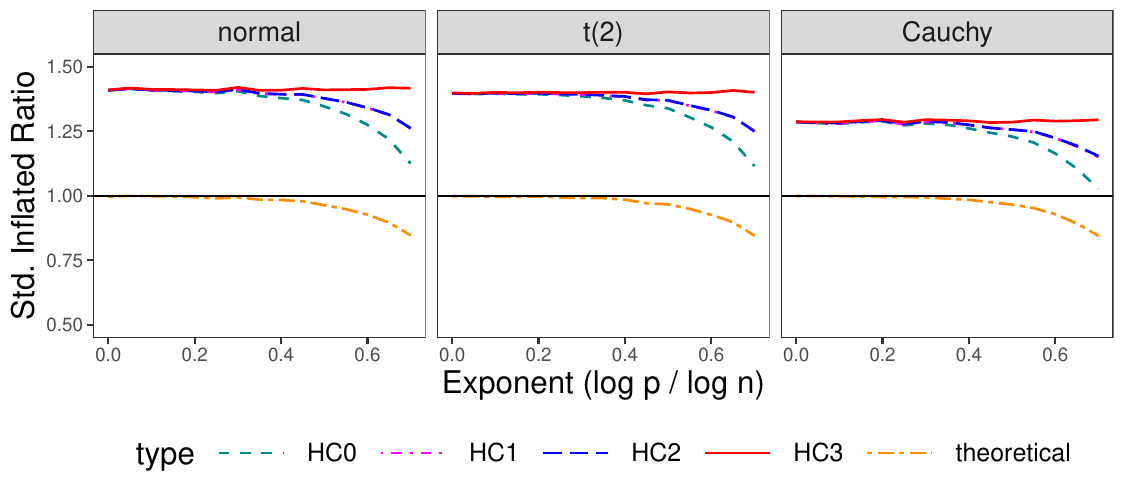}
		\caption{Ratio of standard deviation between five standard deviation estimates, $\sigma_{n}, \hat{\sigma}_{\textup{HC}0}, \hat{\sigma}_{\textup{HC}1}, \hat{\sigma}_{\textup{HC}2}, \hat{\sigma}_{\textup{HC}3}$, and the true standard deviation of $\htau$.}
	\end{subfigure}
		\begin{subfigure}{0.99\textwidth}  
                  \centering
		\includegraphics[width=0.48\textwidth]{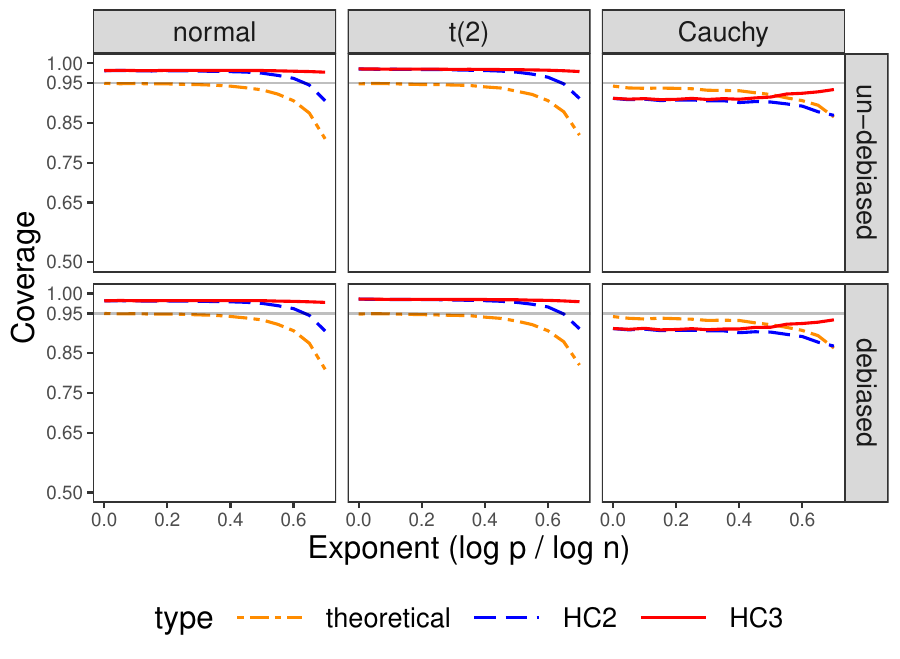}
		\includegraphics[width=0.48\textwidth]{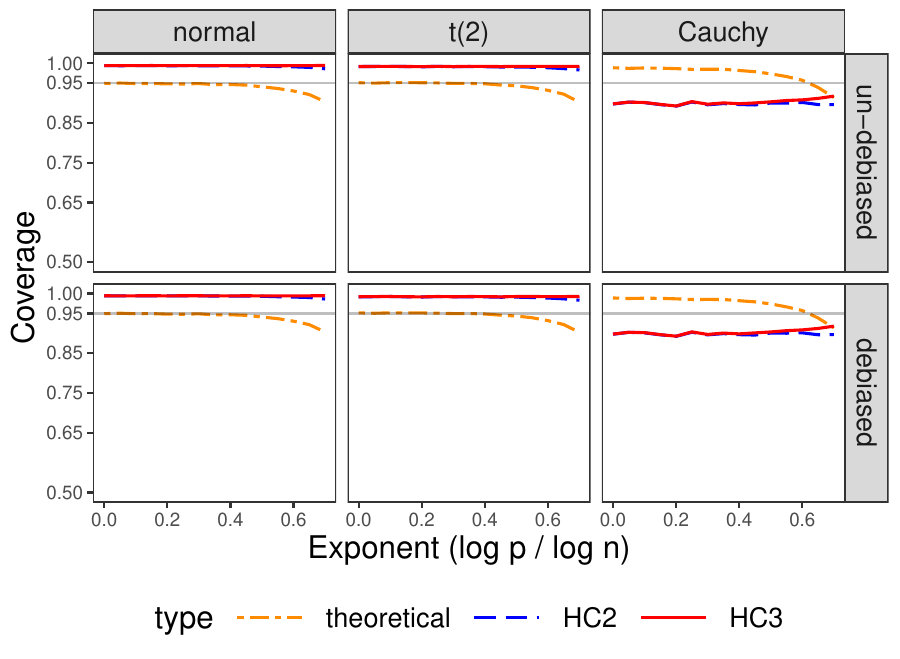}
		\caption{Empirical $95\%$ coverage rates of $t$-statistics derived from two estimators and four variance estimators (
                  ``theoretical'' for $\sigma_{n}^{2}$, ``HC2'' for $\hat{\sigma}_{\text{HC}2}^{2}$ and ``HC3''  for $\hat{\sigma}_{\text{HC}3}^{2}$)}
	\end{subfigure}
	\caption{Simulation with covariate trimming. $X$ is a realization of a random matrix with i.i.d. $t(2)$ entries and $\eps(t)$ is a realization of a random vector with i.i.d. entries: (Left) $\pi_{1} = 0.2$; (Right) $\pi_{1} = 0.5$. Each column corresponds to a distribution of $\eps(t)$. }
\end{figure}

\begin{figure}
	\centering
	\begin{subfigure}{0.95\textwidth}  
          \centering
		\includegraphics[width=0.48\textwidth]{{synthetic-simul-worst-t2-pi0.2-rho0-bias}.pdf}
		\includegraphics[width=0.48\textwidth]{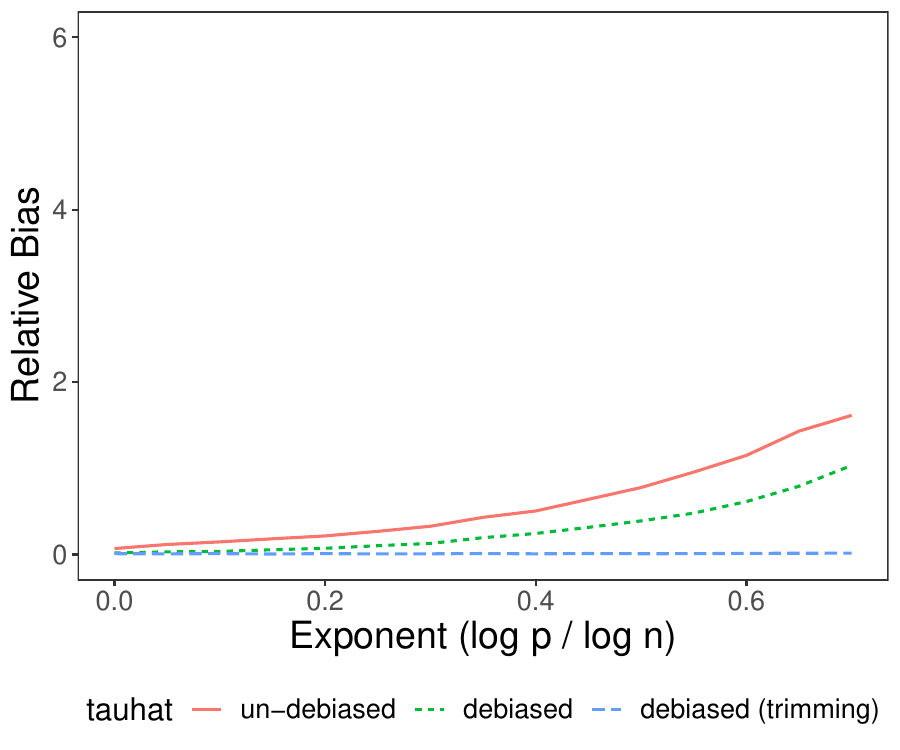}
		\caption{ Relative bias of $\hdtau$ and $\htau$.}\label{fig:t2_bias_worst}
	\end{subfigure}
		\begin{subfigure}{0.95\textwidth}  
                  \centering
		\includegraphics[width=0.48\textwidth]{{synthetic-simul-worst-t2-pi0.2-rho0-coverage}.pdf}
		\includegraphics[width=0.48\textwidth]{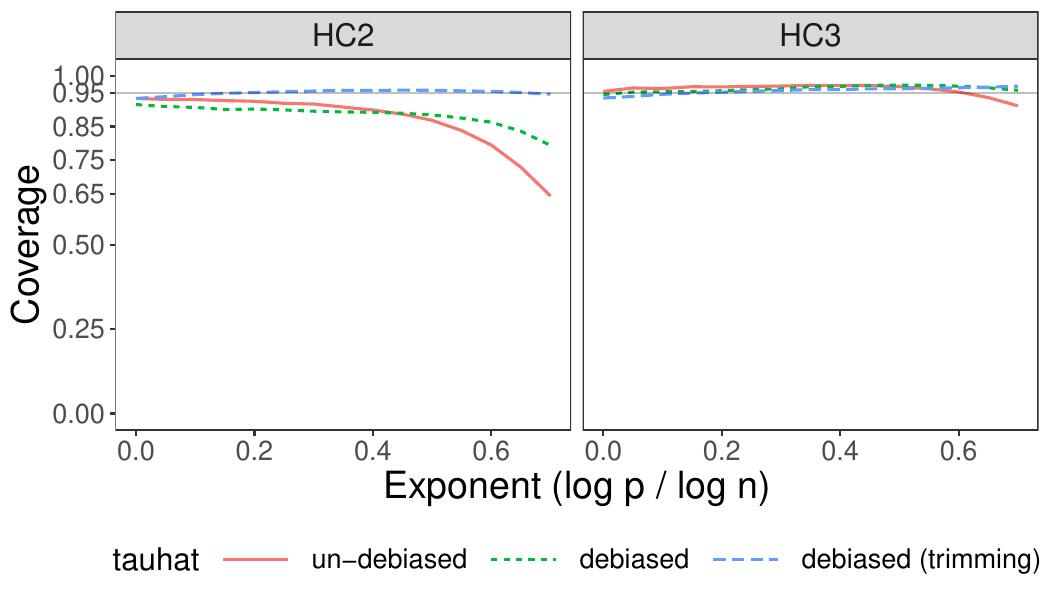}
		\caption{Empirical $95\%$ coverage rates of $t$-statistics derived from two estimators and two variance estimators (``HC2'' for $\hat{\sigma}_{\text{HC}2}^{2}$ and ``HC3''  for $\hat{\sigma}_{\text{HC}3}^{2}$)}
	\end{subfigure}
	\caption{Simulation with and without covariate trimming with $\eps(t)$ defined in \eqref{eq:worst_pout}. $X$ is a realization of a random matrix with i.i.d. $t(2)$ entries: (Left) $\pi_{1} = 0.2$; (Right) $\pi_{1} = 0.5$. }\label{fig::simulation_t2_02}
      \end{figure}

\begin{figure}[htp]
	\centering
	\begin{subfigure}{0.99\textwidth}  
          \centering
		\includegraphics[width=0.48\textwidth]{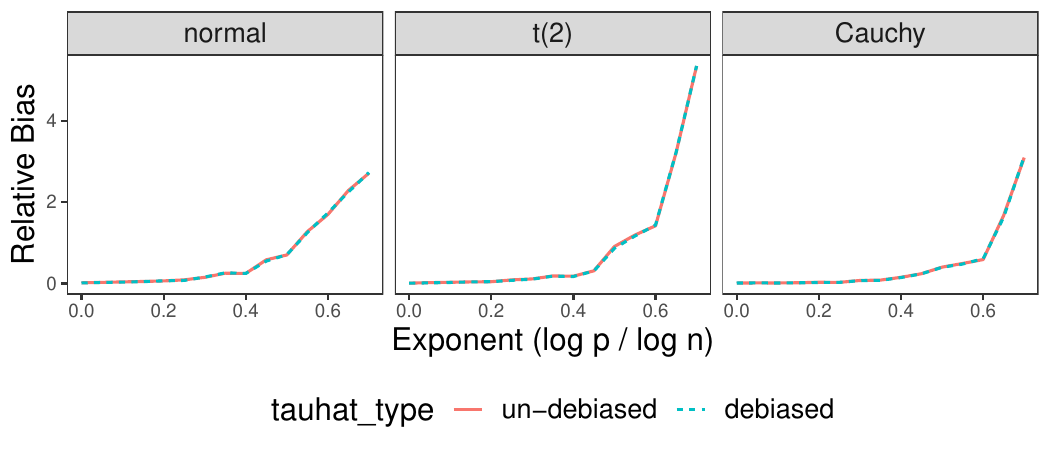}
		\includegraphics[width=0.48\textwidth]{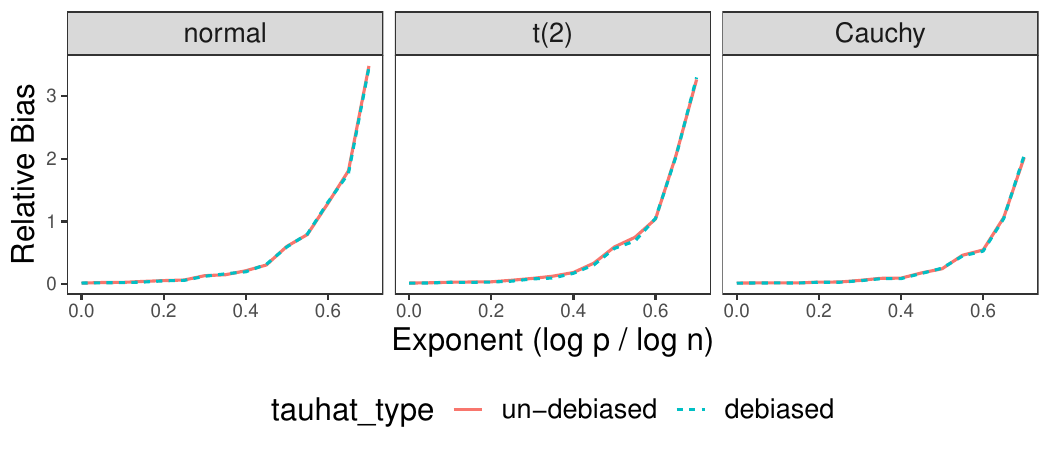}
		\caption{ Relative bias of $\hdtau$ and $\htau$.}
	\end{subfigure}
	\begin{subfigure}{0.99\textwidth}  
          \centering
		\includegraphics[width=0.48\textwidth]{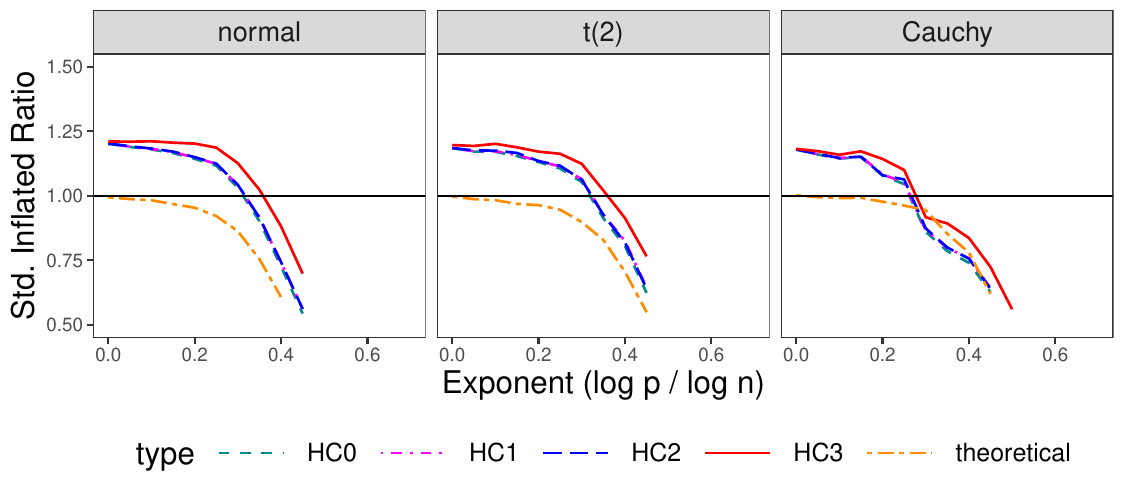}
		\includegraphics[width=0.48\textwidth]{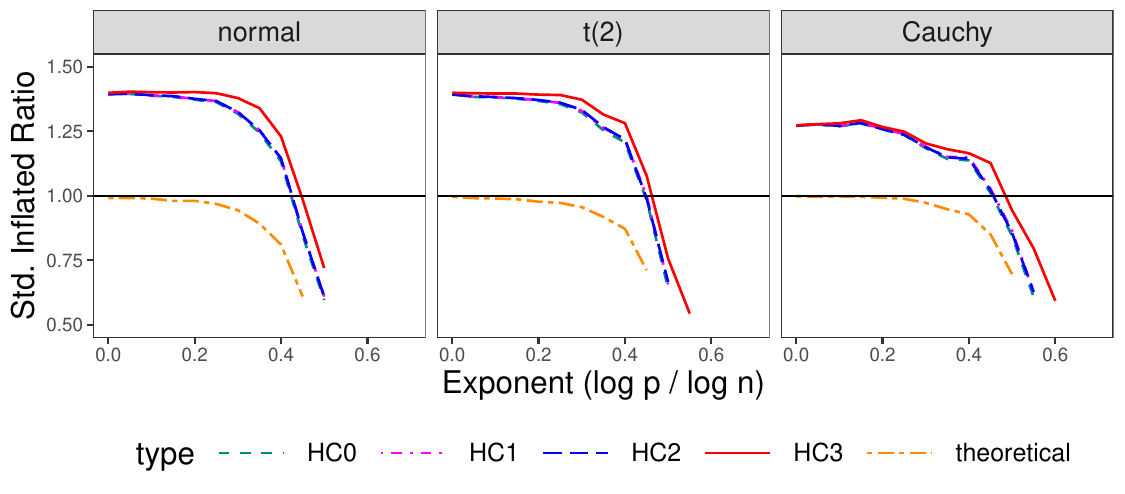}
		\caption{Ratio of standard deviation between five standard deviation estimates, $\sigma_{n}, \hat{\sigma}_{\textup{HC}0}, \hat{\sigma}_{\textup{HC}1}, \hat{\sigma}_{\textup{HC}2}, \hat{\sigma}_{\textup{HC}3}$, and the true standard deviation of $\htau$.}
	\end{subfigure}
		\begin{subfigure}{0.99\textwidth}  
                  \centering
		\includegraphics[width=0.48\textwidth]{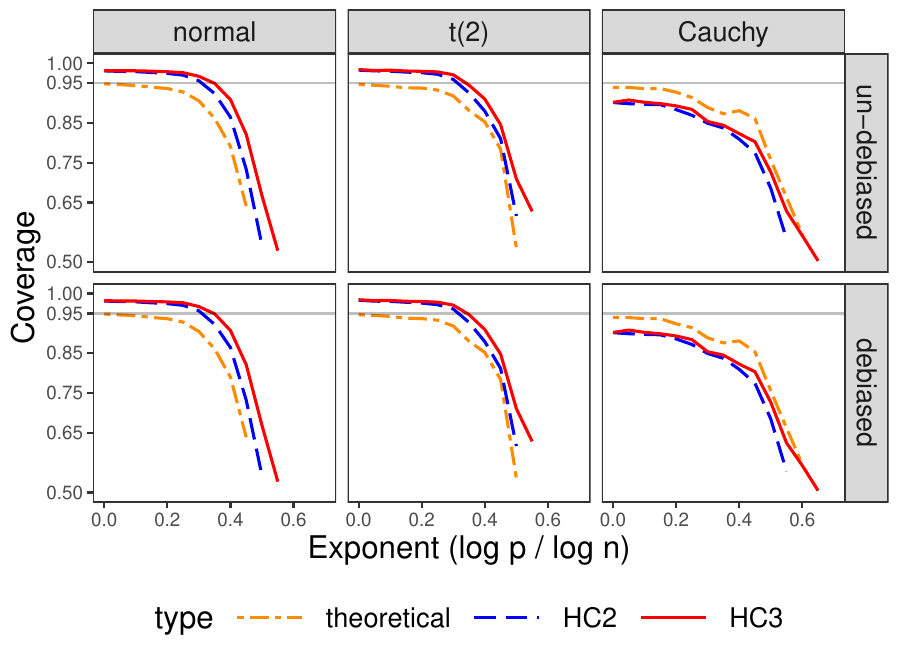}
		\includegraphics[width=0.48\textwidth]{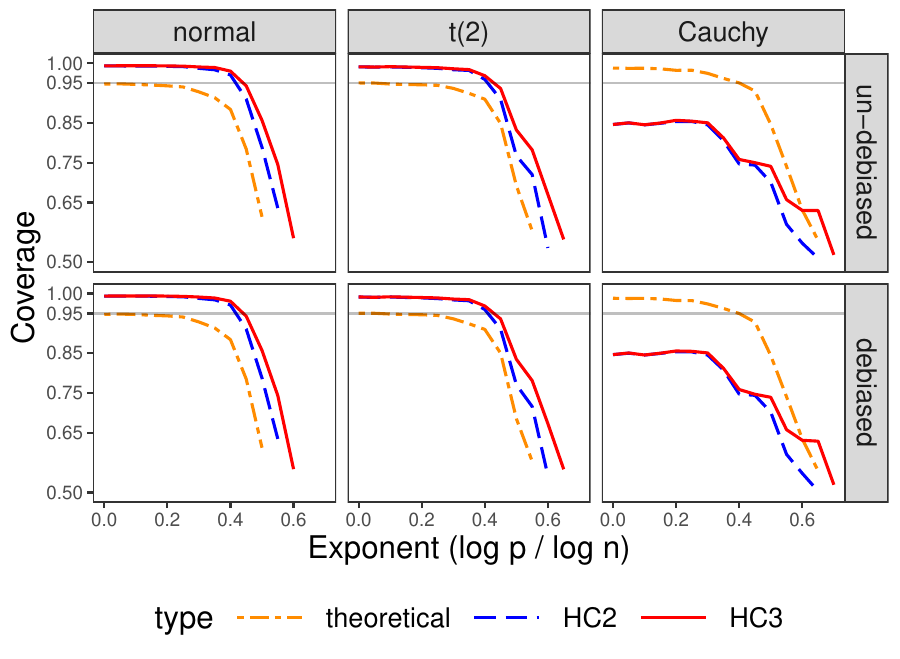}
		\caption{Empirical $95\%$ coverage rates of $t$-statistics derived from two estimators and four variance estimators (
                  ``theoretical'' for $\sigma_{n}^{2}$, ``HC2'' for $\hat{\sigma}_{\text{HC}2}^{2}$ and ``HC3''  for $\hat{\sigma}_{\text{HC}3}^{2}$)}
	\end{subfigure}
	\caption{Simulation without covariate trimming. $X$ is a realization of a random matrix with i.i.d. $t(1)$ entries and $\eps(t)$ is a realization of a random vector with i.i.d. entries: (Left) $\pi_{1} = 0.2$; (Right) $\pi_{1} = 0.5$. Each column corresponds to a distribution of $\eps(t)$. }\label{fig::simulation_t1_01}
\end{figure}

\begin{figure}[htp]
	\centering
	\begin{subfigure}{0.99\textwidth}  
          \centering
		\includegraphics[width=0.48\textwidth]{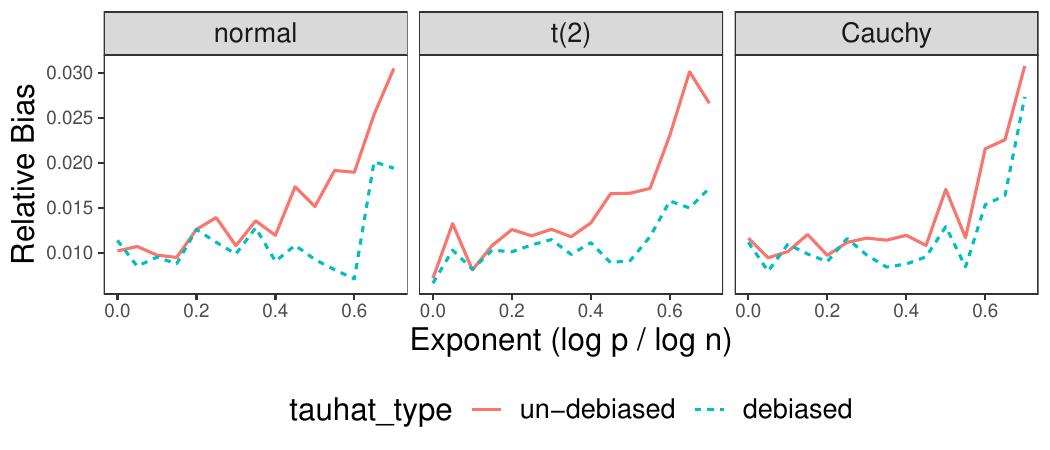}
		\includegraphics[width=0.48\textwidth]{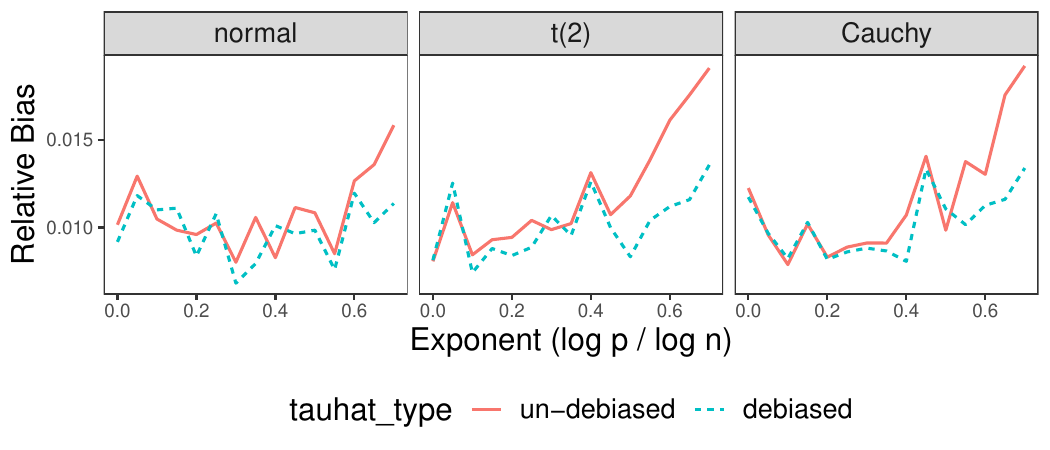}
		\caption{ Relative bias of $\hdtau$ and $\htau$.}
	\end{subfigure}
	\begin{subfigure}{0.99\textwidth}  
          \centering
		\includegraphics[width=0.48\textwidth]{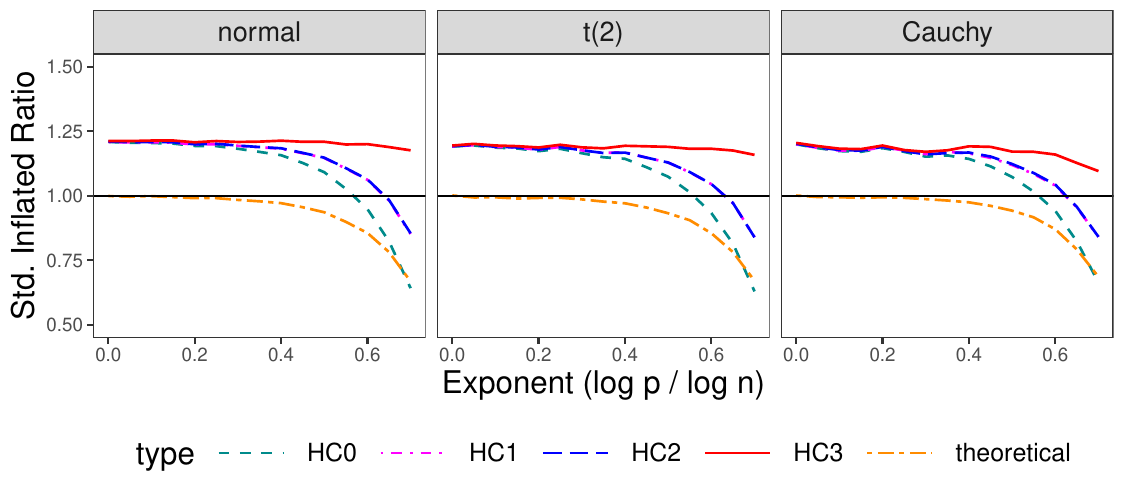}
		\includegraphics[width=0.48\textwidth]{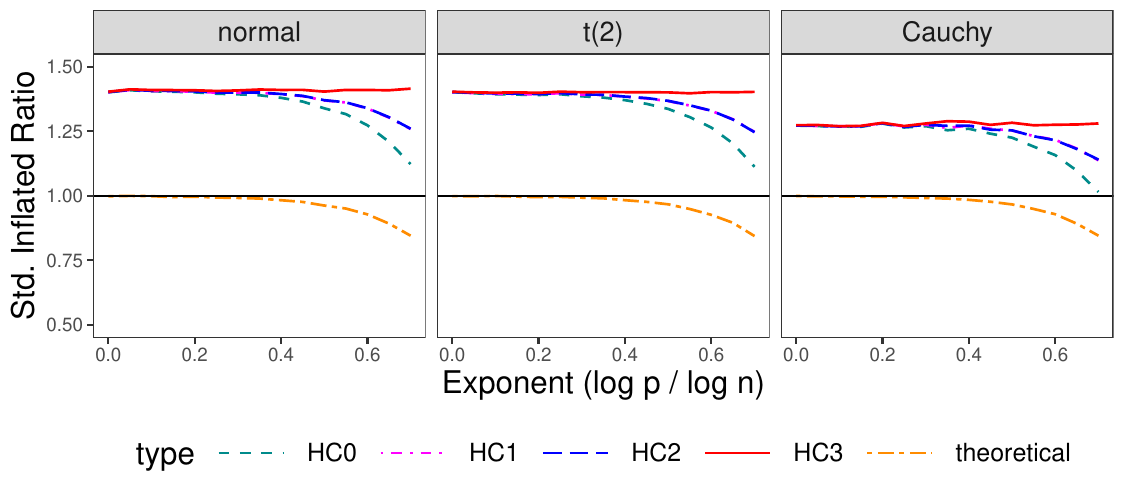}
		\caption{Ratio of standard deviation between five standard deviation estimates, $\sigma_{n}, \hat{\sigma}_{\textup{HC}0}, \hat{\sigma}_{\textup{HC}1}, \hat{\sigma}_{\textup{HC}2}, \hat{\sigma}_{\textup{HC}3}$, and the true standard deviation of $\htau$.}
	\end{subfigure}
		\begin{subfigure}{0.99\textwidth}  
                  \centering
		\includegraphics[width=0.48\textwidth]{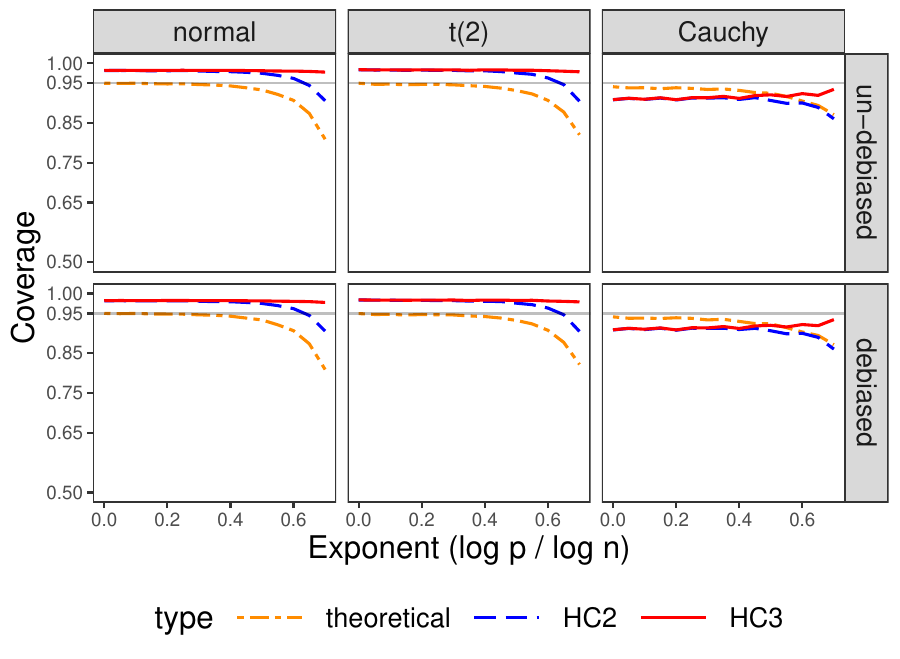}
		\includegraphics[width=0.48\textwidth]{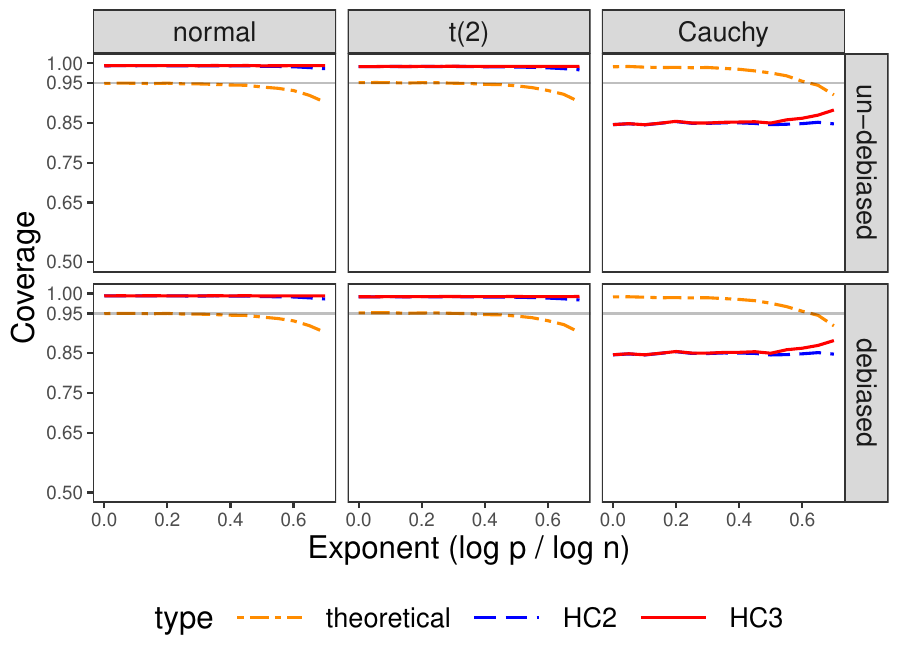}
		\caption{Empirical $95\%$ coverage rates of $t$-statistics derived from two estimators and four variance estimators (
                  ``theoretical'' for $\sigma_{n}^{2}$, ``HC2'' for $\hat{\sigma}_{\text{HC}2}^{2}$ and ``HC3''  for $\hat{\sigma}_{\text{HC}3}^{2}$)}
	\end{subfigure}
	\caption{Simulation with covariate trimming. $X$ is a realization of a random matrix with i.i.d. $t(1)$ entries and $\eps(t)$ is a realization of a random vector with i.i.d. entries: (Left) $\pi_{1} = 0.2$; (Right) $\pi_{1} = 0.5$. Each column corresponds to a distribution of $\eps(t)$. }
\end{figure}

\begin{figure}
	\centering
	\begin{subfigure}{0.95\textwidth}  
          \centering
		\includegraphics[width=0.48\textwidth]{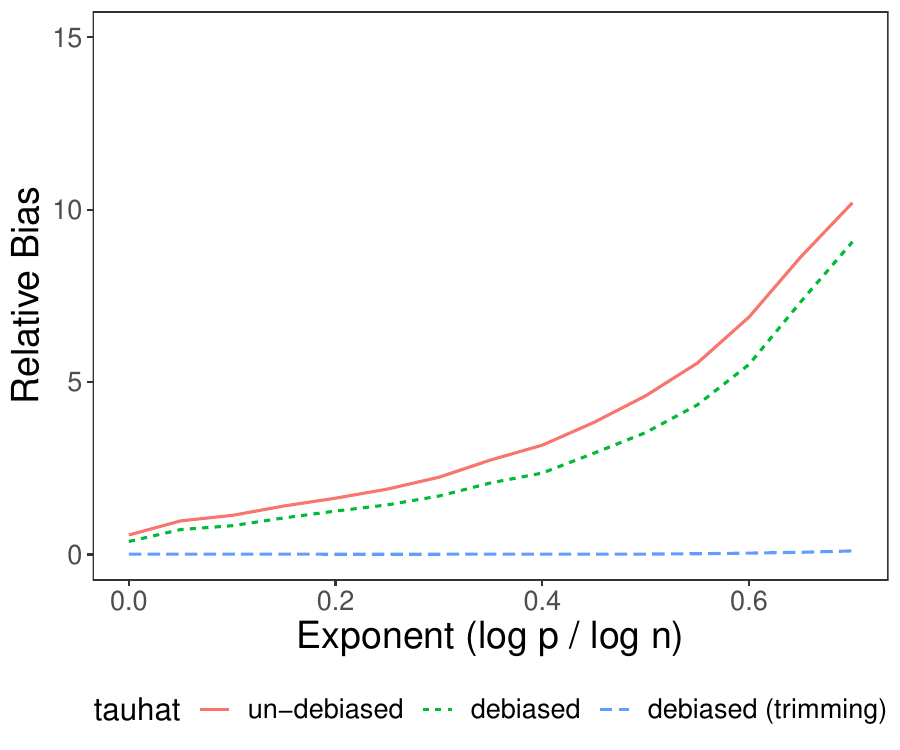}
		\includegraphics[width=0.48\textwidth]{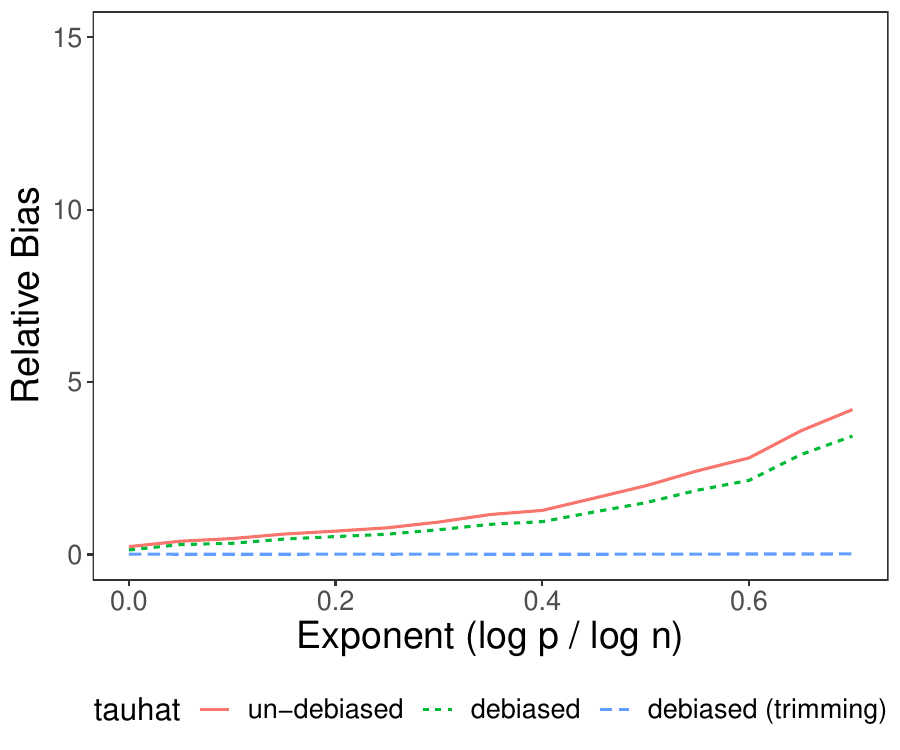}
		\caption{ Relative bias of $\hdtau$ and $\htau$.}\label{fig:t1_bias_worst}
	\end{subfigure}
		\begin{subfigure}{0.95\textwidth}  
                  \centering
		\includegraphics[width=0.48\textwidth]{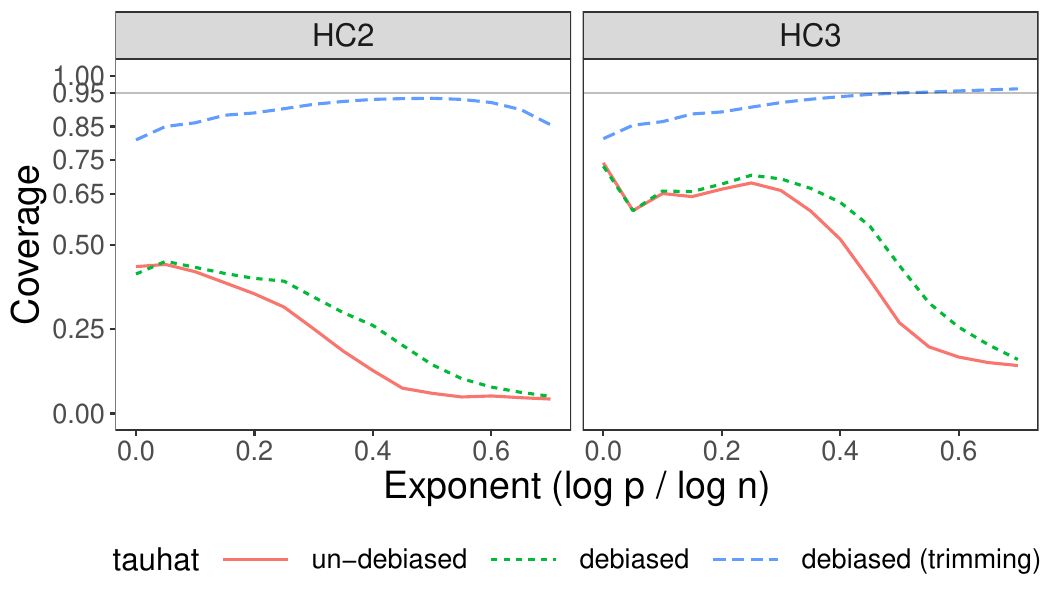}
		\includegraphics[width=0.48\textwidth]{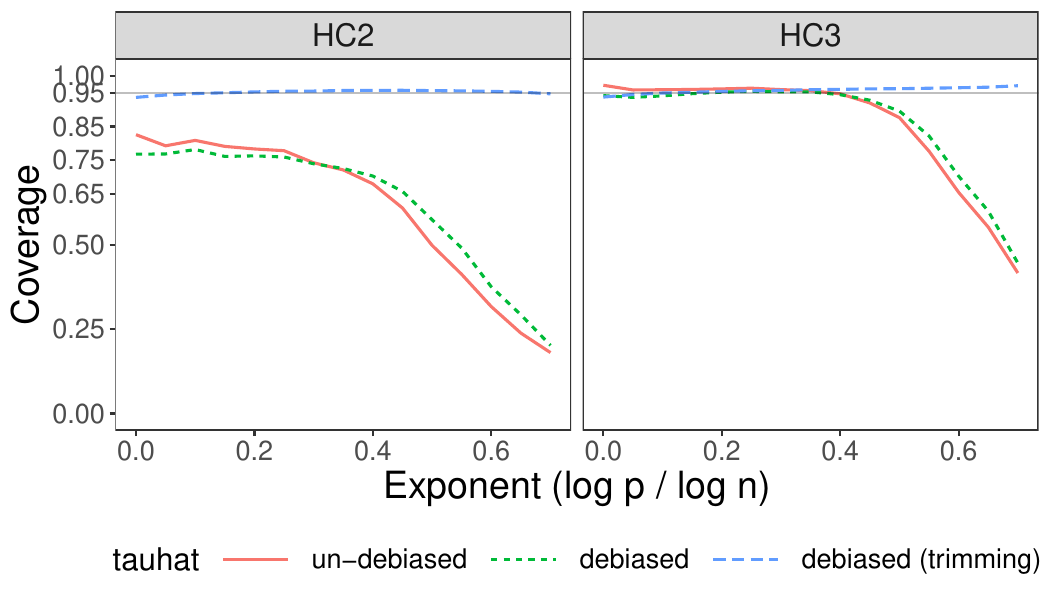}
		\caption{Empirical $95\%$ coverage rates of $t$-statistics derived from two estimators and two variance estimators (``HC2'' for $\hat{\sigma}_{\text{HC}2}^{2}$ and ``HC3''  for $\hat{\sigma}_{\text{HC}3}^{2}$)}
	\end{subfigure}
	\caption{Simulation with and without covariate trimming with $\eps(t)$ defined in \eqref{eq:worst_pout}. $X$ is a realization of a random matrix with i.i.d. $t(1)$ entries: (Left) $\pi_{1} = 0.2$; (Right) $\pi_{1} = 0.5$. }\label{fig::simulation_t1_02}
\end{figure}      

\clearpage

\subsection{Experimental results on synthetic datasets with sharp nulls}
For each setting of $X$ and $\pi_1$, we also consider the sharp null with $\eps(1) = \eps(0)$. Specifically, we generate $\eps(1)$ with i.i.d. entries from $N(0, 1)$ or $t(2)$ or $t(1)$. The latter is more challenging in terms of the variance estimation and coverage rate because the term $S_{\tau}^2 = 0$ in the asymptotic variance. The results are presented in Figs. \ref{fig::simulation_normal_01_sharp}--\ref{fig::simulation_t1_02_sharp}. We observe that all results are qualitatively similar to those for the cases in the last subsection.

\begin{figure}[htp]
	\centering
	\begin{subfigure}{0.99\textwidth}  
          \centering
		\includegraphics[width=0.48\textwidth]{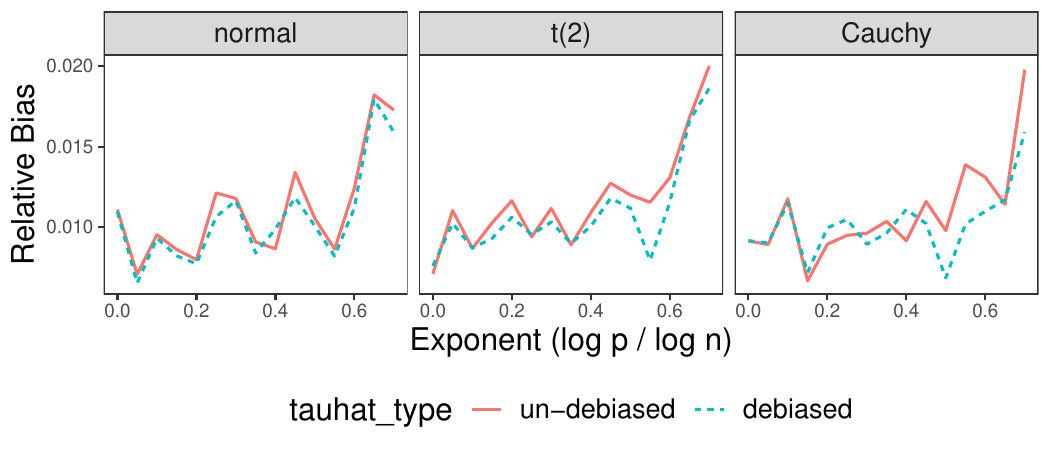}
		\includegraphics[width=0.48\textwidth]{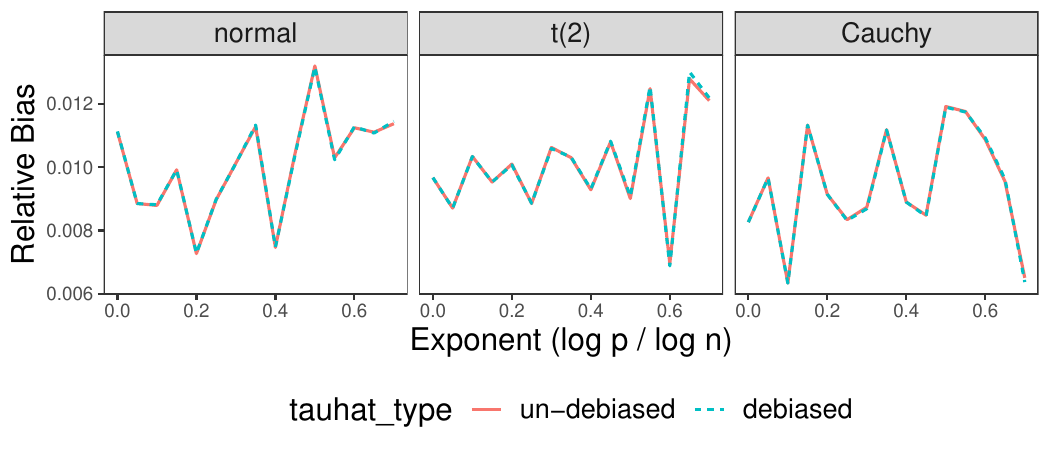}
		\caption{ Relative bias of $\hdtau$ and $\htau$.}
	\end{subfigure}
	\begin{subfigure}{0.99\textwidth}  
          \centering
		\includegraphics[width=0.48\textwidth]{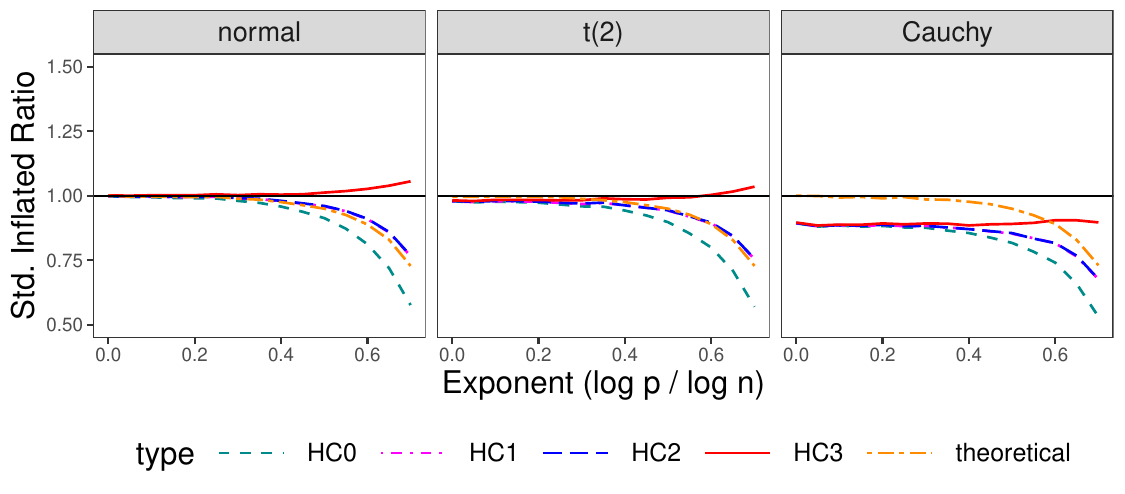}
		\includegraphics[width=0.48\textwidth]{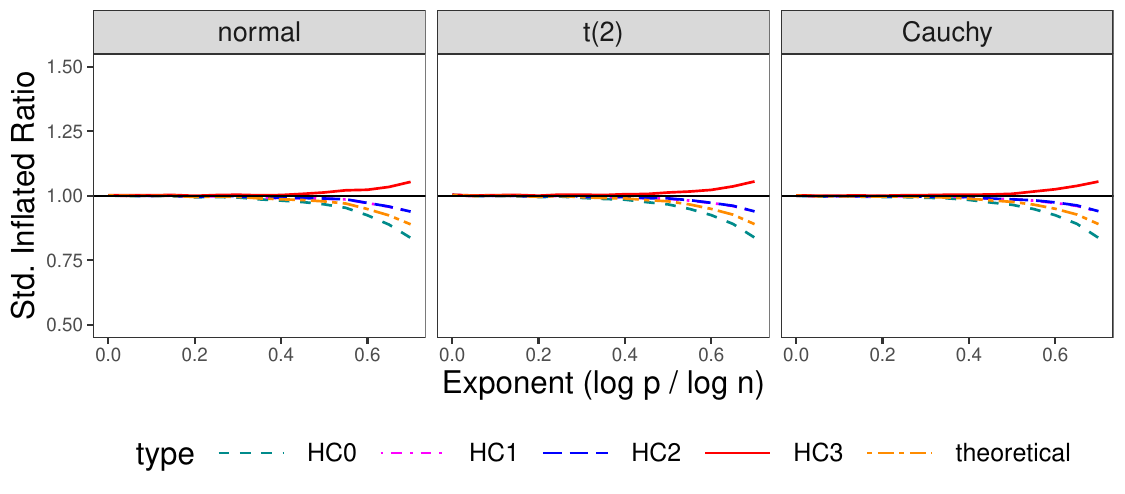}
		\caption{Ratio of standard deviation between five standard deviation estimates, $\sigma_{n}, \hat{\sigma}_{\textup{HC}0}, \hat{\sigma}_{\textup{HC}1}, \hat{\sigma}_{\textup{HC}2}, \hat{\sigma}_{\textup{HC}3}$, and the true standard deviation of $\htau$.}
	\end{subfigure}
		\begin{subfigure}{0.99\textwidth}  
                  \centering
		\includegraphics[width=0.48\textwidth]{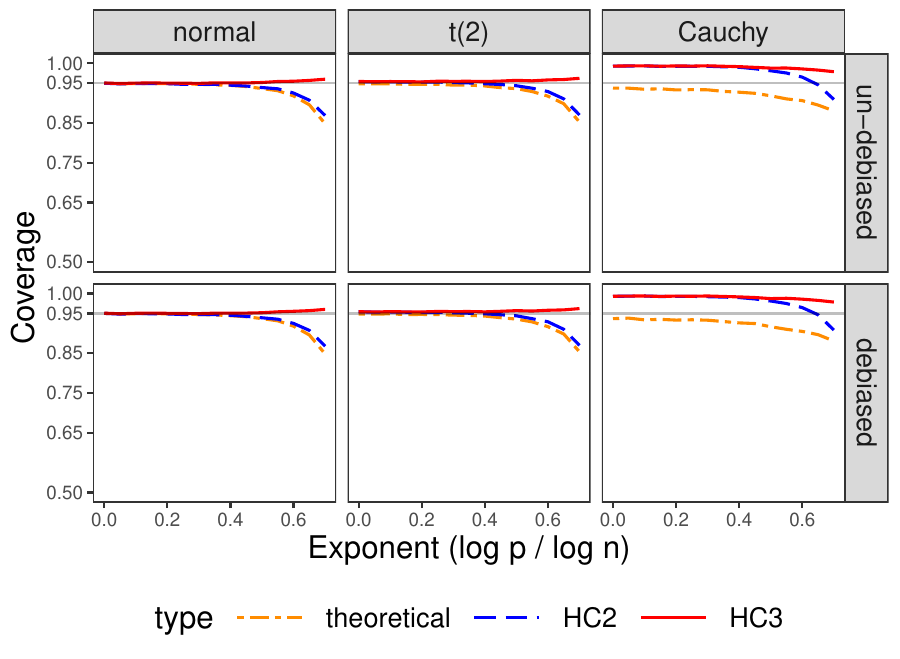}
		\includegraphics[width=0.48\textwidth]{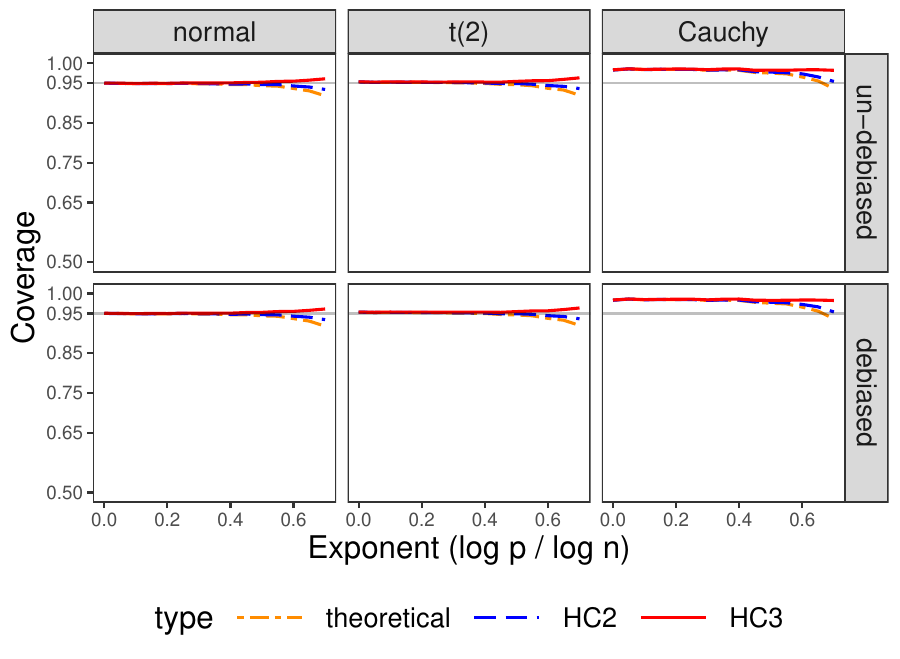}
		\caption{Empirical $95\%$ coverage rates of $t$-statistics derived from two estimators and four variance estimators (
                  ``theoretical'' for $\sigma_{n}^{2}$, ``HC2'' for $\hat{\sigma}_{\text{HC}2}^{2}$ and ``HC3''  for $\hat{\sigma}_{\text{HC}3}^{2}$)}
	\end{subfigure}
	\caption{Simulation without covariate trimming. $X$ is a realization of a random matrix with i.i.d. $N(0, 1)$ entries and $\eps(1) = \eps(0)$ is a realization of a random vector with i.i.d. entries: (Left) $\pi_{1} = 0.2$; (Right) $\pi_{1} = 0.5$. Each column corresponds to a distribution of $\eps(t)$. }\label{fig::simulation_normal_01_sharp}
\end{figure}

\begin{figure}[htp]
	\centering
	\begin{subfigure}{0.99\textwidth}  
          \centering
		\includegraphics[width=0.48\textwidth]{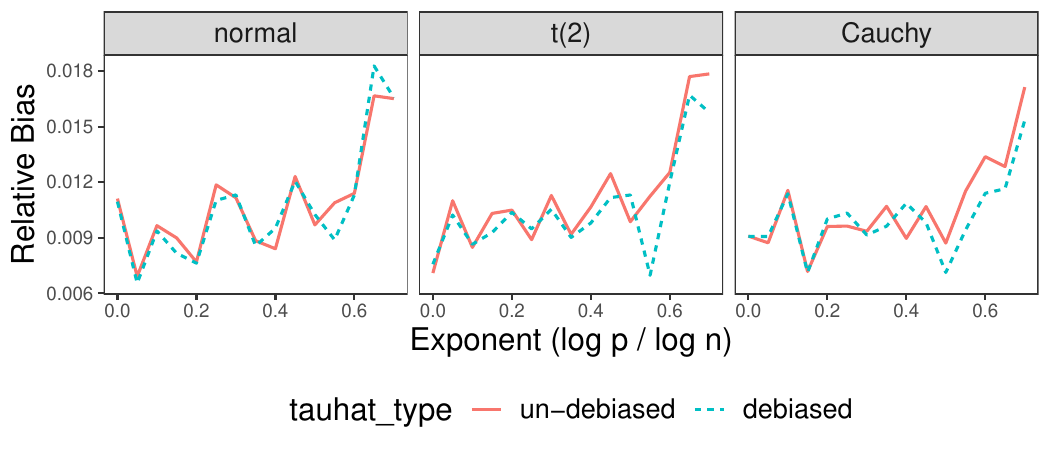}
		\includegraphics[width=0.48\textwidth]{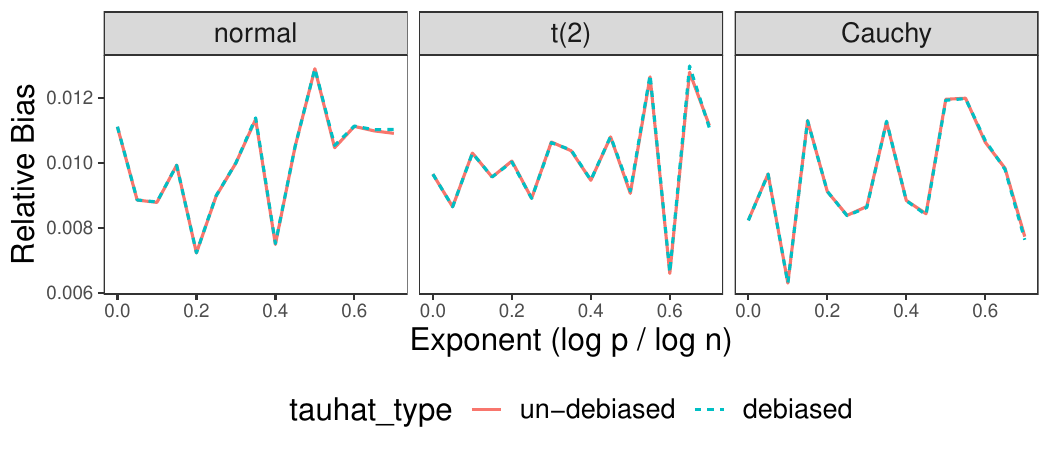}
		\caption{ Relative bias of $\hdtau$ and $\htau$.}
	\end{subfigure}
	\begin{subfigure}{0.99\textwidth}  
          \centering
		\includegraphics[width=0.48\textwidth]{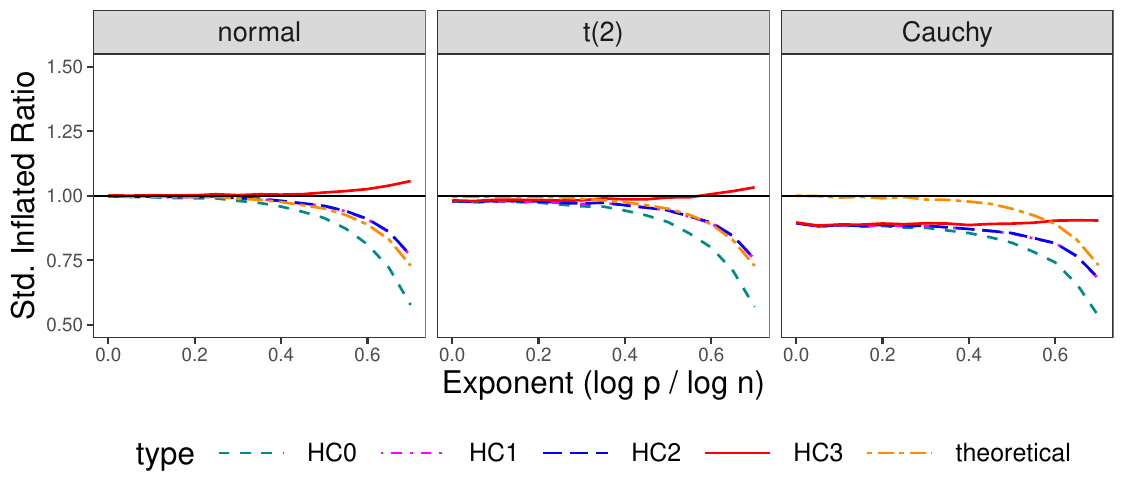}
		\includegraphics[width=0.48\textwidth]{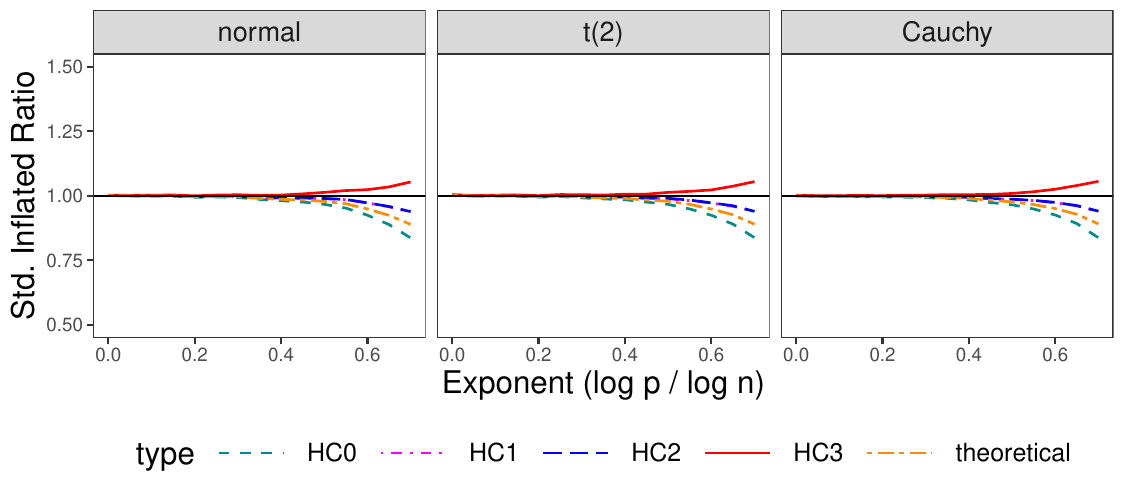}
		\caption{Ratio of standard deviation between five standard deviation estimates, $\sigma_{n}, \hat{\sigma}_{\textup{HC}0}, \hat{\sigma}_{\textup{HC}1}, \hat{\sigma}_{\textup{HC}2}, \hat{\sigma}_{\textup{HC}3}$, and the true standard deviation of $\htau$.}
	\end{subfigure}
		\begin{subfigure}{0.99\textwidth}  
                  \centering
		\includegraphics[width=0.48\textwidth]{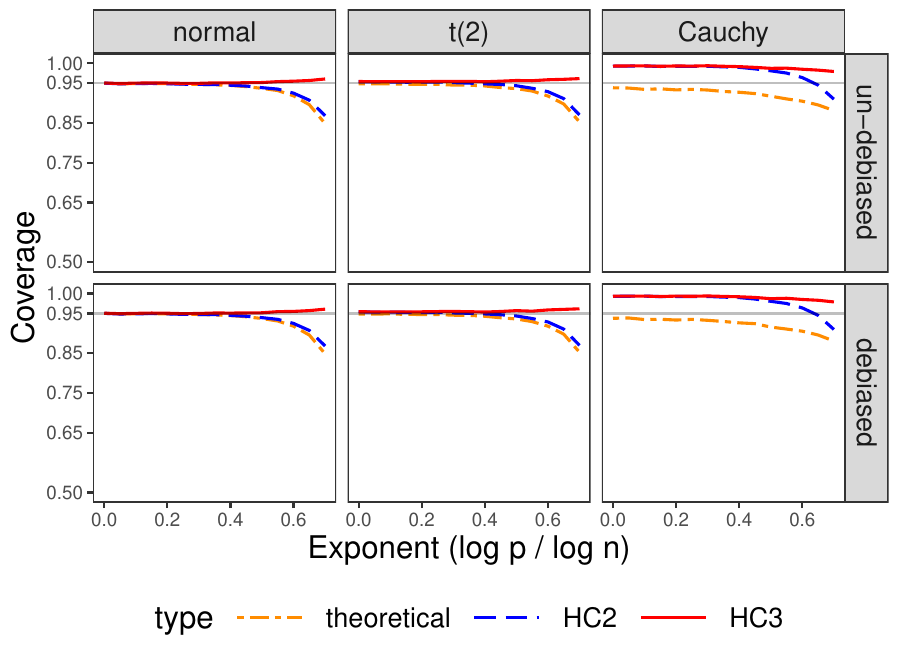}
		\includegraphics[width=0.48\textwidth]{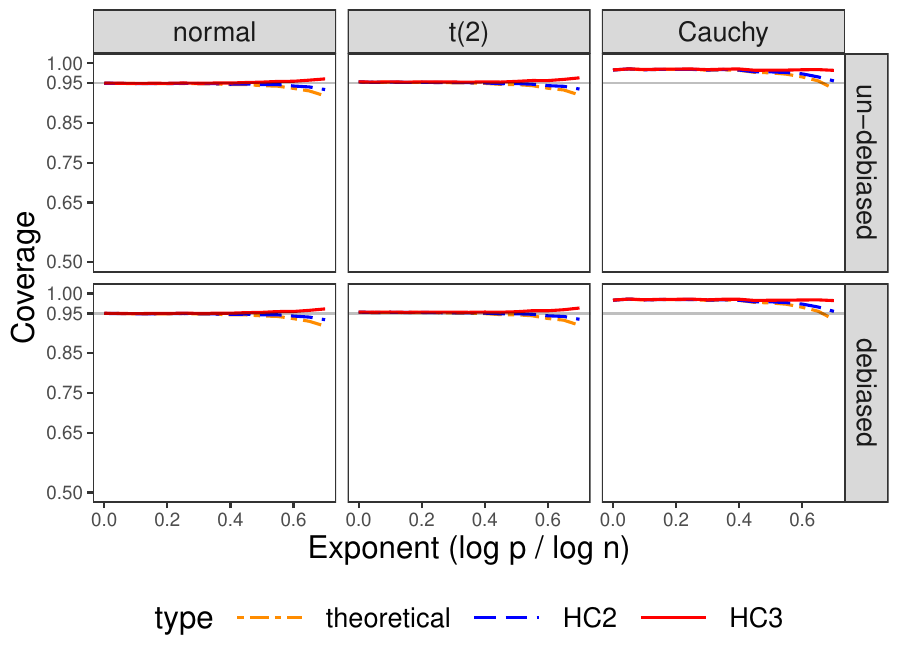}
		\caption{Empirical $95\%$ coverage rates of $t$-statistics derived from two estimators and four variance estimators (
                  ``theoretical'' for $\sigma_{n}^{2}$, ``HC2'' for $\hat{\sigma}_{\text{HC}2}^{2}$ and ``HC3''  for $\hat{\sigma}_{\text{HC}3}^{2}$)}
	\end{subfigure}
	\caption{Simulation with covariate trimming. $X$ is a realization of a random matrix with i.i.d. $N(0, 1)$ entries and $\eps(1) = \eps(0)$ is a realization of a random vector with i.i.d. entries: (Left) $\pi_{1} = 0.2$; (Right) $\pi_{1} = 0.5$. Each column corresponds to a distribution of $\eps(t)$. }
\end{figure}

\begin{figure}[htp]
	\centering
	\begin{subfigure}{0.99\textwidth}  
          \centering
		\includegraphics[width=0.48\textwidth]{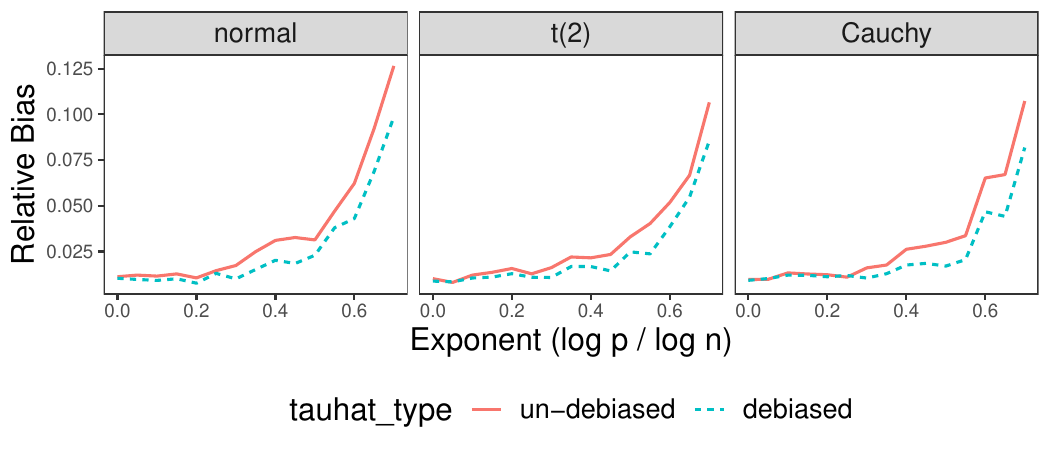}
		\includegraphics[width=0.48\textwidth]{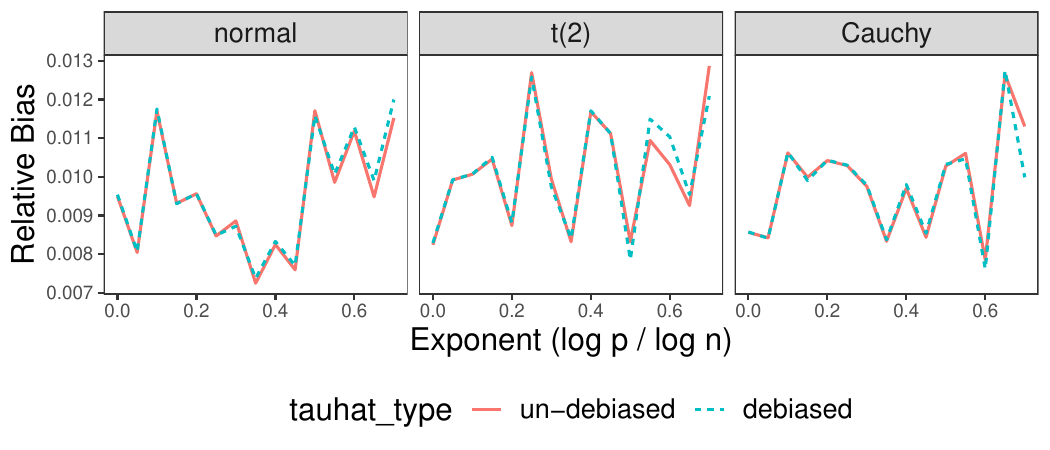}
		\caption{ Relative bias of $\hdtau$ and $\htau$.}
	\end{subfigure}
	\begin{subfigure}{0.99\textwidth}  
          \centering
		\includegraphics[width=0.48\textwidth]{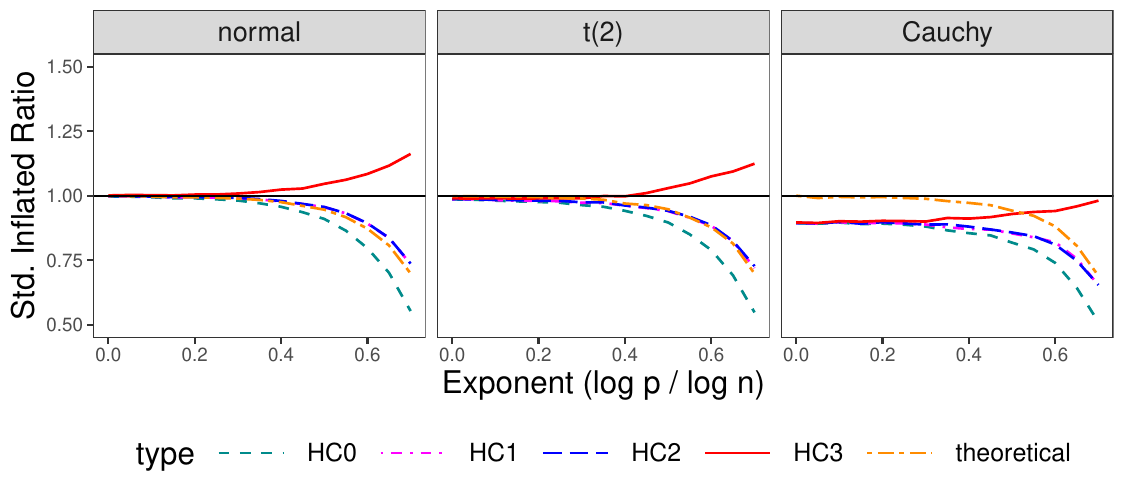}
		\includegraphics[width=0.48\textwidth]{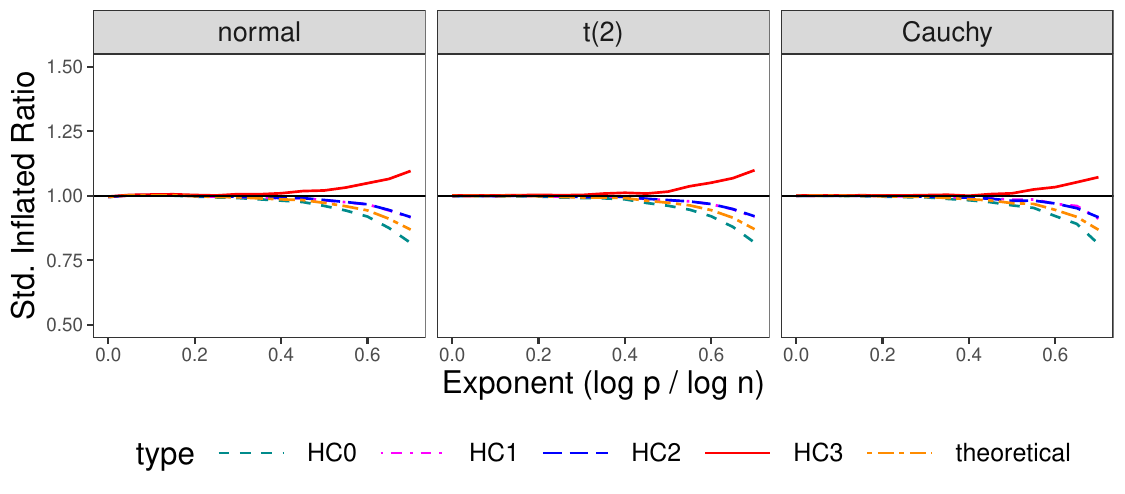}
		\caption{Ratio of standard deviation between five standard deviation estimates, $\sigma_{n}, \hat{\sigma}_{\textup{HC}0}, \hat{\sigma}_{\textup{HC}1}, \hat{\sigma}_{\textup{HC}2}, \hat{\sigma}_{\textup{HC}3}$, and the true standard deviation of $\htau$.}
	\end{subfigure}
		\begin{subfigure}{0.99\textwidth}  
                  \centering
		\includegraphics[width=0.48\textwidth]{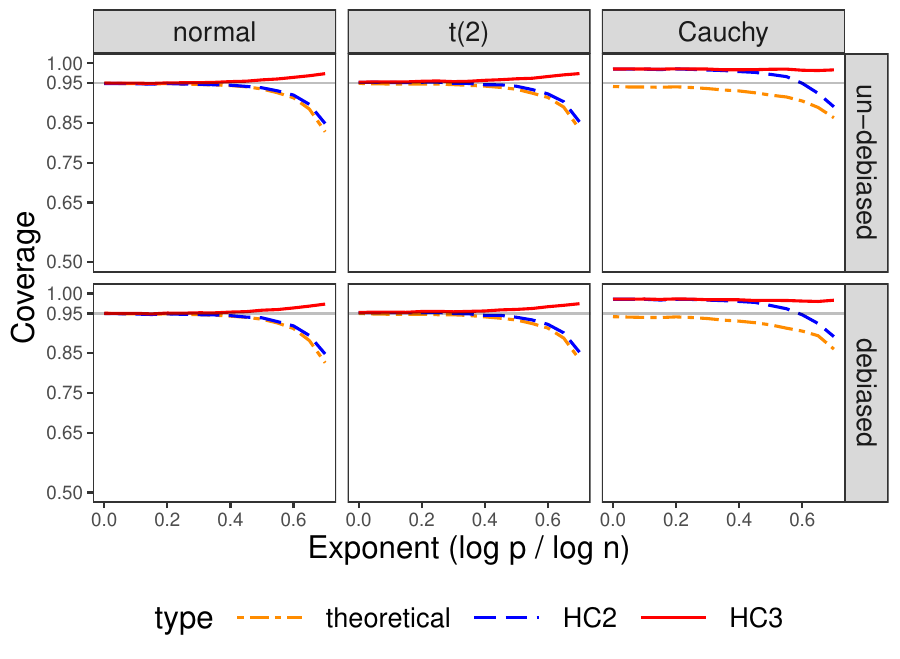}
		\includegraphics[width=0.48\textwidth]{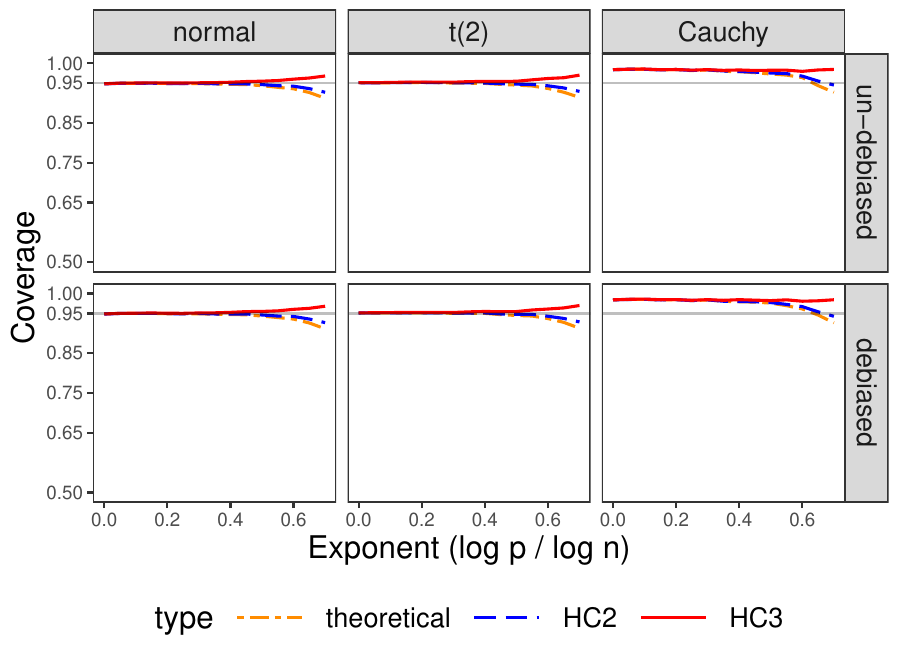}
		\caption{Empirical $95\%$ coverage rates of $t$-statistics derived from two estimators and four variance estimators (
                  ``theoretical'' for $\sigma_{n}^{2}$, ``HC2'' for $\hat{\sigma}_{\text{HC}2}^{2}$ and ``HC3''  for $\hat{\sigma}_{\text{HC}3}^{2}$)}
	\end{subfigure}
	\caption{Simulation without covariate trimming. $X$ is a realization of a random matrix with i.i.d. $t(2)$ entries and $\eps(1) = \eps(0)$ is a realization of a random vector with i.i.d. entries: (Left) $\pi_{1} = 0.2$; (Right) $\pi_{1} = 0.5$. Each column corresponds to a distribution of $\eps(t)$. }\label{fig::simulation_t2_01_sharp}
\end{figure}

\begin{figure}[htp]
	\centering
	\begin{subfigure}{0.99\textwidth}  
          \centering
		\includegraphics[width=0.48\textwidth]{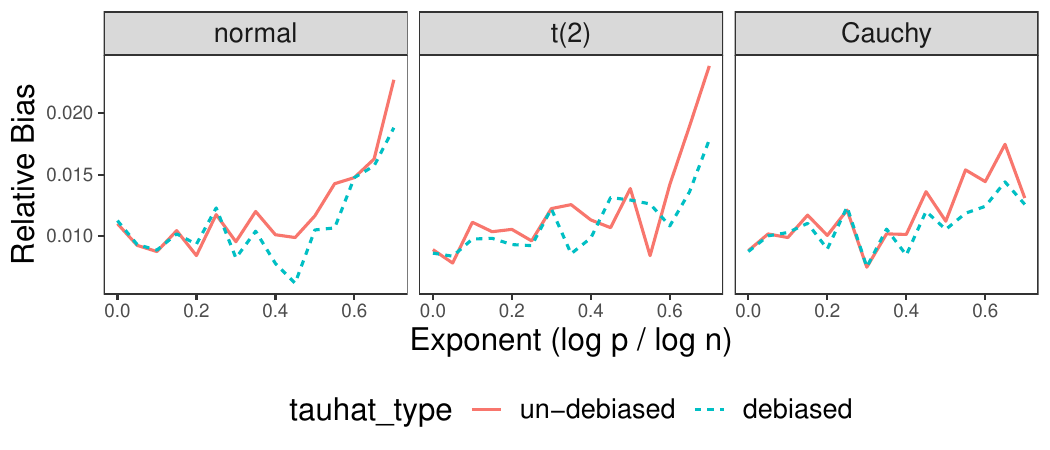}
		\includegraphics[width=0.48\textwidth]{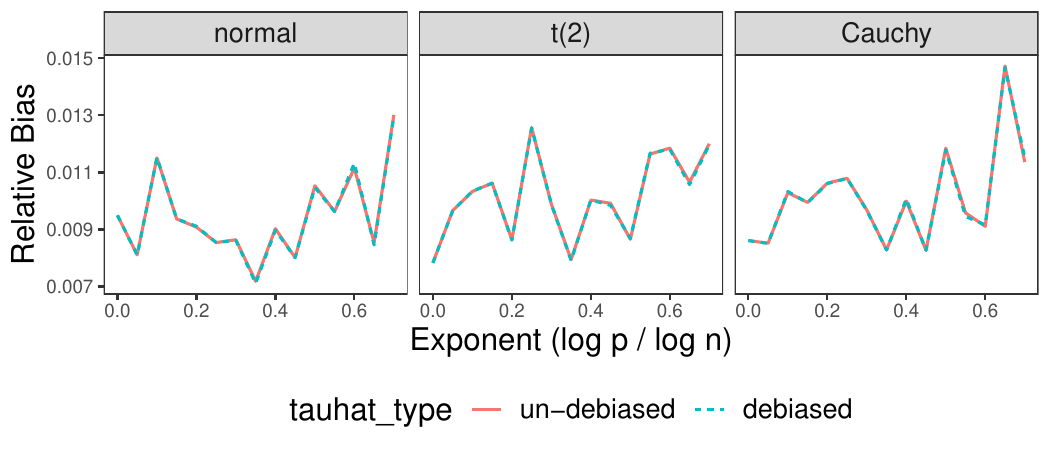}
		\caption{ Relative bias of $\hdtau$ and $\htau$.}
	\end{subfigure}
	\begin{subfigure}{0.99\textwidth}  
          \centering
		\includegraphics[width=0.48\textwidth]{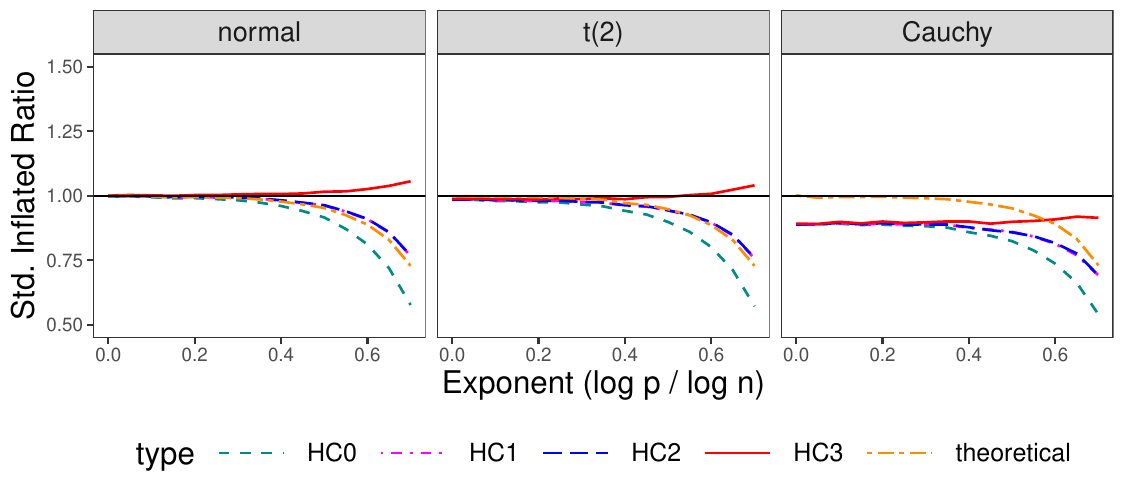}
		\includegraphics[width=0.48\textwidth]{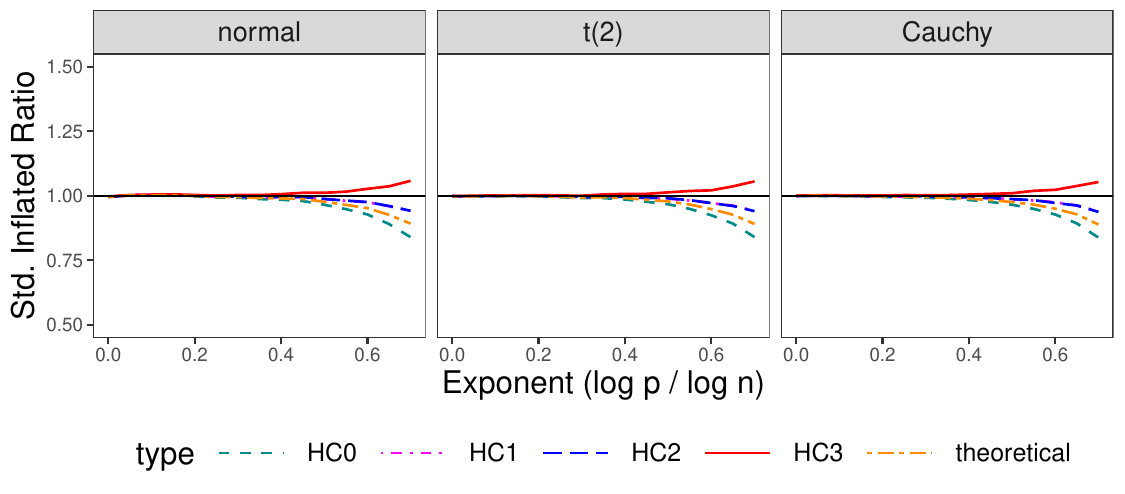}
		\caption{Ratio of standard deviation between five standard deviation estimates, $\sigma_{n}, \hat{\sigma}_{\textup{HC}0}, \hat{\sigma}_{\textup{HC}1}, \hat{\sigma}_{\textup{HC}2}, \hat{\sigma}_{\textup{HC}3}$, and the true standard deviation of $\htau$.}
	\end{subfigure}
		\begin{subfigure}{0.99\textwidth}  
                  \centering
		\includegraphics[width=0.48\textwidth]{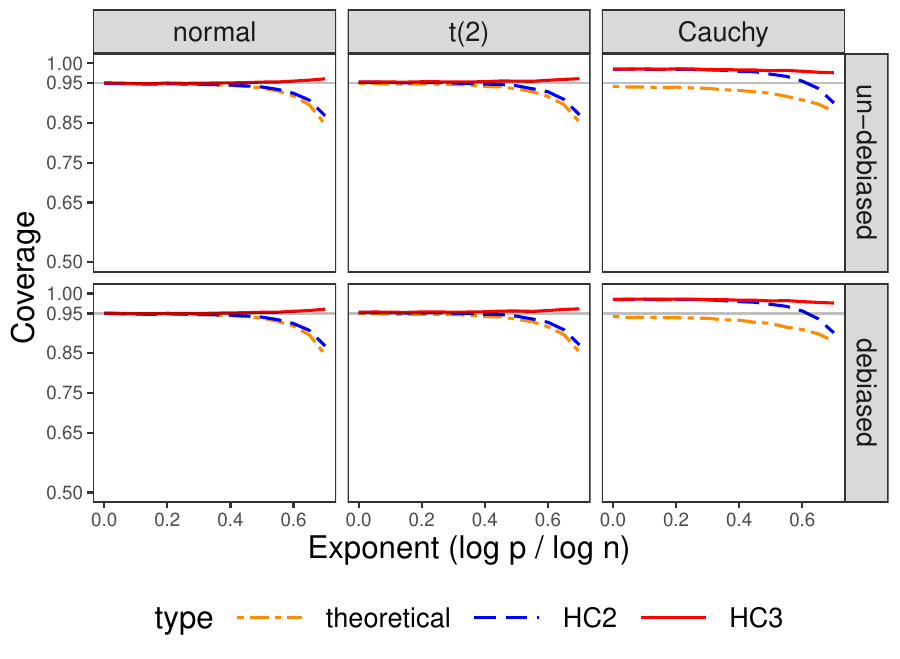}
		\includegraphics[width=0.48\textwidth]{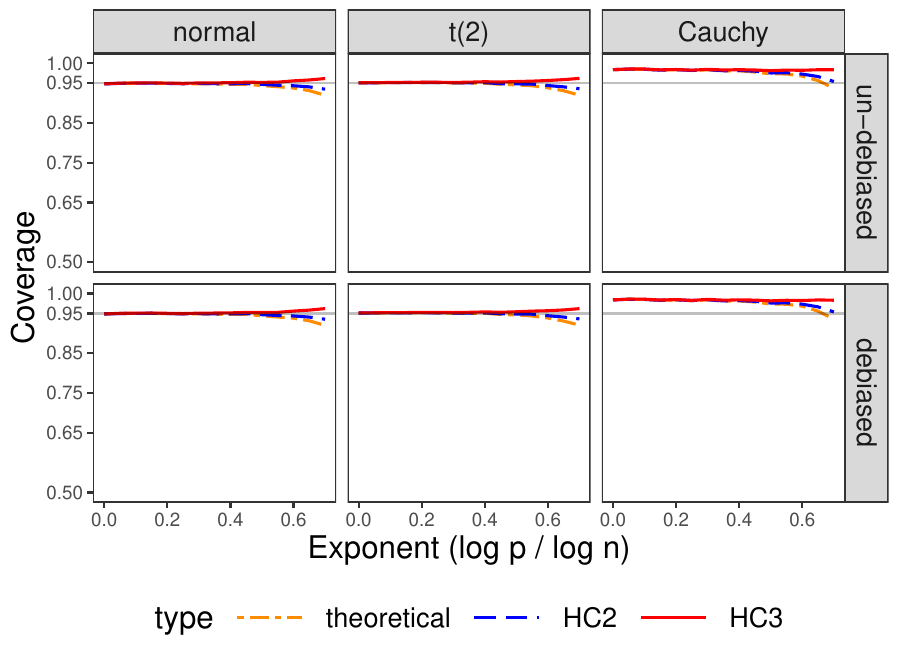}
		\caption{Empirical $95\%$ coverage rates of $t$-statistics derived from two estimators and four variance estimators (
                  ``theoretical'' for $\sigma_{n}^{2}$, ``HC2'' for $\hat{\sigma}_{\text{HC}2}^{2}$ and ``HC3''  for $\hat{\sigma}_{\text{HC}3}^{2}$)}
	\end{subfigure}
	\caption{Simulation with covariate trimming. $X$ is a realization of a random matrix with i.i.d. $t(2)$ entries and $\eps(1) = \eps(0)$ is a realization of a random vector with i.i.d. entries: (Left) $\pi_{1} = 0.2$; (Right) $\pi_{1} = 0.5$. Each column corresponds to a distribution of $\eps(t)$. }
\end{figure}

\begin{figure}[htp]
	\centering
	\begin{subfigure}{0.99\textwidth}  
          \centering
		\includegraphics[width=0.48\textwidth]{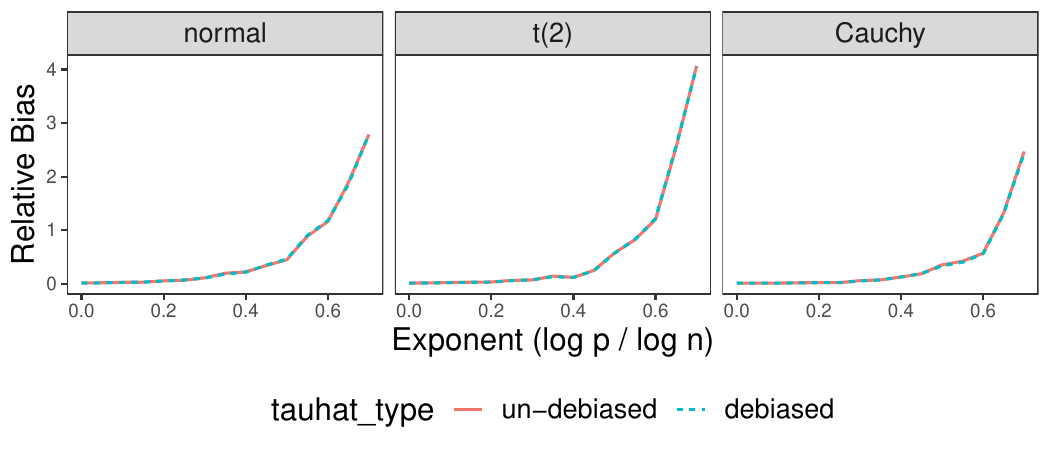}
		\includegraphics[width=0.48\textwidth]{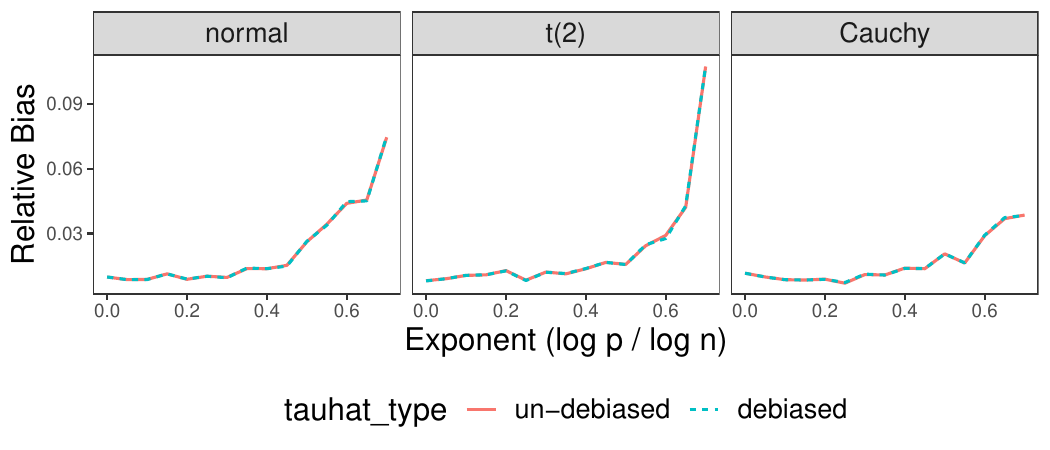}
		\caption{ Relative bias of $\hdtau$ and $\htau$.}
	\end{subfigure}
	\begin{subfigure}{0.99\textwidth}  
          \centering
		\includegraphics[width=0.48\textwidth]{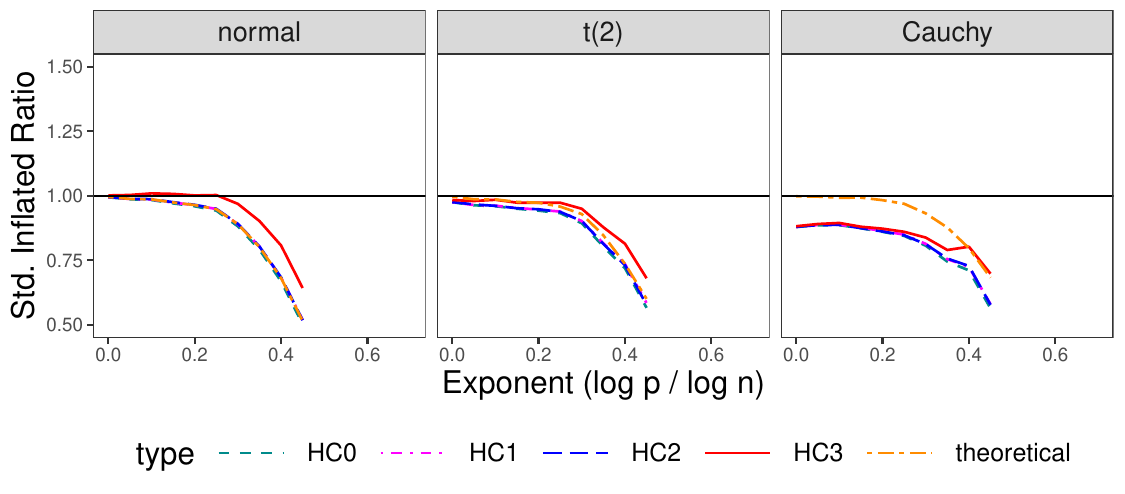}
		\includegraphics[width=0.48\textwidth]{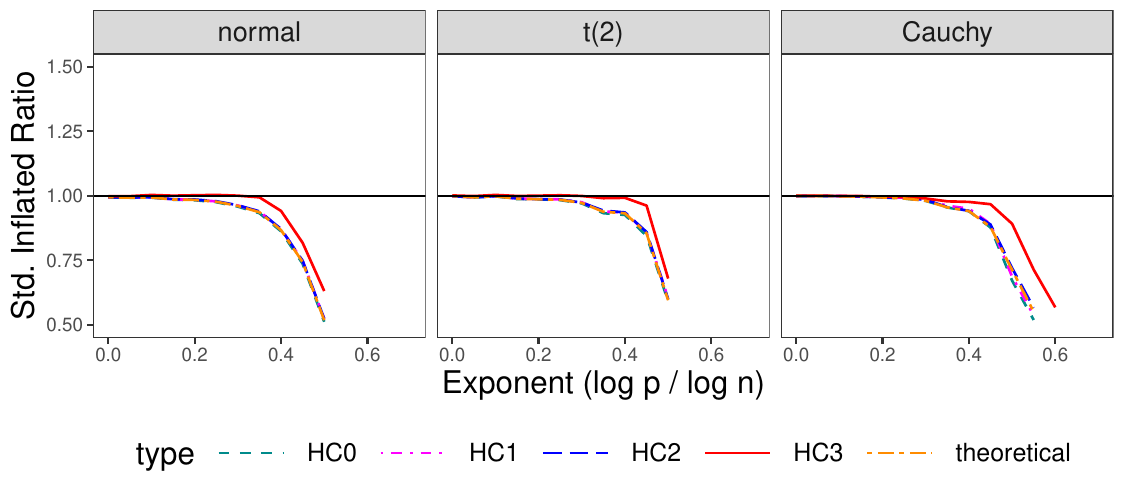}
		\caption{Ratio of standard deviation between five standard deviation estimates, $\sigma_{n}, \hat{\sigma}_{\textup{HC}0}, \hat{\sigma}_{\textup{HC}1}, \hat{\sigma}_{\textup{HC}2}, \hat{\sigma}_{\textup{HC}3}$, and the true standard deviation of $\htau$.}
	\end{subfigure}
		\begin{subfigure}{0.99\textwidth}  
                  \centering
		\includegraphics[width=0.48\textwidth]{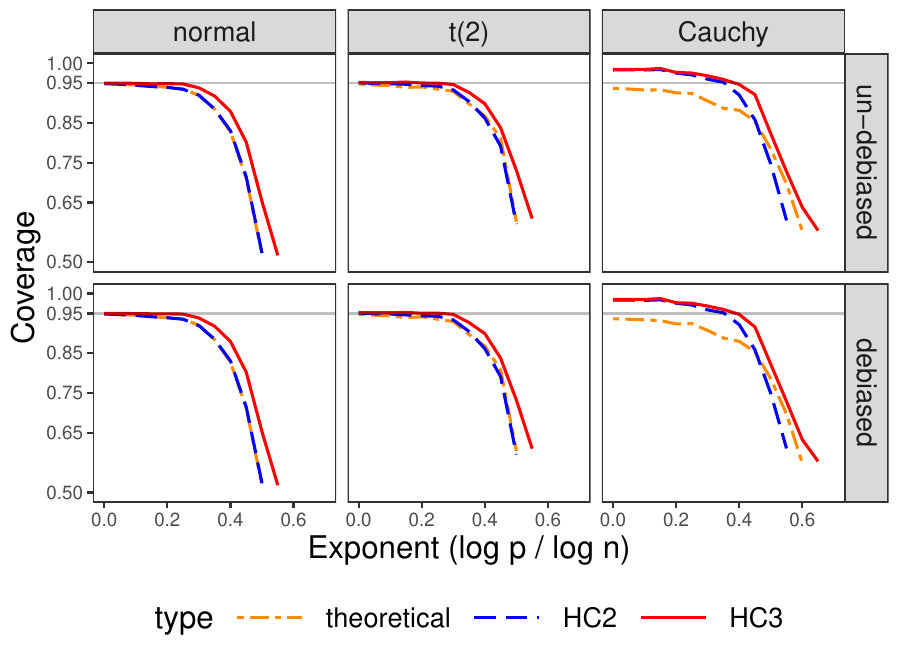}
		\includegraphics[width=0.48\textwidth]{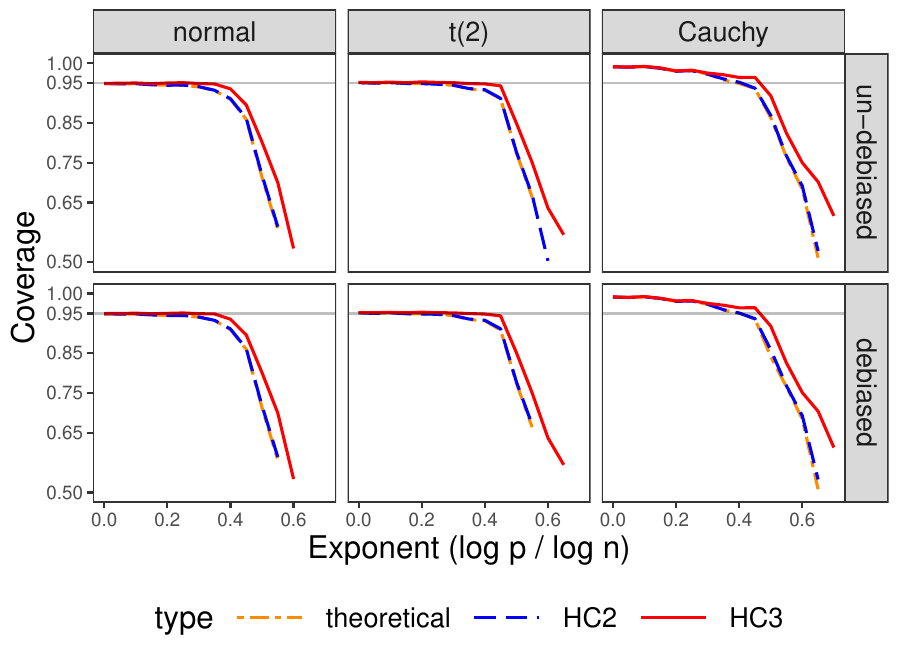}
		\caption{Empirical $95\%$ coverage rates of $t$-statistics derived from two estimators and four variance estimators (
                  ``theoretical'' for $\sigma_{n}^{2}$, ``HC2'' for $\hat{\sigma}_{\text{HC}2}^{2}$ and ``HC3''  for $\hat{\sigma}_{\text{HC}3}^{2}$)}
	\end{subfigure}
	\caption{Simulation without covariate trimming. $X$ is a realization of a random matrix with i.i.d. $t(1)$ entries and $\eps(1) = \eps(0)$ is a realization of a random vector with i.i.d. entries: (Left) $\pi_{1} = 0.2$; (Right) $\pi_{1} = 0.5$. Each column corresponds to a distribution of $\eps(t)$. }\label{fig::simulation_t1_01_sharp}
\end{figure}

\begin{figure}[htp]
	\centering
	\begin{subfigure}{0.99\textwidth}  
          \centering
		\includegraphics[width=0.48\textwidth]{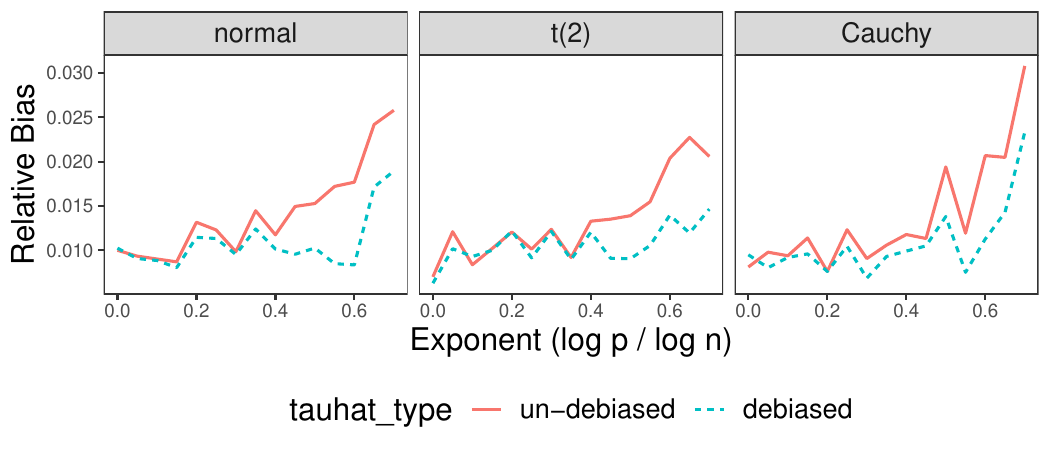}
		\includegraphics[width=0.48\textwidth]{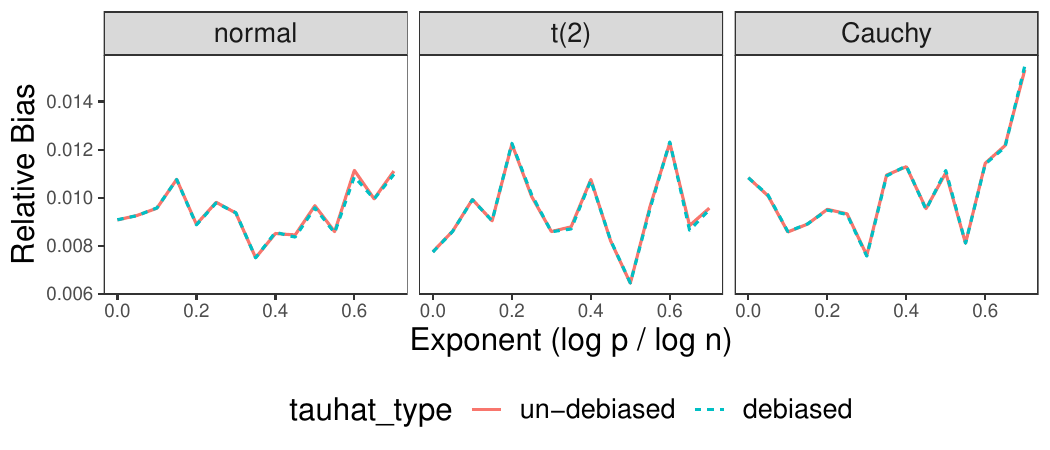}
		\caption{ Relative bias of $\hdtau$ and $\htau$.}
	\end{subfigure}
	\begin{subfigure}{0.99\textwidth}  
          \centering
		\includegraphics[width=0.48\textwidth]{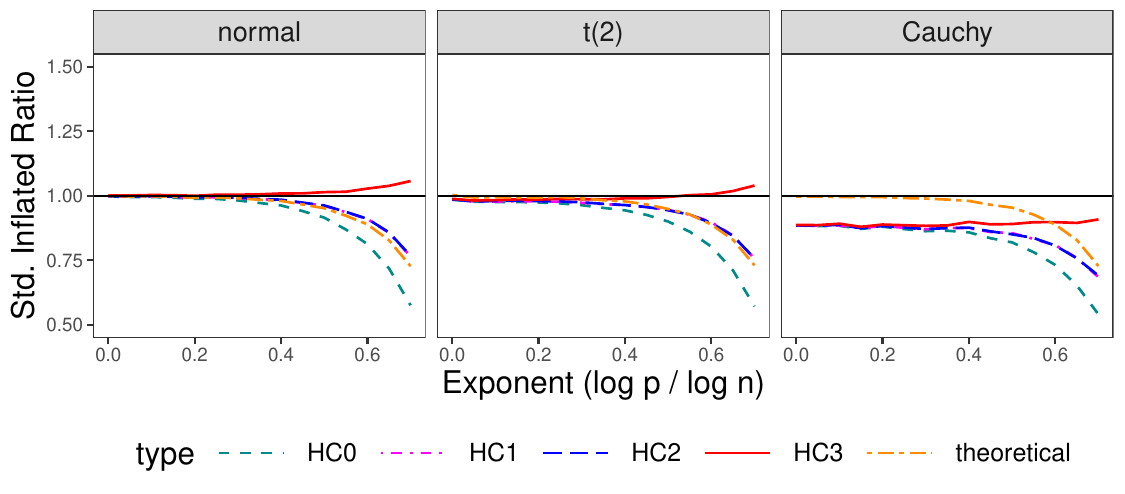}
		\includegraphics[width=0.48\textwidth]{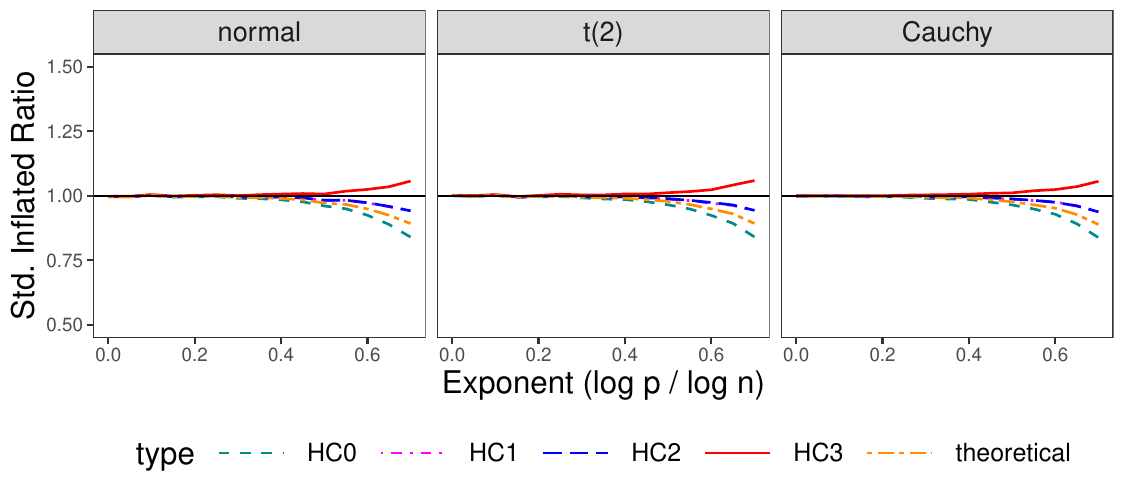}
		\caption{Ratio of standard deviation between five standard deviation estimates, $\sigma_{n}, \hat{\sigma}_{\textup{HC}0}, \hat{\sigma}_{\textup{HC}1}, \hat{\sigma}_{\textup{HC}2}, \hat{\sigma}_{\textup{HC}3}$, and the true standard deviation of $\htau$.}
	\end{subfigure}
		\begin{subfigure}{0.99\textwidth}  
                  \centering
		\includegraphics[width=0.48\textwidth]{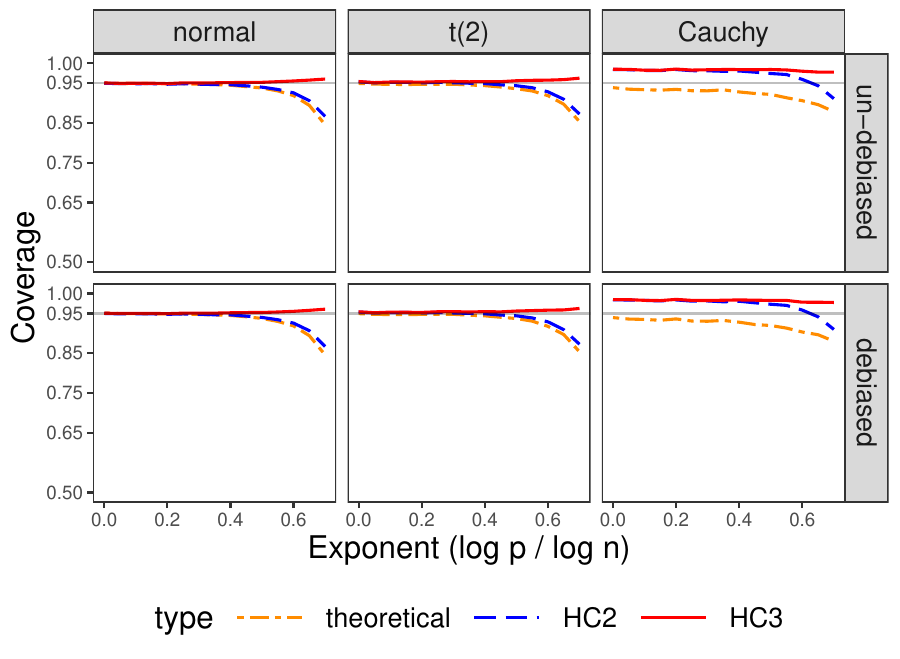}
		\includegraphics[width=0.48\textwidth]{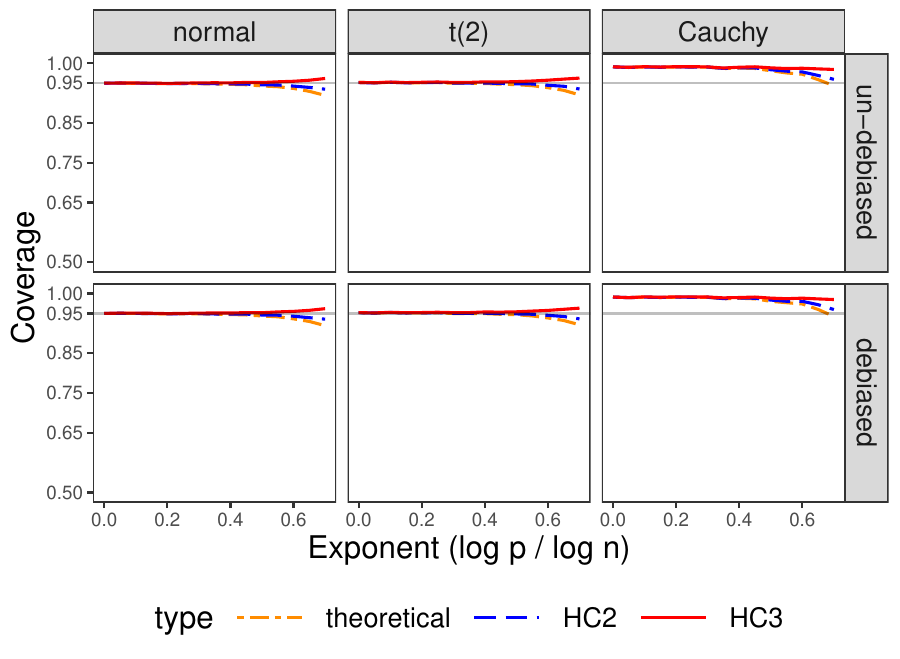}
		\caption{Empirical $95\%$ coverage rates of $t$-statistics derived from two estimators and four variance estimators (
                  ``theoretical'' for $\sigma_{n}^{2}$, ``HC2'' for $\hat{\sigma}_{\text{HC}2}^{2}$ and ``HC3''  for $\hat{\sigma}_{\text{HC}3}^{2}$)}
	\end{subfigure}
	\caption{Simulation with covariate trimming. $X$ is a realization of a random matrix with i.i.d. $t(1)$ entries and $\eps(1) = \eps(0)$ is a realization of a random vector with i.i.d. entries: (Left) $\pi_{1} = 0.2$; (Right) $\pi_{1} = 0.5$. Each column corresponds to a distribution of $\eps(t)$. } \label{fig::simulation_t1_02_sharp}
\end{figure}

\end{document}